\documentclass[a4paper]{article}

\usepackage[utf8]{inputenc}
\usepackage[T1]{fontenc}
\usepackage{csquotes}
\usepackage[english]{babel}

\usepackage[backend=bibtex,style=alphabetic,giveninits=true,isbn=false,doi=false,url=false,maxbibnames=4]{biblatex}
\usepackage{lmodern}

\usepackage{amssymb,amsmath,amsthm,mathrsfs}

\usepackage{todonotes}
\usepackage{mathtools,bbm,bm}
\usepackage[normalem]{ulem}
\usepackage{cancel}
\usepackage{verbatim}
\usepackage{tikz}
\usetikzlibrary{matrix}
\pgfdeclarelayer{background} 
\pgfsetlayers{background,main} 
\usepackage{enumerate}
\usepackage[shortlabels]{enumitem}

\usepackage{hyperref}

\usepackage{a4}
\usepackage{multirow}
\usepackage[footnotesize,bf,centerlast]{caption}
\usepackage[small,euler-digits,icomma,OT1,T1]{eulervm}

\usepackage{xcolor,colortbl}
\definecolor{Gray}{gray}{0.80}
\definecolor{LightGray}{gray}{0.90}
\definecolor{darkpastelgreen}{rgb}{0.01, 0.75, 0.24}

\usepackage{tikz}  
\usetikzlibrary{matrix} 
\pgfdeclarelayer{background} 
\pgfsetlayers{background,main} 

\newcommand{\cC}{\mathcal{C}}
\newcommand{\cD}{\mathcal{D}}
\newcommand{\cE}{\mathcal{E}}
\newcommand{\cF}{\mathcal{F}}

\newcommand{\cI}{\mathcal{I}}

\newcommand{\cP}{\mathcal{P}}

\newcommand{\cT}{\mathcal{T}}

\newcommand{\bE}{\mathbb{E}}

\newcommand{\bN}{\mathbb{N}}

\newcommand{\bR}{\mathbb{R}}

\newcommand{\PR}{\mathbb{P}}
\newcommand{\bONE}{\mathbbm{1}}

\newcommand{\dd}{ \mathrm{d}}

\DeclareMathOperator*{\argmin}{argmin}

\DeclareMathOperator{\grad}{grad}

\renewcommand{\epsilon}{\varepsilon}

\newcommand{\vn}[1]{\left| \! \left| #1\right| \!\right|}

\newcommand{\ip}[2]{\langle #1,#2\rangle}

\numberwithin{equation}{section}

\newtheorem{theorem}{Theorem}[section]
\newtheorem{lemma}[theorem]{Lemma}
\newtheorem{proposition}[theorem]{Proposition}
\newtheorem{corollary}[theorem]{Corollary}

\theoremstyle{definition}
\newtheorem{definition}[theorem]{Definition}
\newtheorem{remark}[theorem]{Remark}

\newtheorem{assumption}[theorem]{Assumption}

\newcommand{\R}{\mathbb{R}}
\newcommand{\De}{\mathrm{d}}

\newcommand{\geo}[2]{\bm{\zeta}^{{#1}\to{#2}}}

\title{Hamilton--Jacobi equations for controlled gradient flows: cylindrical test functions}

\author{Conforti G. \thanks{CMAP, Ecole Polytechnique, Route de Saclay, 91128, Palaiseau Cedex, France. \emph{E-mail address}: giovanni.conforti@polytechnique.edu. Research supported by the ANR project  ANR-20-CE40-0014.} , Kraaij  R. C. \thanks{Delft Institute of Applied Mathematics, Delft University of Technology, Mekelweg 4, 2628 CD Delft, The Netherlands. \emph{E-mail address}: r.c.kraaij@tudelft.nl} , Tonon D.\thanks{Dipartimento di Matematica "Tullio Levi-Civita", Universit\`a degli Studi di Padova, via Trieste 63, 35121 Padova, Italy. \emph{E-mail address}: daniela.tonon@unipd.it. Research supported by the project King Abdullah University of Science and Technology (KAUST)  ORA-CRG2021-4674
“Mean-Field Games: models, theory and computational aspects”; by the porject SID BIRD 2022 "Stochastic mean field control and the Schrödinger problem"; and by the project PRIN 2022 (prot. 2022W58BJ5)  "PDEs and optimal control methods in mean field games, population dynamics and multi-agent models"}} 
\date{\today}

\begin{document}

\maketitle
\abstract{This work is the second part of a program initiated in \cite{CoKrTo21} aiming at the development of an intrinsic geometric well-posedness theory for Hamilton-Jacobi equations related to controlled gradient flow problems in metric spaces. Our main contribution is that of showing that the comparison principle proven therein implies a comparison principle for viscosity solutions relative to smoother Hamiltonians, acting on test functions that are mere cylindrical functions of the underling squared metric distance and whose rigorous definition is achieved from the Evolutional Variational Inequality formulation of gradient flows (EVI). In particular, the new Hamiltonians no longer require to work with test functions containing Tataru's distance. This substantial simplification paves the way for the development of a comprehensive existence theory.} 
\tableofcontents

%

%
%
%
%


\section{Introduction}
The main goal of this paper is to take a second step in the study initiated in \cite{CoKrTo21} of infinite-dimensional Hamilton-Jacobi (HJ) equations characterizing the value function of controlled gradient flows problems. To fix ideas, consider a metric space $(E,d)$ where the $d$ is generated by a (formal) Riemannian metric $\ip{\cdot}{\cdot}$. Then the equations considered here may be seen as versions of the following prototype
\begin{equation} \label{eqn:formal_HJ_eq}
    f- \lambda H f = h, \quad Hf:=-\ip{\grad f}{ \grad \cE} + \frac{1}{2} \vn{\grad f}^2,
\end{equation}
where $\grad$ is the gradient associated with $\ip{\cdot}{\cdot}$. Equation \eqref{eqn:formal_HJ_eq} is expected to characterize the value function of the control problem 
\begin{equation}
    \sup\left\{ \int_0^{+\infty} e^{-\lambda^{-1} t}[\lambda^{-1} h(\rho^u(t))-\frac{1}{2}\vn{u(t)}^2\big]\De t: \dot{\rho}^{u} = -\grad \cE(\rho^{u}) + u, \,\rho^u(0)=\rho_0\right\},
\end{equation}
which can be interpreted as the problem of steering the gradient flow 
\begin{equation*}
    \dot{\rho} = -\grad \cE(\rho)
\end{equation*}
in such a way that an optimal balance is struck between the cost of controlling, modeled through the term $-\frac{1}{2}\vn{u(t)}^2$, and the reward obtained, modeled by the term $\lambda^{-1}h(\rho^u(t))$. A relevant setting where instances of\eqref{eqn:formal_HJ_eq} arise naturally is that of the Wasserstein space $(E,d)=(\cP_2(\R^d),W_2(\cdot,\cdot))$ equipped with an energy functional $\cE$ satisfying McCann's condition \cite{mccann1997convexity}: in this case, the underlying formal Riemannian metric is the so called Otto metric \cite{Ot01}.  We refer to \cite{BeDSGaJLLa02,FK06,FeMiZi21} for applications to statistical mechanics and large deviations, \cite{backhoff2020mean,monsaingeon2020dynamical} for applications to stochastic mass transport problems and general versions of the Schröodinger problem \cite{Le14}, as well as \cite{chen2021density} for automatic control.

One of the difficulties in analysing \eqref{eqn:formal_HJ_eq} is that the energy functional $\cE$ is typically not differentiable but only geodesically semiconvex. Even more fundamentally, the Riemannian metric is only formal, and cannot be rigorously defined. These obstructions are indeed all present in the Wasserstein space example.

\paragraph{Hamilton--Jacobi equations in infinite dimensional spaces}

The theory of viscosity solutions for Hamilton--Jacobi equations in the setting of Hilbert spaces or Banach spaces possessing the Radon-Nikodym property spaces was initiated by Crandall and Lions in \cite{CrLi84} and later developed in a series of influential papers. Beyond the above mentioned applications to large deviations and statistical mechanics, the rise of interest for McKean--Vlasov control problems \cite{carmona2018probabilistic} and Mean Field Games \cite{CaDeLaLi19} have driven the efforts to construct a theory of viscosity solutions for Hamilton--Jacobi equations on metric spaces that are not necessarily Hilbert, and in particular over the space of probability measures endowed with a transport--like distance. A first approach exploits the possibility of lifting the space of probability distributions to the space of square integrable random variables in order to take advantage of the Hilbertian structure of the latter: we refer to \cite{BaCoFuPh19,PhWe18,BeGrYa20} for some results recently obtained following this method. A second approach is more intrinsic and consists of working directly at the level of the space of probability measures and develop a notion of viscosity solution relying on a suitable metric or subdifferential structure, that often turns out to be that induced by optimal transport \cite{AmGiSa08}. We mention \cite{burzoni2020viscosity,AmFe14,gangbo2015existence,gangbo2015metric,GaTu19,WuZh20,CoGoKhPhRo21,cecchin2022weak,soner2022viscosity} as a sample of the recent contributions following this approach. We refer to the introduction of \cite{CoKrTo21} for a more thorough analysis of similarities and differences between the results obtained there and our approach, which draws inspiration from the work of Feng and coauthors \cite{FK06,FeKa09,FeMiZi21}. Mean Field Games theory led to the study of a class of measure--valued partial differential equations called master equations \cite{CaDeLaLi19}. Though related to infinite--dimensional Hamilton-Jacobi equations, master equations have a different nature than the one studied here. We refer to the introduction of \cite{CoKrTo21} for a brief explanation of the main differences, as well as for a summary and comments on recent contributions to this rapidly expanding research field.

\paragraph{Contribution of this work and perspectives}

This manuscript is the second chapter of a more general program initiated in \cite{CoKrTo21} whose aim is to develop a well posedness theory for \eqref{eqn:formal_HJ_eq}. In \cite{CoKrTo21} we established a comparison principle for viscosity solutions under mild assumptions, the most notable one being the existence of a gradient flow for the energy functional $\cE$ in Evolutional Variational Inequality (EVI) formulation. The result is stated in terms of rigorously defined upper and lower bounds $\widetilde{H}_{\dagger},\widetilde{H}_{\ddagger}$ for the formal Hamiltonian $H$, that are recalled at Definition \ref{definition:HdaggerHddagger} below and are constructed exploiting the evolutional variational inequality (EVI) characterization of gradient flows \cite{MuSa20}. The strength of this result lies in its generality; however, it is not the most practical in view of applications, as $\widetilde{H}_{\dagger},\widetilde{H}_{\ddagger}$ act on test functions that include Tataru's distance (see \eqref{eqn:tataru_def}), that is not a standard metric to manipulate. One would rather prefer to work with operators that act on cylindrical test functions of the form 
\begin{equation*}
\varphi\Big(\frac12d^2(\cdot,\rho_1),\ldots,\frac12d^2(\cdot,\rho_k)\Big),
\end{equation*}
where $\varphi$ is smooth and $\rho_1,\ldots,\rho_k$ are elements of the metric space $(E,d)$. The main result of this work is Theorem \ref{theorem:carry_over_solutions} where we show that viscosity solutions defined in terms of newly proposed upper and lower bounds $H_{\dagger},H_{\ddagger}$ (see Definition \ref{definition:H1}) acting on cylindrical test functions are indeed viscosity solutions for the operators $\widetilde{H}_{\dagger},\widetilde{H}_{\ddagger}$ introduced in \cite{CoKrTo21}. As a consequence, we bootstrap a comparison principle for $H_{\dagger},H_{\ddagger}$ from the one already available for $\widetilde{H}_{\dagger},\widetilde{H}_{\ddagger}$, see Corollary \ref{corollary:comparison_principle}. In addition to being more natural objects to consider, cylindrical test functions enjoy better regularity properties than Tataru's distance and are thus a preferable alternative for building an existence theory generalizing classical arguments from the finite dimensional setup. To build intuition,  we shall provide in the upcoming Section \ref{section:technical_intro} a heuristic derivation of the operators $H_{\dagger},H_{\ddagger}$, that relies once again on formal convexity properties of the energy functional and is eventually made rigorous through a systematic use of EVI.  The proof of the main results consists in a series of approximation steps in which we transfer viscosity sub(super)solutions from one operator onto another that acts on test functions that are increasingly closer to one that includes the Tataru distance. We refer to the discussion in Section \ref{section:strategy:approximatingTataru} and Figure \ref{figure:HJ_implications} at page \pageref{figure:HJ_implications} below for a more detailed overview of the proof architecture and a brief explanation of each approximation step.\\
As already stated above, this work can be inscribed in a larger effort to tackle \eqref{eqn:formal_HJ_eq} in its more general version. The level of generality of the results of \cite{CoKrTo21} and consequently of those of the present article is already quite large (see the examples of Sec. $5$ therein). For example, it covers the case of the Wasserstein space equipped with Boltzmann's entropy or a Rény entropy as energy functional. The addition of an interaction energy modeled through a pair potential is also covered as soon as the potential satisfies some mild conditions. What remains to be done to gain a comprehensive understanding of \eqref{eqn:formal_HJ_eq} is to build a solid existence theory for viscosity solutions and show that wellposedness of several other equations of interest can be established by verifying that the underlying space and energy satisfy the hypothesis required for our comparison principle to apply. In what concerns existence, the first step in this directions are taken in \cite{CoKrTaTo24}. About extending the range of applicability of our main results, there are several possibilities: in first place, one can mention the class of dynamic transport distances introduced in \cite{dolbeault2009new} for which gradient flows in EVI formulation for energy functionals relating to Macroscopic Fluctuation Theory \cite{BeDSGaJLLa02} have been constructed in \cite{carrillo2010nonlinear}. Another viable direction is that of considering the HJ equations arising in the study of GENERIC systems \cite{GrOt97,DPZ13} that are the first step beyond systems that are of gradient type.

\paragraph{Organization}

The paper is organized as follows. 

In Section \ref{section:technical_intro} we justify heuristically the definition of viscosity solutions we are going to work with and provide several insights on the key  concepts and mathematical objects we shall use in the rest of the paper. 

In Section \ref{section:general_framework}, we introduce the setting of gradient flows in metric spaces and introduce the context in which we will be working. Additionally, we introduce rigorously the two sets of Hamiltonians that we are working with, and state our main result, Theorem \ref{theorem:carry_over_solutions}.

In Section \ref{section:preparations_proof}, we give an outline of the key steps of the proof of Theorem \ref{theorem:carry_over_solutions}. To establish these steps, we make use of two technical lemmas that are proven in this section also.
                                                            
In Section \ref{proof:carry_over_solutions}, we use the aforementioned technical lemmas to establish the key steps of the proof.

Technical results and background material are gathered in the Appendix sections.

\paragraph{Acknowledgements}
The authors thank Luca Tamanini for helpful discussions. RK thanks Jin Feng for an introduction into using large deviation statements when relating Hamiltonians.


\section{Introduction to the technical aspects of the paper} \label{section:technical_intro}\label{sec: two_comp}

In this section, we introduce the main concepts of the paper on an intuitive level. As above, consider the Hamilton-Jacobi equation
\begin{equation} \label{eqn:formal_HJ_eq1}
    f - \lambda H f  = h 
\end{equation}
for a Hamiltonian $H$ that formally acts as
\begin{equation} \label{eqn:formal_H}
Hf(\pi) = \ip{\grad_\pi f(\pi)}{- \grad_\pi \cE(\pi)} + \frac{1}{2} \vn{\grad_\pi f(\pi)}^2 
\end{equation}
where $\cE: E \to (-\infty, +\infty]$ is some energy functional and gradients are taken w.r.t. a formal Riemannian structure on $E$.  The rigorous definition of such an Hamiltonian depends on the precise notion of gradient that we are going to use. This is the first difficulty to face, since in many of the cases we are going to consider, the energy functional $\cE$ lacks in differentiability (e.g. the Wasserstein space $(\cP_2(\R^d),W_2(\cdot,\cdot))$). Therefore, aiming at well-posedness of such equations, one has to find bounds for the Hamiltonian using notions of gradient flows that do not appeal to the $\mathrm{grad}_{\pi}\cE$ directly, see \cite{AmGiSa08}.

\subsection{Bounds via the evolutional variational inequality} \label{section:intro:bounds_via_EVI}

We start by introducing the strongest possible formulation of a gradient flow on a metric space, i.e. that of a solution to the evolutional variational inequality (EVI), see \cite{AmGiSa08,MuSa20}. We say that $\gamma(t)$ solves \eqref{item:intro_EVI} for $\kappa\in \R$, if
\begin{equation}\label{item:intro_EVI}\tag{$EVI_{\kappa}$}
		\frac{1}{2} {\frac{\dd^+}{\dd t}} \left(d^2(\gamma(t),\rho)\right) \leq \cE(\rho) - \cE(\gamma(t)) - \frac{\kappa}{2} d^2(\gamma(t),\rho),\quad \forall \rho \in \cD(\cE),t\in [0,+\infty).
\end{equation}
We will formally work with metric spaces satisfying the (formal) Riemannian property of the distance
\begin{equation}  \label{eqn:ass_formal_noise_input}
		\forall \pi,\rho\in E \quad \left|\partial\left( \frac{1}{2} d^2(\cdot,\rho)\right)\right|^2(\pi) = d^2(\pi,\rho).
\end{equation}
Note that the above equation holds in the case of a smooth Riemaniann manifold as well as on the Wasserstein space $(\cP_2(\R^d),W_2(\cdot,\cdot))$.

Let us now consider a test function $f^\dagger : E \to \R$  that is given in terms of the squared distance as $f^\dagger(\pi) = \frac{1}{2} a d^2(\pi,\rho)$ for some $\rho\in \cD(\cE)$ and $a>0$. Applying formally the expression for $H$ from \eqref{eqn:formal_H} and  \eqref{eqn:ass_formal_noise_input}
(as if $\pi\in \cD(\cE)$), we obtain that
\begin{equation*}
Hf^\dagger(\pi) =  \frac{1}{2} a\frac{\dd}{\dd t} \left( d^2 (\pi(t), \rho) \right)\Big|_{t=0} + \frac{1}{2} a^2 d^2(\pi,\rho).
\end{equation*}
Then, applying (formally) \eqref{item:intro_EVI}  and being $a > 0$, we get 
\begin{equation} \label{eqn:HJ_formal_upperbound}
Hf^\dagger(\pi) \leq  a\left[ \cE(\rho) - \cE(\pi) \right] -  {a\frac{\kappa}{2} d^2(\pi,\rho) } + \frac{1}{2} a^2 d^2(\pi,\rho).
\end{equation}
Similarly, we get a formal lower bound for a test function $f^\ddagger : E \to \R$ defined as $f^\ddagger(\mu) = - \frac{1}{2}a d^2(\gamma,\mu)$, $\gamma \in \cD(\cE)$
\begin{equation}\label{eqn:HJ_formal_lowerbound}
H f^\ddagger(\mu) \geq a\left[ \cE(\mu) - \cE(\gamma)\right] + a\frac{\kappa}{2} d^2(\gamma,\mu) + \frac{1}{2} a^2 d^2(\gamma,\mu).
\end{equation}

Making rigorous the steps above, we can start developing a correct formulation of the Hamilton--Jacobi equation. The key point here is that, on one hand, the upper and lower bounds introduced with (EVI) are sufficiently tight to allow for uniqueness proofs, whereas on the other, they are sufficiently relaxed to allow for existence theory.

In two companion papers, we treat both issues separately. In \cite{CoKrTo21}, we show that if we include the Tataru distance, to be introduced in Section \ref{section:intro_technical_Tataru} below, in our test functions, we indeed have sufficiently tight bounds for a comparison principle. In \cite{CoKrTaTo24}, we specify to the important relevant context $(E,d) = (\cP_2(\bR^d),W_2)$. Therein we establish existence for a class of Hamilton-Jacobi equations formulated in terms of smooth cylindrical test functions, whose definition we will use as a blue-print for a more general definition of the upper and lower bound in this paper. We introduce the heuristics regarding these test functions in Section \ref{section:intro_technical_smoothcylinders} below. 

The main result of this paper is to connect both collections of test functions. This way, we can connect the general uniqueness theory established for non-smooth test functions of \cite{CoKrTo21} to existence theory for smooth test functions.  We consider the heuristics of the connection in Section \ref{section:intro_technical_smoothcylinders_firststeptowardsTataru}.

\subsection{Test functions: {Comparison via Ekeland and Tataru's distance}} \label{section:intro_technical_Tataru}

To establish the comparison principle {and consequently uniqueness of solutions}, one needs to have test functions that capture more information than merely a quadratic distance. In \cite{CoKrTo21}, we build upon ideas from \cite{Ta92,Ta94,CrLi94,Fe06} and established the comparison principle for an upper and lower bound using the Ekeland variational principle where the {Ekeland} perturbation is performed using the Tataru test function. Thus, instead of giving upper and lower bounds in terms of test functions that only include the squared metric, as in the discussion above, we included the Tataru test function $d_T: E\times E \to [0,+\infty)$ defined as 
\begin{equation}\label{eqn:tataru_def}
    d_T(\pi,\rho) = \inf_{t \geq 0} \left\{ t + e^{\hat{\kappa} t} d(\pi,\rho(t))\right\}, \quad \forall \pi, \rho \in E,
\end{equation}
where $\hat{\kappa} = (0 \wedge \kappa)\leq 0$ and where $\rho(t)$ is the gradient flow for $\cE$ that starts in $\rho$. In \cite[Section 4]{CoKrTo21} we established that $d_T$ is $1$-Lipschitz along the gradient flow, and $1$-Lipschitz in terms of $d$. Formally,
\begin{equation}\label{eqn:tataru_formalbound}
      \forall \pi, \rho \in E \qquad \qquad  \Big| \frac{\dd}{\dd t} \left( d_T(\pi(t),\rho)\right)\big|_{t=0}\Big|\leq 1, \qquad \qquad \big|\partial d_T(\cdot,\rho) \big|(\pi) \leq 1.
\end{equation}

Extending upon the analysis of Section \ref{section:intro:bounds_via_EVI}, one can consider test functions of the type  $f^\dagger(\pi) = \frac{1}{2} a d^2(\pi,\rho) + b d_T(\pi,\mu) + c$ for some $\rho \in \cD(\cE)$ and $\mu\in E$, $a,b >0$ and $c \in \bR$. Via a formal computation one obtains the upper bound
\begin{equation} \label{eqn:H_tataru_sub}
Hf^\dagger(\pi) \leq  a\left[ \cE(\rho) - \cE(\pi) \right] -  {a\frac{\kappa}{2} d^2(\pi,\rho) } + b + \frac{1}{2} a^2 d^2(\pi,\rho) + ab d(\pi,\rho) + \frac{1}{2}b^2.
\end{equation}
To obtain a formal lower bound we consider $f^\ddagger : E \to \R$ defined as $f^\ddagger(\mu) = - \frac{1}{2}a d^2(\gamma,\mu) - b d_T(\gamma,\pi) + c$, $a,b > 0$, $c \in \bR$, $\gamma \in \cD(\cE)$ and $\pi \in E$:
\begin{equation}\label{eqn:H_tataru_super}
H f^\ddagger(\mu) \geq a\left[ \cE(\mu) - \cE(\gamma)\right] + a\frac{\kappa}{2} d^2(\gamma,\mu) - b + \frac{1}{2} a^2 d^2(\gamma,\mu) - ab d(\gamma,\mu) - \frac{1}{2}b^2.
\end{equation}

The  upper and lower bound given by \eqref{eqn:H_tataru_sub} and \eqref{eqn:H_tataru_super} respectively are precisely the one for which we establish the comparison principle (implying uniqueness of viscosity solutions) for the Hamilton-Jacobi equation \eqref{eqn:formal_HJ_eq1} in in \cite{CoKrTo21}.  Nevertheless, as the Tataru distance is non-smooth, establishing existence for the Hamilton-Jacobi equation in terms of Hamiltonians \eqref{eqn:H_tataru_sub} and \eqref{eqn:H_tataru_super} is a non-trivial matter.

\subsection{{Test functions: Existence via smooth cylindrical test functions}} \label{section:intro_technical_smoothcylinders}

To establish existence of viscosity sub- and supersolutions it pays off to work with smooth test functions. Clearly, the quadratic test functions and the bounds given in \eqref{eqn:HJ_formal_upperbound} and \eqref{eqn:HJ_formal_lowerbound} serve this purpose. However, this class of test functions is not always sufficiently large to easily allow for a uniqueness theory: we are not able to connect these simple test functions to the ones that include the Tataru distance. A typical method to enlarge the class of test functions, but to stay within the class of smooth test functions is to relax to the class of cylindrical test functions. The main goal of this paper is to show that this larger class of test functions is sufficient to connect to the Tataru distance.

For the upper bound, we will work with test functions of the type

\begin{equation} \label{eqn:def_intro_cylindertest}
    f^\dagger(\pi)  = \varphi\left( \frac{1}{2}d^2(\pi,\rho_0), \dots, \frac{1}{2}d^2(\pi,\rho_k)\right) 
\end{equation}
where $k \in \bN$, $\rho_0,\dots,\rho_k \in E$, $\varphi : [0,+\infty)^{k+1}\to\R $ is bounded and continuous and where  for any $i \in \{0,\dots,k\}$ we have $\partial_i \varphi > 0$. Arguing as in Section \ref{section:intro:bounds_via_EVI}, writing $\bm{\rho} = (\rho_0,\dots,\rho_k)$ and $d (\cdot,\bm\rho)= (d(\cdot,\rho_0),\dots, d(\cdot, \rho_k))$, we obtain the formal upper bound
\begin{equation} \label{eqn:H_cylindrical_sub}
    \begin{aligned}
     Hf^\dagger(\pi) & ={\frac{\dd}{\dd t} \left( f^\dagger (\pi(t)) \right)|_{t=0}+ \frac{1}{2} |\partial  f^\dagger|^2(\pi) }\\
    &\leq { \sum_{i=0}^k \partial_i \varphi\left( \frac{1}{2} d^2(\pi,\bm\rho)\right) \frac{\dd}{\dd t} \left( \frac 1 2 d^2(\pi(t), \rho_i) \right)\Big|_{t=0}} \\
    & \qquad { \,+ \frac{1}{2}\left(\sum_{i=0}^k \partial_i \varphi\left( \frac{1}{2}d^2(\pi,\bm\rho)\right) {\left| \partial\left(\frac{1}{2} d^2(\cdot,\rho_i)\right)\right|(\pi)}\right)^2} \\
    & \leq \sum_{i=0}^k \partial_i \varphi\left( \frac{1}{2}d^2(\pi,\bm\rho)\right) \left[\cE(\rho_i) - \cE(\pi) - \frac{\kappa}{2} d^2(\pi,\rho_i) \right] \\
    & \qquad \, + \frac{1}{2}\sum_{i=0}^k \left( \partial_i \varphi\left( \frac{1}{2}d^2(\pi,\bm\rho)\right)  d(\pi,\rho_i)\right)^2
\end{aligned}
\end{equation}
where  we used \eqref{item:ass_EVI} and \eqref{eqn:ass_formal_noise_input}. In a similar fashion, we can obtain a lower bound using functions of the type $f^\ddagger(\mu) = - \varphi\left(\frac{1}{2}d^2(\mu,\bm\gamma)\right)$,  in this case we need to work  slightly harder to find an appropriate lower bound for the square $\frac{1}{2} |\partial  f^\ddagger|^2(\mu)$. To give a particular, but relevant context, consider the Wasserstein space on $\bR^d$ and $d(\mu,\gamma) = W_2(\mu,\gamma)$. In this context, we find for $f^\ddagger$, but similar for $f^\dagger$, that 
\begin{equation*}
|\partial f^\ddagger|^2(\mu) = \sum_{i,j=0}^k \partial_i \varphi\left( \frac{1}{2}d^2(\mu,\bm\gamma)\right)\partial_j \varphi\left( \frac{1}{2}d^2(\mu,\bm\gamma) \right) \int  \ip{\bm{t_\mu^{\gamma_i}-id}}{\bm{t_\mu^{\gamma_j}-id}} \,\dd \mu 
\end{equation*}
where $t_\mu^{\gamma_i}-id$ is the transport map from $\mu$ to $\gamma_i$, see e.g. \cite{AmGiSa08}. This formula reflects the typical Hilbert space inner product structure underlying the square of the slope. Using Cauchy-Schwarz inequality, one would easily find the upper bound as in \eqref{eqn:H_cylindrical_sub}. To obtain  a useful lower bound, however, we find ourselves with a problem since we end up with negative terms for the off-diagonal terms: 
\begin{equation} \label{eqn:H_cylindrical_super}
\begin{aligned}
|\partial f^\ddagger|^2(\mu) & = \sum_{i,j=0}^k \partial_i \varphi\left( \frac{1}{2}d^2(\mu,\bm\gamma)\right)\partial_j \varphi\left( \frac{1}{2}d^2(\mu,\bm\gamma)\right)\int \ip{\bm{t_\mu^{\gamma_i}-id}}{\bm{t_\mu^{\gamma_j}-id}} \, \dd \mu \\
& \geq \sum_{i,j=0}^k \partial_i \varphi\left( \frac{1}{2}d^2(\mu,\bm\gamma)\right)^2  d^2(\mu,\gamma_i) \\
& \qquad - \sum_{i \neq j} \partial_i \varphi\left( \frac{1}{2}d^2(\mu,\bm\gamma)\right)\partial_j \varphi\left( \frac{1}{2}d^2(\mu,\bm\gamma)\right) d(\mu,\gamma_i)d(\mu,\gamma_j).
\end{aligned}
\end{equation}

The bounds \eqref{eqn:H_cylindrical_sub} and \eqref{eqn:H_cylindrical_super} for bounded $\varphi$ are a possible starting point for existence theory. We will, however, choose a slightly different starting point in next section by specifying our cylindrical test functions to one where a single quadratic component is singled out. This has multiple advantages:
\begin{itemize}
    \item The formulas bring our definitions closer to \eqref{eqn:H_tataru_sub} and \eqref{eqn:H_tataru_super}, 
    \item We can further lower bound \eqref{eqn:H_cylindrical_super} to more closely resemble the upper bound.
    \item We obtain an unbounded term that can be used to establish coercivity.
\end{itemize}
For those those that are interested in the bounds \eqref{eqn:H_cylindrical_sub} and \eqref{eqn:H_cylindrical_super} for bounded $\varphi$, can refer to our Appendix \ref{appendix:boundedcylinders}.

\subsection{{Relating the sets of test functions:  a stepping stone towards well-posedness theory}} \label{section:intro_technical_smoothcylinders_firststeptowardsTataru}

In \cite{CoKrTaTo24} we establish existence of solutions to the Hamilton-Jacobi equation where we specify our cylinders to
\begin{align}
    f^\dagger(\pi) & = \varphi\left( \frac{1}{2}d^2(\pi,\rho),\frac{1}{2}d^2(\pi,\mu_1)\dots, \frac{1}{2}d^2(\pi,\mu_k)\right) \notag \\
    & = \frac{1}{2}a d^2(\pi,\rho) +  \varphi_0\left( \frac{1}{2}d^2(\pi,\mu_1),\dots, \frac{1}{2}d^2(\pi,\mu_k)\right), \label{eqn:split_smooth_test_dagger}
\end{align}
and
\begin{equation}\label{eqn:split_smooth_test_ddagger}
    f^\ddagger(\mu) = - \frac{a}{2} d^2(\mu,\gamma) -  \varphi_0\left( \frac{1}{2}d^2(\mu,\pi_1),\dots, \frac{1}{2}d^2(\mu,\pi_k)\right).
\end{equation}
where $a > 0$. Splitting of the quadratic term of \eqref{eqn:H_cylindrical_sub} according to the decomposition in \eqref{eqn:split_smooth_test_dagger}, we find
\begin{equation}\label{eqn:H_cylindrical_specific_sub}
\begin{aligned}
    Hf^\dagger(\pi) & \leq a\left[\cE(\rho) - \cE(\pi) - \frac{\kappa}{2} d^2(\pi,\rho) \right] +
    \frac{1}{2} a^2 d(\pi,\rho)^2 \\
    & \qquad + \sum_{i=1}^k \partial_i \varphi_0\left( \frac{1}{2}d^2(\pi,\bm\mu)\right) \left[\cE(\mu_i) - \cE(\pi) - \frac{\kappa}{2} d^2(\pi,\mu_i) \right]  \\
    & \qquad + ad(\pi,\rho) \left(\sum_{i=1}^k  \partial_i \varphi_0\left( \frac{1}{2}d^2(\pi,\bm\mu)\right)  d(\pi,\mu_i)\right) \\
    & \qquad + \frac{1}{2} \sum_{i,j=1}^k \left( \partial_i \varphi_0\left( \frac{1}{2}d^2(\pi,\bm\mu)\right)  d(\pi,\mu_i)\right)^2.
\end{aligned}
\end{equation}
We can similarly split of the action of the Hamiltonian on the first term when working with the lower bound. In particular, when working with the terms in the Hamiltonian that arise from the squared gradient in \eqref{eqn:H_cylindrical_super}, we can use the elementary estimate
\begin{equation} \label{eqn:elementary_lowerbound}
    \sum_{i=0}^k a_i^2 -  \sum_{i\neq j} a_i a_j \geq a_0^2 - \sum_{i=1}^k a_i^2  -  \sum_{i\neq j} a_i a_j  = a_0^2 - \left(\sum_{i=1}^k a_i\right)^2 {- 2 a_0 \sum_{i=1}^k a_i},
\end{equation}
 which reads 
\begin{equation} \label{eqn:lowerbound_offdiagonal}
\begin{aligned}
|\partial f^\ddagger|^2(\mu) & \geq \frac{1}{2}a^2 d^2(\mu,\gamma)  - \left(\sum_{i=1}^k  \partial_i \varphi_0\left( \frac{1}{2}d^2(\mu,\bm\pi)\right) d(\mu,\pi_i)\right)^2 \\
& \qquad - 2 a d(\mu,\gamma)\left(\sum_{i=1}^k  \partial_i \varphi_0\left( \frac{1}{2}d^2(\mu,\bm\pi)\right) d(\mu,\pi_i)\right).
\end{aligned}
\end{equation}
The estimates then lead to a candidate lower bound
\begin{equation}\label{eqn:H_cylindrical_specific_super}
\begin{aligned}
    Hf^\ddagger(\mu) & \geq a \left[\cE(\mu) -\cE(\gamma) + \frac{\kappa}{2}d^2(\mu,\gamma) \right] + \frac{1}{2} a^2 d^2(\gamma,\mu) \\
    &  + \sum_{i=1}^k \partial_i \varphi\left(\frac{1}{2}d^2(\mu,\bm\pi)\right)\left[\cE(\mu) -\cE(\gamma_i) + \frac{\kappa}{2}d^2(\mu,\gamma_i) \right] \\
    &  \qquad - \frac{1}{2}  \left(\sum_{i=1}^k  \partial_i \varphi_0\left( \frac{1}{2}d^2(\mu,\bm\pi)\right) d(\mu,\pi_i)\right)^2 \\
    &  \qquad - a d(\mu,\gamma)\left(\sum_{i=1}^k  \partial_i \varphi_0\left( \frac{1}{2}d^2(\mu,\bm\pi)\right) d(\mu,\pi_i)\right).
\end{aligned}
\end{equation}

The bounds \eqref{eqn:H_cylindrical_specific_sub} and \eqref{eqn:H_cylindrical_specific_super} bounds are the basis for our paper, and the starting point for the existence theory in  \cite{CoKrTaTo24}. As uniqueness was established in \cite{CoKrTo21} in terms of the bounds \eqref{eqn:H_tataru_sub} and \eqref{eqn:H_tataru_super}, we need to relate both sets of test functions to obtain a satisfactory well-posedness theory. Our main theorem, Theorem \ref{theorem:carry_over_solutions}, shows that indeed solutions are related, i.e. any subsolution for the Hamilton-Jacobi equation formulated in terms of the upper bound of \eqref{eqn:H_cylindrical_specific_sub} is also a subsolution for the Hamilton-Jacobi equation formulated in terms of the upper bound in \eqref{eqn:H_tataru_sub}. A similar statement holds for supersolutions.

\smallskip

To connect both sets of test functions, again consider \eqref{eqn:split_smooth_test_dagger}:
\begin{equation*}
    f^\dagger(\pi)  = \frac{1}{2}a d^2(\pi,\rho) +  \varphi_0\left( \frac{1}{2}d^2(\pi,\mu_1),\dots, \frac{1}{2}d^2(\pi,\mu_k)\right), 
\end{equation*}
and compare this to 
\begin{equation*}
    \tilde{f}^\dagger(\pi) = \frac{1}{2} a d^2(\pi,\rho) + b d_T(\pi,\mu) + c.
\end{equation*}
We will connect the two by choosing $\mu_1,\dots,\mu_k$ equal to $(\mu(t_1),\dots,\mu(t_k))$ for some well chosen times $t_1,\dots,t_k$ where $t \mapsto \mu(t)$ is the gradient flow for $\cE$ started from $\mu(0) = \mu$ and where $\varphi_0$ is chosen such that 
\begin{equation} \label{eqn:approx_Tataru}
    \varphi_0\left( \frac{1}{2}d^2(\pi,\mu_1),\dots, \frac{1}{2}d^2(\pi,\mu_k)\right) \approx b d_T(\pi,\mu) + c.
\end{equation}
A large part of the analysis of this paper and the proof of the main Theorem \ref{theorem:carry_over_solutions} is spent on making this approximation precise, as well as verifying that the action of the gradient flow on this approximation behaves as it should. Here we follow ideas introduced by \cite{Fe06} in the context of large deviations on Hilbert spaces, implementing them  in this more involved context. 

We sketch here the key issues in facing the relation between the two sets of test functions. Taking the approximation \eqref{eqn:approx_Tataru} at face value, we find a striking resemblance between the pair \eqref{eqn:H_cylindrical_specific_sub} and \eqref{eqn:H_cylindrical_specific_super} on one hand and \eqref{eqn:H_tataru_sub} and \eqref{eqn:H_tataru_super} on the other.

Using the approximate identity \eqref{eqn:approx_Tataru}, we see that the term in the second line of \eqref{eqn:H_cylindrical_specific_sub} approximately equals
\begin{equation*}
    \Big| \frac{\dd}{\dd t} \left( b d_T(\pi(t),\mu)\right)\big|_{t=0}\Big|\leq b
\end{equation*}
where the bound follows by the first estimate in \eqref{eqn:tataru_formalbound}. The terms on the final two lines of \eqref{eqn:H_cylindrical_specific_sub} can be treated with the second estimate in \eqref{eqn:tataru_formalbound} and therefore give
\begin{multline*}
    ad(\pi,\rho) \left(\sum_{i=1}^k  \partial_i \varphi_0\left( \frac{1}{2}d^2(\pi,\bm\mu)\right)  d(\pi,\mu_i)\right) + \frac{1}{2} \sum_{i,j=1}^k \left( \partial_i \varphi_0\left( \frac{1}{2}d^2(\pi,\bm\mu)\right)  d(\pi,\mu_i)\right)^2 \\
    \leq ab d(\pi,\rho) + \frac{1}{2} b^2, 
\end{multline*}
thus showing that \eqref{eqn:H_tataru_sub} is an appropriate upper bound for \eqref{eqn:H_cylindrical_specific_sub}. Similar bounds on the basis of \eqref{eqn:tataru_formalbound} connect \eqref{eqn:H_tataru_super} to \eqref{eqn:H_cylindrical_specific_super}.

Thus, the proof of our main Theorem \ref{theorem:carry_over_solutions} reduces to making rigorous three main steps that are key in relating the two upper bound or lower bounds respectively:
\begin{enumerate}[(1)]
    \item Finding approximate functions $\varphi_0$ and $\mu_1,\dots,\mu_k$ such that \eqref{eqn:approx_Tataru} holds:
    \begin{equation*} 
    \varphi_0\left( \frac{1}{2}d^2(\pi,\mu_1),\dots, \frac{1}{2}d^2(\pi,\mu_k)\right) \approx b d_T(\pi,\mu) + c,
    \end{equation*}
    \item checking that, working with this $\varphi_0$, it approximately holds that the action of the gradient flow on this test function is bounded by $b$, as in the first bound of \eqref{eqn:tataru_formalbound},
    \item checking that, in the same situation, it approximately holds that the squared slope of the part involving $\varphi_0$ is bounded by $b^2$ as predicted by the second bound of \eqref{eqn:tataru_formalbound}.
\end{enumerate}

\textbf{Frequently used notation}

We write $C(E), LSC(E)$, and $USC(E)$ for the spaces of continuous, lower semi-continuous and upper semi-continuous functions from $E$ into $\bR$. We denote by $C_u(E), C_l(E), LSC_l(E)$ and $USC_u(E)$ the subsets of functions that admit a lower or upper bound. Finally $C_b(E) = C_u(E) \cap C_l(E)$. 
We write $\bN_\infty$ equals $\bN \cup \{\infty\}$ equipped with the topology where any unbounded sequence converges to $\infty$.

Finally, for a constant $\kappa \in \bR$, we write $\hat{\kappa} = (0 \wedge \kappa)\leq 0$.

\section{EVI-gradient flows and main results} \label{section:general_framework}

\subsection{Set-up}

The setting of this paper will be a complete metric space $(E,d)$ where we define $\cE : E\to (-\infty,+\infty]$ an extended energy (entropy) functional. Being the notion of gradient too strong for the considered energy functional we will use the definition of local slope, as defined in the first chapter of \cite{AmGiSa08}.
\begin{definition} {Let $\phi: E\to (-\infty,+\infty]$ be an extended functional with proper effective
domain, i.e. 
$\cD(\phi):=\{ \pi\in E: \phi(\pi)<+\infty \}\neq \emptyset$.}
Then the local slope of $\phi$ at {$\rho\in \cD(\phi)$} is defined as 
\begin{equation*}
    |\partial \phi|(\rho):= \begin{cases} \limsup_{\pi \rightarrow \rho} \frac{(\phi(\rho)-\phi(\pi))^+}{d(\rho,\pi)}, &\quad\mbox{if $\phi(\rho)<+\infty.$ }  \\
    +\infty, &\quad \mbox{otherwise.}
    \end{cases}
\end{equation*}

\end{definition}

Moreover, our metric space will be required to be  a geodesic space in the sense of the following definition.

\begin{definition} $(E,d)$ is a geodesic space, if for any $\rho,\pi\in E$ there exists a curve  $(\geo{\rho}{\pi}(t))_{t\in[0,1]}$ such that $\geo{\rho}{\pi}(0)=\rho,\geo{\rho}{\pi}(1)=\pi$ and for all $s,t\in[0,1]$
		\begin{equation}\label{eq: geodesicproperty}
		    d(\geo{\rho}{\pi}(s),\geo{\rho}{\pi}(t))=|t-s|d(\rho,\pi).
		\end{equation}
		Such a curve will be called \emph{geodesic}.
\end{definition}

\begin{assumption}[Metric and energy] \label{assumption:distance_and_energy}
The complete metric space $(E,d)$ and the energy functional $\cE$ satisfy the following assumptions: 
\begin{enumerate}[(a)]
		\item   $(E,d)$ is a geodesic space. 
		\item We assume that the energy functional $\cE : E \to { (-\infty, +\infty]}$ is an extended functional such that: 
        \begin{itemize}
            \item  It has a proper effective domain, i.e. $\cD(\mathcal{E}):=\{ \pi\in E:\mathcal{E}(\pi)<+\infty \}\neq \emptyset
             $.   
            \item It is lower semi-continuous. 
        \end{itemize}
\end{enumerate}
\end{assumption}

\begin{assumption}[Weak topology]\label{assumption:weak_topology}
 We assume the existence of a topology on $E$ that is weaker than the topology generated by $d$. We will call this topology the weak topology. We assume that
\begin{itemize}
    \item 
    The metric $d(\cdot,\cdot)$ is weakly lower semi-continuous. The energy functional $\cE$ is weakly lower semi-continuous on metric balls. 
    \item For all $\rho \in E$ and $c,d \in \bR$, the set 
    \begin{equation} \label{def: K^rho}
    K^{\rho}_{c,d} := \{ \pi\in E\ : d(\rho,\pi) \leq c, \cE(\pi)\leq d \}   
    \end{equation}
    is weakly compact.
\end{itemize}

\end{assumption}

In most of the examples of interest, metric balls are not compact with respect to the topology generated by $d$ and this in Assumption \ref{assumption:weak_topology} we do not use the standard topology but a weaker one. Let us note that for  the fundamental example
$(E,d)=(\cP_2(\R^d),W_2(\cdot,\cdot))$, Assumption \ref{assumption:weak_topology} is verified by the topology generated by convergence in the $W_p(\cdot,\cdot)$ metric for $p <2$.

\begin{remark}
    Assumption \ref{assumption:weak_topology} combined with the lower bound  \eqref{eq: lower bound on energy} on the energy $\cE$ defined in Assumption \ref{assumption:distance_and_energy} are equivalent to Assumption 2.1a, 2.1b and 2.1c of \cite{AmGiSa08}.
\end{remark}

We now make precise the definition of  EVI (Evolutional Variational Inequality) gradient flow of $\cE$. The important properties of this inequality are fully detailed in the monograph \cite{AmGiSa08} and in the more recent article \cite{MuSa20}.

\begin{definition}
    Given  $\kappa\in \R$, we define   \emph{solution of the $EV\!I_k$ inequality}   a continuous curve  $\gamma: [0,+\infty)\to E$  such that $ \gamma((0,+\infty))\subseteq \cD(\cE) $ and for all $\rho\in E$
		 \begin{equation}\label{item:ass_EVI}\tag{$EV\!I_{\kappa}$}
		\frac{1}{2} {\frac{\dd^+}{\dd t}} \left(d^2(\gamma(t),\rho)\right) \leq \cE(\rho) - \cE(\gamma(t)) - \frac{\kappa}{2} d^2(\gamma(t),\rho),\quad \forall \rho \in \cD(\cE),t\in [0,+\infty).
		\end{equation}
		Here $\frac{\dd^+}{\dd t}$ denotes the upper right time derivative.

		 An \emph{ $EV\!I_k$ gradient flow} of $\mathcal{E}$  defined in $D\subset \overline{\mathcal {D} (\cE)}$ is a family of continuous maps $S(t): D\to D, t\geq 0$  such that for every $\pi\in D$:
		 \begin{itemize}
	    \item The semigroup property holds 
	    \begin{equation}\label{ass:semigroup ppty and continuity} 
	    S[\pi](0)=\pi,\quad S[\pi](t+s)=S[S[\pi](t)](s) \quad \forall t,s\geq0.
	    \end{equation}
	    \item The curve $ (S[\pi](t))_{t\geq0}$ is a solution to \ref{item:ass_EVI}.
	    \end{itemize}
		   We shall refer to $(S[\pi](t))_{t\geq 0 }$ as  the \emph{gradient flow} of $\mathcal{E}$ started at $\pi$. 
		To lighten the notation, from now on, we will denote with $(\pi(t))_{t\geq 0}$ the gradient flow $(S[\pi](t))_{t\geq 0}$. 
\end{definition}

\begin{assumption} \label{assumption:gradientflow}
	 \textbf{[Gradient flow and  EVI]}  
	    We assume the existence of an \ref{item:ass_EVI} gradient flow of $\cE$ defined on  $D=E$. 
\end{assumption}

According to the above assumption we have that  $\overline{\mathcal {D} (\cE)}=E$.

We refer to Lemma \ref{lem: EVI implies boundedness} for the most important consequences of  \ref{item:ass_EVI} that will be used in our proofs, see also  (see \cite{MuSa20}).

For later use, we define the information functional as the squared slope of the energy. 

\begin{definition} \label{definition:information}
    We define  the  \textit{information functional} $I: E \to [0, +\infty]$ as 
		\begin{equation*}
		   I(\pi) : = \left\{\begin{array}{cc}
		      |\partial \cE|^2(\pi)  & \pi\in \cD(\cE) \\
		       +\infty  & \text{otherwise}
		    \end{array} \right..
		\end{equation*}
\end{definition}
The information functional is closely related to the gradient flow via the energy identity
\begin{equation*}
    \cE(\pi(t))-\cE(\pi(0))=-\int_0^t I(\pi(s))\De s,
\end{equation*}
see Lemma \ref{lem: EVI implies boundedness} for a rigorous version of the above relation.

In \cite{CoKrTo21}, we finished with an angle condition, a non-standard assumption that essentially captures the fact that $\cE$ is differentiable in its effective domain.

\begin{assumption} \label{assumption:regularized_geodesics}
     For any $\rho,\pi\in E$ satisfying  $I(\rho)+\cE(\pi)<+\infty$, there exist a geodesic $\geo{\rho}{\pi}$ such that, for any $\theta>0$, there exists $\tau>0$ and a curve, not necessarily a geodesic, $(\geo{\rho}{\pi}_{\theta}(t))_{t\in[0,\tau]}$ , satisfying  
            \begin{equation}\label{eq: angle condition} 
            \limsup_{t \downarrow 0}  \frac{d(\geo{\rho}{\pi}_{\theta}(t),\geo{\rho}{\pi}(t))}{t} \leq \theta ,\quad 
            \end{equation} 
            and 
            \begin{equation}\label{eq: energy directional derivative}
            \liminf_{t\downarrow 0} \frac{\cE(\geo{\rho}{\pi}_{\theta}(t))-\cE(\rho)}{t}\leq |\partial \cE |(\rho)(d(\rho,\pi)+\theta).
            \end{equation}
            Note that \eqref{eq: angle condition} implies that $\geo{\rho}{\pi}_{\theta}(0)=\rho.$ 
\end{assumption}

     \eqref{eq: energy directional derivative} can be interpreted as the controllability of directional derivatives of regularized geodesics by the local slope of the energy.

\subsection{Two collections of Hamiltonians}

We next formalize the two collections of upper and lower bounds introduced in Section \ref{section:intro:bounds_via_EVI}.

We start with the set of Hamiltonians in terms of smooth cylindrical test functions. Let $\cT$ be the collection of functions $\varphi$ defined as 
     \begin{equation} \label{eqn:defT}
    \cT := \left\{ \varphi \in \mathcal{C}_{\infty}([0,\infty)^{k};\mathbb{R}) \middle| \, k\in\mathbb{N},\,\forall \, i=1,\dots, k,  \, \partial_i \varphi > 0   \, \right\}, 
    \end{equation}
    and where $\cC_\infty([0,\infty)^{k};\mathbb{R})$ is the set of smooth functions mapping $[0,\infty)^{k}$ into $\bR$. Recall that for $\mu_1,\dots,\mu_k \in E$, we write  $\bm{\mu} = (\mu_1,\dots,\mu_k)$ and $\bm{\mu} \in {\cD(I) }$  if all elements in the vector are in $\cD(I)$. Moreover $d (\cdot,\bm\mu)= (d(\cdot,\mu_1),\dots, d(\cdot, \mu_k))$.

\begin{definition} \label{definition:H1}
	For $a > 0$, $\varphi \in \cT$ and $\rho \in \cD(I)$, and $\bm\mu=(\mu_1,\ldots,\mu_{k}) \in E^{k}$ such that $\bm{\mu} \in \cD(I)$,  we define $f^{\dagger} = f^\dagger_{a,\varphi,\rho,\bm\mu} \in C_l(E)$ and $g^{\dagger}=g^\dagger_{a,\varphi,\rho,\bm\mu} \in USC(E)$ for all $\pi \in E$ as
    \begin{align}
	    f^\dagger(\pi) & := \frac{a}{2} d^2(\pi,\rho) + \varphi\left( \frac{1}{2}d^2(\pi,\bm\mu)\right), \label{eq:reg_f_1dag} \\
	    g^\dagger(\pi) & := a \left[\cE(\rho) - \cE(\pi) - \frac{\kappa}{2} d^2(\pi,\rho) \right] + \frac{a^2}{2} d^2(\pi,\rho) \label{eq:reg_g_1dag}\\
     & \qquad +\sum_{i=1}^k \partial_i \varphi\left( \frac{1}{2}d^2(\pi,\bm\mu)\right) \left[\cE(\mu_i) - \cE(\pi) - \frac{\kappa}{2} d^2(\pi,\mu_i) \right]  \notag \\
    & \qquad + \frac{1}{2} \left( \sum_{i = 1}^k \partial_i \varphi\left( \frac{1}{2}d^2(\pi,\bm\mu)\right) d(\pi,\mu_i) \right)^2 \notag \\
    & \qquad + a d(\pi,\rho) \left(\sum_{i = 1}^k \partial_i \varphi\left( \frac{1}{2}d^2(\pi,\bm\mu)\right) d(\pi,\mu_i)\right) \notag
	\end{align}
	and set $H \subseteq C_l(E) \times USC(E)$ by
	\begin{equation*}
	    H_{\dagger} := \left\{ (f^{\dagger},g^{\dagger}) \, \middle| \, \forall \varphi \in \cT, a > 0, \rho \in \cD(I), \bm\mu\in {\cD(I) } \right\}.
	\end{equation*}
	In the same way, for $a > 0$, $\varphi \in \cT$, $\mu \in \cD(I)$ and $\bm\pi=(\pi_1,\ldots,\pi_{k})\in E^{k}$ such that $\bm\pi\in {\cD(I) }$ we define 
 	$f^{\ddagger} = f^\ddagger_{a,\varphi,\mu,\bm\pi} \in C_u(E) $ and $g^{\ddagger} = g^\ddagger_{a,\varphi,\mu,\bm\pi} \in LSC(E)$ for all $\mu \in E$ as
	\begin{align*}
    f^{\ddagger}(\mu) & := - \frac{a}{2} d^2(\mu,\gamma) -\varphi\left( \frac{1}{2}d^2(\mu,\bm\pi)\right), \\
	    g^{\ddagger}(\mu) & := a \left[\cE(\mu) -\cE(\gamma) + \frac{\kappa}{2}d^2(\mu,\gamma) \right] + \frac{a^2}{2} d^2(\gamma,\mu) \\
        & \qquad + \sum_{i=1}^k \partial_i \varphi\left(\frac{1}{2}d^2(\mu,\bm\pi)\right)\left[\cE(\mu) -\cE(\pi_i) + \frac{\kappa}{2}d^2(\mu,\pi_i) \right] \label{eq:reg_g_ddag}\\
        & \nonumber \qquad - \frac{1}{2}  \left(\sum_{i=1}^k  \partial_i \varphi\left( \frac{1}{2}d^2(\mu,\bm\pi)\right) d(\mu,\pi_i)\right)^2 \\
    	& \nonumber \qquad - a d(\mu,\gamma)\left(\sum_{i=1}^k  \partial_i \varphi\left( \frac{1}{2}d^2(\mu,\bm\pi)\right) d(\mu,\pi_i)\right)
	\end{align*}	
	and set $H_\ddagger \subseteq C_u(E) \times LSC(E)$ by
	\begin{equation*}
	    H_{\ddagger} := \left\{ (f^{\ddagger},g^{\ddagger}) \, \middle| \, \forall \varphi \in \cT, a >0, \gamma \in \cD(I), \bm\pi \in {\cD(I) } \right\}.
	\end{equation*}

\end{definition}

\begin{remark}
    A starting point for our analysis could have been a set of Hamiltonians in which the test functions are build up from a smooth and bounded $\varphi$ and $a = 0$. We consider this setting in Appendix \ref{appendix:boundedcylinders}. A distinct advantage arises, however, for the choice made in Definition \ref{definition:H1}, one that is exploited in \cite{CoKrTaTo24}.

    Suppose that $E$ can be equipped with a weaker topology, and suppose that a (candidate) subsolution $u$ of $f - \lambda H_{\dagger} f = h$ is upper semi-continuous for the weaker topology. Then for any $f \in \cD(H_{\dagger})$ there exists a $\pi$ such that
    \begin{equation*}
        u(\pi) - f(\pi) = \sup u-f,
    \end{equation*}
    which is of significant help for further arguments.
\end{remark}

We proceed with the Hamiltonians that include the Tataru distance function in the domain. The definition follows that of \cite{CoKrTo21}.

\begin{definition} \label{definition:Tataru}
We define the Tataru distance {$d_T: E\times E \to [0,+\infty)$} with respect to the metric $d$ and energy $\cE$ as
    \begin{equation*}
    d_T(\pi,\rho) = \inf_{t \geq 0} \left\{ t + e^{\hat{\kappa} t} d(\pi,\rho(t))\right\}, \quad \forall \pi, \rho \in E,
\end{equation*}
where $\hat{\kappa} = (0 \wedge \kappa)\leq 0$.
\end{definition}

Computations done in \cite{CoKrTo21}, based on the Lipschitz property of $d_T$, lead to the following definition for a second pair of upper and lower Hamiltonians. Note that in this case we prefer to underline in the definition the fact that the Hamiltonians are operators. 

\begin{definition} \label{definition:HdaggerHddagger} [Non smooth Hamiltonians]

\begin{enumerate}[(1)]
    \item For each $a > 0, b > 0, c \in \bR$, and $\mu,\rho\in E : \, \cE(\rho) < \infty$ let $f^\dagger = f^\dagger_{a,b,c,\mu,\rho} \in C_l(E)$ and  $g^\dagger = g^\dagger_{a,b,c,\mu,\rho} \in USC(E)$  
    be given for any $\pi \in E$ by
    \begin{align*}
        f^\dagger(\pi) & := \frac{1}{2}a d^2(\pi,\rho) + b d_T(\pi,\mu) + c \\
        g^\dagger(\pi) & := a\left[ \cE(\rho) - \cE(\pi)\right] - a \frac{\kappa}{2} d^2(\pi,\rho) + b  + \frac{1}{2} a^2 d^2(\pi,\rho) + ab d(\pi,\rho) 
+ \frac{1}{2} b^2.
    \end{align*}
    Then the  operator $\widetilde H_\dagger \subseteq C_l(E) \times USC (E)$ is defined by
    \begin{equation*}
        \widetilde{H}_\dagger := \left\{\left(f^\dagger_{a,b,c,\mu,\rho},g^\dagger_{a,b,c,\mu,\rho}\right) \, \middle| \, a , b>0,c\in \bR, \mu,\rho\in E : \, \cE(\rho) < \infty \right\}.
    \end{equation*}
    \item For each $a > 0, b > 0, c \in \bR$, and $\pi, \gamma\in E : \, \cE(\gamma) < \infty$ let $f^\ddagger = f^\ddagger_{a,b,c,\pi,\gamma} \in  C_u(E)$ and $g^\ddagger = g^\ddagger_{a,b,c,\pi,\gamma} \in LSC(E)$  be given for any $\mu\in E$ by
    \begin{align*}
        f^\ddagger(\mu) & := -\frac{1}{2}a d^2(\gamma,\mu) - b d_T(\mu,\pi) + c \\
        { g^\ddagger(\mu) } & := { a\left[ \cE(\mu) - \cE(\gamma)\right] + a \frac{\kappa}{2} d^2(\gamma,\mu) - b  + \frac{1}{2} a^2 d^2(\gamma,\mu) - ab d(\gamma,\mu) - \frac{1}{2}b^2. }
    \end{align*}
    Then the  operator $\widetilde H_\ddagger \subseteq  C_u(E)  \times LSC(E)$ is defined by
    \begin{equation*}
        \widetilde{H}_\ddagger := \left\{\left(f^\ddagger_{a,b,c,\pi,\gamma},g^\ddagger_{a,b,c,\pi,\gamma}\right) \, \middle| \, a,b>0,c\in \bR, \pi,\gamma\in E : \, \cE(\gamma) < \infty \right\}.
    \end{equation*}
\end{enumerate}
\end{definition}

\subsection{Main results}

To state the main results, we first precise the notion of solution we are looking for. We will state it for general Hamiltonians $A_\dagger\subseteq C_l(E) \times USC (E)$ and $A_\ddagger\subseteq C_u(E) \times LSC (E)$.

	\begin{definition} \label{definition:viscosity_solutions_HJ_sequences}
		 Fix $\lambda > 0$ and $h^\dagger,h^\ddagger \in C_b(E)$. Consider the equations
		\begin{align} 
		f - \lambda  A_\dagger f & = h^\dagger, \label{eqn:differential_equation_tildeH1} \\
		f - \lambda A_\ddagger f & = h^\ddagger. \label{eqn:differential_equation_tildeH2}
		\end{align}

		\begin{gather}
		\lim_{n \uparrow \infty} u(\pi_n) - f(\pi_n)  = \sup_\pi u(\pi) - f(\pi), \label{eqn:viscsub1} \\
		\limsup_{n \uparrow \infty}  u(\pi_n) - \lambda g(\pi_n) - h^\dagger(\pi_n) \leq 0. \label{eqn:viscsub2}
		\end{gather}

		We say that $v$ is a \textit{(viscosity) supersolution} of equation \eqref{eqn:differential_equation_tildeH2} if $v$ is bounded, lower semi-continuous and if for all $(f,g) \in A_\ddagger$ there exists a sequence $(\pi_n)_{n\in \mathbb N}\in E$ such that
		\begin{gather*}
		\lim_{n \uparrow \infty} v(\pi_n) - f(\pi_n)  = \inf_\pi v(\pi) - f(\pi),  \\
		\liminf_{n \uparrow \infty} v(\pi_n) - \lambda g(\pi_n) - h^\ddagger(\pi_n) \geq 0.  
		\end{gather*}
		If $h^\dagger = h^\ddagger$, we say that $u$ is a \textit{(viscosity) solution} of equations \eqref{eqn:differential_equation_tildeH1} and \eqref{eqn:differential_equation_tildeH2} if it is both a subsolution of \eqref{eqn:differential_equation_tildeH1} and a supersolution of \eqref{eqn:differential_equation_tildeH2}.
		
		We say that \eqref{eqn:differential_equation_tildeH1} and \eqref{eqn:differential_equation_tildeH2} satisfy the \textit{comparison principle} if for every subsolution $u$ to \eqref{eqn:differential_equation_tildeH1} and supersolution $v$ to \eqref{eqn:differential_equation_tildeH2}, we have $\sup_E u-v \leq \sup_E h^\dagger - h^\ddagger$.
	\end{definition}

	In classical works on viscosity solutions, instead of working with the statement "there exists a sequence such that...", one has 'for all optimizers one has'. Even though the classical stronger definition has advantages when proving the comparison principle, the weaker definition allows for easier approximation arguments that are needed later on in our setting, see Section \ref{sec: push over}, Lemma \ref{lemma:viscosity_push} and Proposition \ref{proposition:Hamiltonian_convergence_pseudo_coercive}.

    \smallskip

    We next state the main result of the paper. Recall that the weak upper semi-continuous regularization of an upper semi-continuous function $u^*$ is the smallest weak upper semi-continuous function $\hat u$ such that $\hat u \geq u$. Analogously for a weak lower semi-continuous function.

    \begin{theorem} \label{theorem:carry_over_solutions}
    Let Assumptions \ref{assumption:distance_and_energy}, \ref{assumption:weak_topology} and \ref{assumption:gradientflow} be satisfied. Let $\lambda > 0$ and $h \in C_b(E)$ be weakly continuous. We then have 
    \begin{enumerate}[(a)]
        \item Let $u$ be a viscosity subsolution to $f - \lambda H_{\dagger}f = h$, then the weak upper semi-continuous regularization $u^*$ of $u$ is a viscosity subsolution to $f - \lambda \widetilde{H}_\dagger f = h$.
        \item Let $v$ be a viscosity supersolution to $f - \lambda H_{\ddagger} f = h$. Then the weak lower semi-continuous regularization $v_*$ of $v$ is a viscosity supersolution to $f - \lambda \widetilde{H}_\ddagger f = h$.
    \end{enumerate}
    \end{theorem}

    Our main result is of relevance in combination with the associated comparison principle for $\widetilde{H}_\dagger,\widetilde{H}_\ddagger$ of \cite{CoKrTo21} and the existence result established in the context of the Wasserstein space in \cite{CoKrTaTo24}, developments for a larger class of spaces is postponed to future work. 
    We refer to \cite{CoKrTaTo24} for the complete picture in the Wasserstein context, and now proceed to give the comparison principle for $H_\dagger,H_\ddagger$ that follows as a consequence of \cite[Theorem 2.13 and Remark 2.14]{CoKrTo21}, repeated here for completeness.

    \begin{theorem} \label{theorem:comparison_Tataru}
	Let Assumptions \ref{assumption:distance_and_energy}, \ref{assumption:gradientflow} and \ref{assumption:regularized_geodesics} be satisfied. Let $\lambda>0$ and  $h^\dagger,h^\ddagger \in C_b(E)$ be uniformly continuous.
	
	Let $u$ be a  viscosity subsolution to $f - \lambda \widetilde{H}_\dagger f = h^\dagger$ and let $v$ be a viscosity supersolution to $f - \lambda \widetilde{H}_\ddagger f = h^\ddagger$. Then we have 
	\begin{equation*}
	    \sup_{\pi \in E} u(\pi) - v(\pi) \leq \sup_{\pi \in E} h^\dagger(\pi) - h^\ddagger(\pi).
	\end{equation*}
	
	The same result holds for $h^{\dagger}, h^{\ddagger} \in C_b(E)$ that are uniformly continuous on sets of the type
    \begin{equation*}
    K^{\rho}_{c,d} := \{ \pi\in E\ : d(\rho,\pi) \leq c, \cE(\pi)\leq d \}.
    \end{equation*}
    \end{theorem}

    Theorem \ref{theorem:comparison_Tataru} in combination with our main result leads to the comparison principle also for the pair $(H_\dagger,H_\dagger)$. For this note that: 
    \begin{itemize}
        \item The construction of upper and lower semicontinuous regularizations yields $u-v \leq u^* - v_*$, so that it suffices to establish an upper bound for $u^* - v_*$.
        \item In the context of Assumption \ref{assumption:weak_topology}, any $h \in C_b(E)$ that is weakly continuous is uniformly continuous on the sets $K^{\rho}_{c,d}$ as in Theorem \ref{theorem:comparison_Tataru}. It follows, that under this assumption Theorem \ref{theorem:comparison_Tataru} can be formulated in terms of $h^\dagger,h^\ddagger \in C_b(E)$ that are weakly continuous.
    \end{itemize}

We thus obtain the following corollary from Theorems \ref{theorem:carry_over_solutions} and \ref{theorem:comparison_Tataru}.

	\begin{corollary} \label{corollary:comparison_principle}
	Let Assumptions \ref{assumption:distance_and_energy}, \ref{assumption:weak_topology}, \ref{assumption:gradientflow} and \ref{assumption:regularized_geodesics} be satisfied. Let $\lambda>0$ and  $h^\dagger,h^\ddagger \in C_b(E)$ be weakly continuous.
	
	Let $u$ be a  viscosity subsolution to $f - \lambda H_\dagger f = h^\dagger$ and let $v$ be a viscosity supersolution to $f - \lambda H_\ddagger f = h^\ddagger$. Then we have 
	\begin{equation*}
	    \sup_{\pi \in E} u(\pi) - v(\pi) \leq \sup_{\pi \in E} h^\dagger(\pi) - h^\ddagger(\pi).
	\end{equation*}
	\end{corollary}

\section{Preparations for the proof of Theorem \ref{theorem:carry_over_solutions}} \label{sec: push over} \label{section:preparations_proof}

In this section, we set up the proof of Theorem \ref{theorem:carry_over_solutions}. The key issue, c.f. \eqref{eqn:approx_Tataru}  that needs to be overcome, is to approximate the Tataru distance by a smooth approximation
\begin{equation*} 
    \varphi\left( \frac{1}{2}d^2(\pi,\rho_1),\dots, \frac{1}{2}d^2(\pi,\rho_k)\right) \approx b d_T(\pi,\rho),
\end{equation*}
and to show that the behaviour of this approximation along the gradient flow is similar to that of the Tataru distance.

\subsection{Approximating the Tataru distance} \label{section:strategy:approximatingTataru}

To approximate the Tataru distance, we make use of the classical Laplace type approximation result
\begin{equation} \label{eqn:Laplace_asymptotics_intro}
    \lim_{m \rightarrow \infty} \frac{1}{m} \log \int_0^\infty e^{m f(t)} \dd t = \sup_{t} f(t).
\end{equation}
For the Tataru distance, this translates to
\begin{equation} \label{eqn:def_Tataru_approx}
    d_T(\pi,\mu) = \inf_{t \geq 0} \left\{t + e^{\hat{\kappa}t}d(\pi,\mu(t))\right\} = \lim_{m \rightarrow \infty} - \frac{1}{m} \log \int_0^\infty e^{-m(t + e^{\hat{\kappa}t} d(\pi,\mu(t)))} \dd t.
\end{equation}
{It was first observed by \cite{Fe06} that  the flow can be  controlled along this approximation, motivating a further exploration of \eqref{eqn:def_Tataru_approx}.}

Indeed, we see that the Tataru distance can be approximated by an integral over objects involving the ordinary distance. In turn, the integral can be approximated by a Riemann sum. This will lead us  close to  an object of this type
\begin{equation}\label{eqn:riemann_sum_in_terms_of_d2}
    \pi \mapsto \varphi\left(\frac{1}{2} d^2(\pi,\mu(t_1)),\dots, \frac{1}{2} d^2(\pi,\mu(t_k)) \right)
\end{equation}
for a finite collection of times $\{t_i\}_{i=1,\dots,k}$ and $\varphi \in \cT$. To end up with something of the type \eqref{eqn:riemann_sum_in_terms_of_d2}, however, we first need to approximate the distance by a smooth function of its square. We thus approximate the function $r \mapsto r$ by $\psi_\varepsilon( \frac{1}{2}r^2)$ where $\psi_\varepsilon$ is a smooth approximation of $r \mapsto \sqrt{2r}$. Therefore, the procedure above will be applied to  a smoothed version of the Tataru distance:
\begin{multline} \label{eqn:smoothed_version_Tataru}
    d_{T,\varepsilon}(\pi,\mu) := \inf_{t\geq 0}\left\{t + e^{\hat{\kappa}t} \psi_\varepsilon\left(\frac{1}{2}d^2(\pi,\mu(t))\right) \right\} \\
    = \lim_{m \rightarrow \infty} - \frac{1}{m} \log \int_0^\infty e^{-m\left(t + e^{\hat{\kappa}t} \psi_\varepsilon\left(\frac{1}{2}d^2(\pi,\mu(t))\right)\right)} \dd t.
\end{multline}
On one hand, we will show that $d_{T,\varepsilon} \rightarrow d_T$ uniformly. On the other hand, we will approximate the integral depending on  $\psi_\varepsilon\left(\frac{1}{2}d^2(\pi,\mu(t))\right)$ in terms of Riemann sums. This can indeed be done by means of function of the form \eqref{eqn:riemann_sum_in_terms_of_d2} for a finite collection of times $t_1,\ldots,t_k$ and $\varphi\in\cT_b$. 

\smallskip

The approximation steps above, c.f. $m \rightarrow \infty$, $\varepsilon \downarrow 0$ and approximating the integral by a Riemann sum, lead to a chain of intermediate results that together will imply Theorem \ref{theorem:carry_over_solutions}. In addition to the three important approximation steps mentioned above, we will carry out some additional smaller steps that are e.g. to remove the assumption that a test function is bounded, or to relax the assumption that we compare with a configuration in the domain of the Fisher energy to that of the energy itself.

\smallskip

Effectively, the proof of Theorem \ref{theorem:carry_over_solutions} is composed of six steps. At each but the last step we define new upper and lower bounds $H_{i,\dagger},H_{i,\ddagger}, i=2, \ldots, 6$ and prove a statement that, for the equations involving subsolutions, loosely speaking looks like
\begin{equation*}
u \,\, \text{subsolution of} \,\,u-\lambda H_{i,\dagger}u=h \quad \Rightarrow u \,\, \text{subsolution of}\,\, u-\lambda H_{i+1,\dagger}u=h,
\end{equation*}
with the understanding that $H_{1,\dagger}= H_{\dagger}$ and $H_{7,\dagger}= \widetilde{H}_{\dagger}$. As $i$ increases, $H_{i,\dagger}$ acts on test functions that get closer to a test function that includes the Tataru distance. 

\smallskip

In order to push over viscosity solutions, in addition to showing that we have the above described  chain of convergence of test functions, we need to show that the action of the Hamiltonian, or more specifically the action of the gradient flow, on these test functions behaves in the right way as well.

We will do so relying on the technical Lemma \ref{lemma:viscosity_push} and Proposition \ref{proposition:Hamiltonian_convergence_pseudo_coercive} below and exploiting in a systematic way the regularizing effects and dissipation estimates implied by \ref{item:ass_EVI}-gradient flows. The application of Proposition \ref{proposition:Hamiltonian_convergence_pseudo_coercive}, we need one additional property, namely that we can work with weak upper or lower semi-continuous viscosity sub or supersolutions respectively. Using Lemma \ref{lemma: regularization} that for the operators $H_{1,\dagger},H_{1,\ddagger}$ we show that we can indeed work with the weak regularizations.

We refer to Figure \ref{figure:HJ_implications} for a formal overview of the six approximation steps that we will be carrying out, the most delicate and technical ones being those that allow to go from $H_{2,\dagger}$ to $H_{3,\dagger}$ (resp. from $H_{2,\ddagger}$ to $H_{3,\ddagger}$) and from $H_{3,\dagger}$ to $H_{4,\dagger}$ (resp. from $H_{3,\ddagger}$ to $H_{4,\ddagger}$). 

In the rest of the section, we state and prove our main tools. We will then carry out the proofs of the steps in Figure \ref{figure:HJ_implications} in Section \ref{proof:carry_over_solutions}.

\begin{figure}
	\centering

    \begin{tikzpicture}
	\matrix (m) [matrix of math nodes,row sep=1em,column sep=4em,minimum width=2em]
	{
		H_\dagger = H_{1,\dagger} &[-5mm] { } &[-5mm] H_\ddagger = H_{1,\ddagger} \\
		{ } & \substack{\text{Replace general} \\ \text{test functions by} \\ \text{Riemann sums}} & { } \\
		H_{2,\dagger} & { } & H_{2,\ddagger} \\
		{ } & \substack{\text{Approximate } \\ \text{Laplace integral by} \\ \text{Riemann sums}} & { } \\
		H_{3,\dagger} & { } & H_{3,\ddagger} \\
		{ } & \substack{\text{Carry out} \\ \text{Laplace asymptotics} \\ \text{to approximate smoothened versions} \\ \text{of the Tataru distances}} & { } \\
		H_{4,\dagger} & { } & H_{4,\ddagger} \\
		{ } & \substack{\text{Bound the outcome} \\ \text{by using the properties} \\ \text{of the gradient flow}} & { } \\
		H_{5,\dagger} & { } &H_{5,\ddagger} \\
		{ } & \substack{\text{Remove the} \\ \text{smoothing in the} \\ \text{Tataru distances}} & { } \\
		H_{6,\dagger} & { } & H_{6,\ddagger} \\
		{ } & \substack{\text{Perform a uniform closure} \\ \text{of the graphs to remove} \\ \text{the finite Fisher information}} & { } \\
	    \widetilde{H}_{\dagger} = H_{7,\dagger} & { } & \widetilde{H}_{\ddagger} = H_{7,\ddagger} \\
	};
	\path[-stealth]
	(m-1-1) edge node [left] {sub} (m-3-1)
	(m-7-1) edge node [left] {sub} (m-9-1)
	(m-9-1) edge node [left] {sub} (m-11-1)
	(m-11-1) edge node [left] {sub} (m-13-1)
	(m-1-3) edge node [right] {super} (m-3-3)
	(m-7-3) edge node [right] {super} (m-9-3)
	(m-9-3) edge node [right] {super} (m-11-3)
	(m-11-3) edge node [right] {super} (m-13-3);

	\path[dashed, -stealth]
	(m-3-1) edge node [left] {sub} (m-5-1)
	(m-5-1) edge node [left] {sub} (m-7-1)
    (m-3-3) edge node [right] {super} (m-5-3)
	(m-5-3) edge node [right] {super} (m-7-3);
	
	\path[dotted, -stealth]
    (m-1-1) edge [loop above] node {weak reg.  sub} (m-1-1)
    (m-1-3) edge [loop above] node {weak reg.  super} (m-1-3);

	\end{tikzpicture}
	\caption{Carrying over viscosity sub- and supersolutions. \\In this diagram, an arrow connecting an operator $A$ with operator $B$ with subscript 'sub' means that viscosity subsolutions of $f - \lambda A f = h$ are also viscosity subsolutions of $f - \lambda B f = h$. If a line is dashed, the result holds under the additional assumption that the subsolution is weakly upper semi-continuous. 	Similarly for arrows with a subscript 'super'. \\
	The loop equipped with 'weak reg. sub' means that the weak upper semi-continuous regularization of any viscosity subsolution is also a viscosity subsolution. Similarly for the loop decorated with 'weak reg. super'.
 }
    \label{figure:HJ_implications} 
	\end{figure}
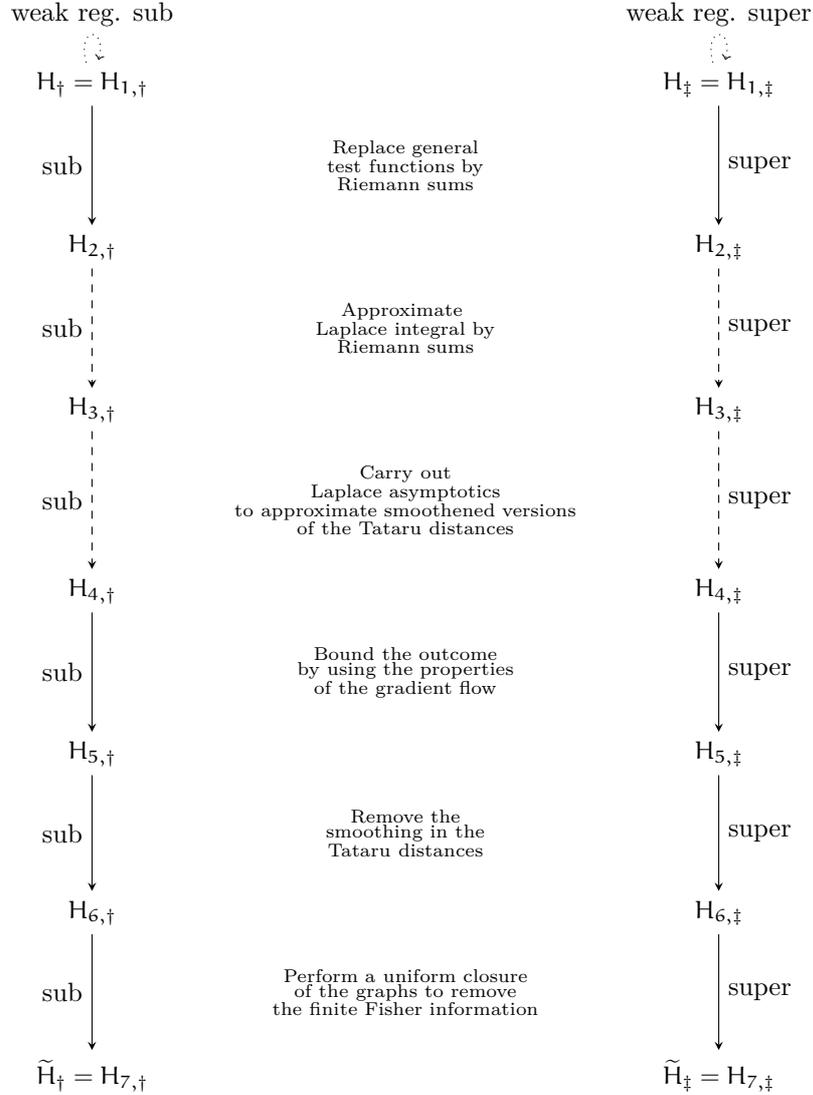

\subsection{A toolbox for pushing over sub- and supersolutions}
The following tool  will allow us to push over subsolutions, from an Hamiltonian to another, in most of the intermediate steps.

\begin{lemma}[Lemma 7.6 of \cite{FK06}] \label{lemma:viscosity_push}
	Suppose that $A_\dagger, \widehat{A}_\dagger \subseteq LSC_l(E) \times USC(E)$ and $A_\ddagger, \widehat{A}_\ddagger \subseteq USC_u(E) \times LSC(E)$.   Let $h \in C_b(E)$.

	\begin{enumerate}[(a)]
		\item Suppose for each $(f^\dagger,g^\dagger) \in A_\dagger$ there are $(f_k,g_k)_{k\geq 1} \in \widehat{A}_\dagger$ such that for all $c,d \in \bR$ we have 
		\begin{align*}
		& \lim_{k \rightarrow \infty} \vn{f_k \wedge c - f^\dagger\wedge c} = 0, \\
		& \limsup_{k \rightarrow \infty} \sup_{\mu: f_k(\mu) \vee f^\dagger(\mu) \leq c} g_k(\mu) \vee d - g^\dagger(\mu)\vee d \leq 0.
		\end{align*}
		If $u$ is a viscosity subsolution to $f - \lambda \widehat A_\dagger f = h$, then it is also a viscosity subsolution to $f - \lambda {A}_\dagger f = h$.
		\item Suppose for each $(f^\ddagger,g^\ddagger) \in A_\ddagger$ there are $(f_k,g_k)_{k\geq 1} \in \widehat{A}_\ddagger$ such that for all $c,d \in \bR$ we have 
		\begin{align*}
		& \lim_{k \rightarrow \infty} \vn{f_k \vee c - f^\ddagger\vee c} = 0, \\
		& \liminf_{k \rightarrow \infty} \inf_{\mu: f_k(\mu) \wedge f(\mu) \geq c} g_k(\mu) \wedge d - g^\ddagger(\mu)\wedge d \geq 0.
		\end{align*}
		If $v$ is a viscosity supersolution to $f - \lambda \widehat A_\ddagger f = h$, then it is also a viscosity supersolution to $f - \lambda {A}_\dagger f = h$.
	\end{enumerate}
	
\end{lemma}

However, in order to pass from $H_{2,\dagger}$ to $H_{3,\dagger}$  and from $H_{3,\dagger}$ to $H_{4,\dagger}$, we need a more elaborate machinery than Lemma \ref{lemma:viscosity_push}. This is due to the fact that uniform estimates are too much to ask for in this context. We will therefore argue on this novel proposition that allows us to obtain a similar result. Recall that the set $K^\rho_{c,d}$ has been defined in \eqref{def: K^rho} for $\rho\in E, c,d \in \R$.

\begin{proposition} \label{proposition:Hamiltonian_convergence_pseudo_coercive}
{Suppose that $A_\dagger, \widehat{A}_\dagger \subseteq LSC_l(E) \times USC(E)$ and $A_\ddagger, \widehat{A}_\ddagger \subseteq USC_u(E) \times LSC(E)$  and } that for all $(f^\dagger,g^\dagger) \in  {A_\dagger}$, there are $(f_k,g_k)_{\kappa\geq1} \in  \widehat A_\dagger$ such that the following conditions are satisfied. 
\begin{enumerate}[(a)]
    \item \label{item:Ham_conv_weaktop_compactness} 

There exists a continuous function $\omega_1 : \bR \rightarrow [0,+\infty)$ such that $\lim_{r \rightarrow +\infty} \omega_1(r) = +\infty$ and $\tilde \rho\in E$ satisfying
\begin{equation}\label{eq:weak_compactness_via_levelsets 1}
  f_k(\pi) \geq \omega_1(d(\pi,\tilde\rho)) ,\quad \forall k\geq 1, \,\pi\in E. 
\end{equation}
Moreover, for any $R>0$, there exist a continuous function $\omega_{2,R} : \bR \rightarrow [0,+\infty)$ such that $\lim_{r \rightarrow +\infty} \omega_{2,R}(r) = +\infty$ satisfying
\begin{equation}\label{eq:weak_compactness_via_levelsets 2}
   g_k(\pi) \leq - \omega_{2,R}(\cE(\pi)), \quad  \forall k\geq1, \,\pi \in B_{R}(\tilde \rho).
\end{equation}
    \item \label{item:Ham_conv_weaktop_equi_pointwise} For any $\pi \in E$ we have
    \begin{equation*}
        \lim_{k \to +\infty} f_k(\pi) = f^\dagger(\pi).
    \end{equation*}
    \end{enumerate}
    {Consider $\tilde{\rho} \in E$ as in \ref{item:Ham_conv_weaktop_compactness} and let $c,d \in \bR$.}
    \begin{enumerate}[(a),resume]
    \item \label{item:Ham_conv_weaktop_equi_uniform_on_compacts_f} For any $(\pi_k)_{k\in\mathbb N}\subseteq K^{\tilde \rho}_{c,d}$ that  converges weakly to $\pi \in K^{\tilde \rho}_{c,d}$ we have
    \begin{equation*}
        \liminf_{k \to +\infty} f_k(\pi_k) \geq f^\dagger(\pi).
    \end{equation*}
    \item \label{item:Ham_conv_weaktop_equi_uniform_on_compacts_g}  For any $(\pi_k)_{k\in\mathbb N}\subseteq K^{\tilde \rho}_{c,d}$ that  converges weakly to $\pi \in K^{\tilde\rho}_{c,d}$ and that is such that
    \begin{equation*}
        \lim_{k \to +\infty} f_k(\pi_k) = f^\dagger(\pi)
    \end{equation*}
    we have
    \begin{equation}\label{eq: HJB transfer 1}
        \limsup_{k \to +\infty} g_k(\pi_k) \leq g^\dagger(\pi).
    \end{equation}
    \end{enumerate}

Fix $\lambda > 0$ and let $h \in C_b(E)$ be weakly continuous. 

Let $u$ be a bounded and weakly upper semi-continuous viscosity subsolution to $f - \lambda \widehat A_\dagger f = h$. Then it is also a viscosity subsolution to $f - \lambda {A_\dagger} f = h$.

	\vspace{0,1cm}
The result holds also for viscosity supersolutions with appropriate modifications. 
\end{proposition} 

In the proposition, we assume that $u$,$v$ have appropriate weak semi-continuity properties, these are properties that are a-priori not known. We therefore include immediately a lemma that shows that, given that the test functions in the considered Hamiltonians  have themselves weak semi-continuity properties, we can replace a viscosity sub or supersolution by its weak upper or lower semi-continuous regularization.

\begin{lemma}\label{lemma: regularization} 
    {Let $h \in C_b(E)$ and $\lambda >0$, $A_\dagger \subseteq LSC_l(E) \times USC(E)$ and $A_\ddagger \subseteq USC_u(E) \times LSC(E)$.}

	Suppose that for all $(f^\dagger,g^\dagger) \in A_\dagger$ the function $f\in LSC_l(E)$ is weakly lower semi-continuous.
	
	Let $u$ be a viscosity subsolution to $f - \lambda A_\dagger f = h$. Then the weak upper semi-continuous regularization $u^*$ is also a viscosity subsolution to $f - \lambda A_\dagger f = h$.
	
	\vspace{0,1cm}
	Suppose that for all $(f^\ddagger,g^\ddagger) \in A_\ddagger$ the function {$f\in USC_u(E)$} is weakly upper semi-continuous.
	
	Let $v$ be a viscosity subsolution to $f - \lambda A_\ddagger f = h$. Then the weak lower semi-continuous regularization $v_*$ is also a viscosity supersolution to $f - \lambda A_\ddagger f = h$.
\end{lemma}

\begin{proof}[Proof of Proposition \ref{proposition:Hamiltonian_convergence_pseudo_coercive}]

First of all, let us recall that, 
 by Assumption \ref{assumption:weak_topology}, the set $K^{\tilde \rho}_{c,d}$ is  weakly compact for all $c,d\in\R$. 

Let $(f^\dagger, g^\dagger)\in {A_\dagger}$ and $u$ be a bounded and weakly upper semi-continuous viscosity subsolution to $f - \lambda \widehat A_\dagger f = h$ and $(f_k,g_k)_{k\geq 1}$ be the corresponding approximating sequence. Consider $(\varepsilon_k)_{k\geq 1} > 0$ be such that $\varepsilon_k \downarrow 0$. By the subsolution property applied to $(f_k, g_k)$ we can find   $(\pi_k)_{k\in \mathbb{N}}$ such that
\begin{equation} \label{eqn:quasi_coercive_Hamiltonian_seq1}
    u(\pi_k) - f_k(\pi_k) \geq \sup_{\rho\in E} \left\{u(\rho) - f_k(\rho)\right\} - \varepsilon_k
\end{equation}
and
\begin{equation}\label{eqn:quasi_coercive_Hamiltonian_seq2}
    u(\pi_k) - \lambda g_k(\pi_k) \leq h(\pi_k) + \varepsilon_k.
\end{equation}

Let us choose $\rho_0\in E$. 
After some straightforward calculations, from \eqref{eqn:quasi_coercive_Hamiltonian_seq1}, we deduce that
\begin{equation*}
 \sup_{k\geq 1} f_k(\pi_k) \leq 2\sup_{\gamma\in E} |u(\gamma)| + \sup_{k\geq 1} f_k(\rho_0) + \varepsilon_1 
 \end{equation*}
and the latter is finite because of \ref{item:Ham_conv_weaktop_equi_pointwise}. 
On the other hand, from \eqref{eqn:quasi_coercive_Hamiltonian_seq2} and the boundedness of $u$ and $h$ we have
\begin{equation*}
 \inf_{k\geq 1} g_k(\pi_k) \geq -\frac{\sup_{\gamma\in E}(|u|+|h|)(\gamma)+ \varepsilon_1}{\lambda} > -\infty
\end{equation*}

Then using \ref{item:Ham_conv_weaktop_compactness} we deduce that there exist $\tilde \rho\in E$ and $c,d\in\R$ such that $(\pi_k)_{k\geq 1}\subseteq K^{\tilde \rho}_{c,d}$ and thanks to Assumption \eqref{assumption:weak_topology} we can assume without loss of generality that $\pi_k$ converges {weakly} to some $\pi_\infty \in E$.

\smallskip

We establish that $u$ is a viscosity subsolution for $f - \lambda  A_\dagger f = h$ 
in three steps.
\begin{itemize}
    \item \textit{Step 1}. We first establish that
    \begin{equation} \label{eqn:push_visc_sol_optimum}
    \begin{split}
         u(\pi_\infty)-f^\dagger(\pi_\infty)&=\sup_{\rho\in E} \left\{ u(\rho) - f^\dagger(\rho) \right\}  \\
         &= \lim_{k\to +\infty} \sup_{\rho\in E} \left\{ u(\rho) - f_k(\rho)\right\}  = \lim _{k \to+\infty} \{u(\pi_k) - f_k(\pi_k)\}.
    \end{split}
    \end{equation}
    \item \textit{Step 2}. 
    We next establish that the outcome of Step 1 implies 
    \begin{align} 
    \lim_{k \to+\infty} u(\pi_k) & = u(\pi_\infty), \label{eqn:convergence_optima_in_subsol} \\
    \lim_{k \to +\infty} f_k(\pi_k) & = f^\dagger(\pi_\infty),\label{eqn:convergence_optima_in_testfunctions} 
\end{align}

    \item \textit{Step 3}. We establish that
    \begin{equation}
        u(\pi_\infty) - \lambda g^\dagger(\pi_\infty) - h(\pi_\infty) \leq 0.
    \end{equation}
\end{itemize}
This proves that $u$ satisfies the definition of viscosity subsolution  for $f - \lambda  {A_\dagger} f = h$ via the sequence $\{\pi_\infty\}.$

\textit{Proof of step 1.}

As $u$ is upper semi-continuous for the weak topology and item \ref{item:Ham_conv_weaktop_equi_uniform_on_compacts_f} holds, we have 
\begin{equation} \label{eqn:suprema_limsup}
    \begin{aligned}
    \sup_{\rho\in E} \left\{ u(\rho) - f^\dagger(\rho) \right\} & \geq u(\pi_\infty) - f^\dagger(\pi_\infty) \\
    & \geq \limsup_{k\to+\infty} \{u(\pi_k) - f_k(\pi_k)\} \\
    & \stackrel{\eqref{eqn:quasi_coercive_Hamiltonian_seq1}}\geq \limsup_{k\to +\infty} \sup_{\rho\in E} \left\{ u(\rho) - f_k(\rho) \right\} - \varepsilon_k \\
    & = \limsup_{k \to+\infty} \sup_{\rho\in E} \left\{ u(\rho) - f_k(\rho) \right\}.
\end{aligned}
\end{equation}
On the other hand, let $(\rho_l)_{l\in \mathbb N}\in E$ be such that
\begin{equation*}
    u(\rho_l) - f^\dagger(\rho_l) \geq \sup_{\rho\in E} \left\{ u(\rho) - f^\dagger(\rho)\right\} - \varepsilon_l.
\end{equation*}
Due to item \ref{item:Ham_conv_weaktop_equi_pointwise},  for all $k\geq 1$ we find
\begin{align*}
    \sup_{\rho\in E} \left\{ u(\rho) - f^\dagger(\rho)\right\} & \leq   u(\rho_l) - f^\dagger(\rho_l)  + \varepsilon_l, \\
    & \leq \liminf_{k\to+\infty} \{u(\rho_l) - f_k(\rho_l)\} + \varepsilon_l \\
    & \leq \liminf_{k\to+\infty} \sup_{\rho\in E} \left\{ u(\rho) - f_k(\rho)\right\} +\varepsilon_l,
\end{align*}
so that
\begin{equation}\label{eqn:suprema_liminf}
    \sup_{\rho\in E} \left\{ u(\rho) - f^\dagger(\rho)\right\} \leq \liminf_{k\to+\infty} \sup_{\rho\in E} \left\{ u(\rho) - f_k(\rho)\right\}.
\end{equation}
Combining \eqref{eqn:suprema_limsup} with  \eqref{eqn:suprema_liminf}   we have
\begin{align*}
    \limsup_{k\to+\infty} \sup_{\rho\in E} \left\{ u(\rho) - f_k(\rho) \right\} & \leq \limsup_{k\to+\infty} \{u(\pi_k) - f_k(\pi_k)\} \\ 
    & \leq  u(\pi_\infty) - f^\dagger(\pi_\infty)\leq \sup_{\rho\in E} \left\{ u(\rho) - f^\dagger(\rho)\right\} \\
    & \leq \liminf_{k \to+\infty} \sup_{\rho\in E} \left\{ u(\rho) - f_k(\rho)\right\}.
\end{align*}

Therefore the above inequalities are equalities. Note also that,  due to \eqref{eqn:quasi_coercive_Hamiltonian_seq1}, we have 
\begin{equation*}
\liminf_{k\to+\infty} \{u(\pi_k) - f_k(\pi_k)\} \geq \lim_{k\to+\infty} \sup_{\rho \in E} \left\{u(\rho) - f_k(\rho)\right\} \geq \limsup_{k\to+\infty} \{u(\pi_k) - f_k(\pi_k)\}
\end{equation*}
and \eqref{eqn:push_visc_sol_optimum} is established.

\textit{Proof of step 2.}

To establish \eqref{eqn:convergence_optima_in_subsol}, it suffices to show $\liminf_{k \to +\infty} u(\pi_k) \geq u(\pi_\infty)$. We have
\begin{align*}
    \liminf_{k\to+\infty} u(\pi_k) & = \liminf_{k\to+\infty} \left\{u(\pi_k) - f_k(\pi_k) + f_k(\pi_k) \right\} \\
    & \geq \liminf_{k\to+\infty} \left\{u(\pi_k) - f_k(\pi_k) \right\} + \liminf_{k\to+\infty} f_k(\pi_k) \\
    & \stackrel{\text{Step 1}}{\geq} u(\pi_\infty) - f^\dagger(\pi_\infty) + f^\dagger(\pi_\infty) = u(\pi_\infty).
\end{align*}
where we used  \eqref{eqn:push_visc_sol_optimum} and \ref{item:Ham_conv_weaktop_equi_uniform_on_compacts_f} to go from line 2 to 3. \eqref{eqn:convergence_optima_in_testfunctions} follows  establishing similarly that $\limsup_{k\to+\infty} f_k(\pi_k) \leq f^\dagger(\pi_\infty)$.

\smallskip

\textit{Proof of step 3}.
Note that \eqref{eqn:convergence_optima_in_testfunctions} enables us to use \ref{item:Ham_conv_weaktop_equi_uniform_on_compacts_g} that, together with the weak continuity of $h$ and  \eqref{eqn:convergence_optima_in_subsol},  gives
\begin{align*}
    u(\pi_\infty) - \lambda g^\dagger(\pi_\infty) - h(\pi_\infty) & \leq \lim_{k\to+\infty} u(\pi_k)  + \liminf_{k\to+\infty} \left( - \lambda g_k(\pi_k) \right) - \lim_{k\to+\infty} h(\pi_k) \\
    & \leq \liminf_{k\to+\infty} \left( u(\pi_k) - \lambda g_k(\pi_k) - h(\pi_k) \right) \\
    & \stackrel{\eqref{eqn:quasi_coercive_Hamiltonian_seq2}}{\leq} 0.
\end{align*}

Since $(f^\dagger,g^\dagger) \in A_\dagger$ has been chosen arbitrarily, we conclude that $u$ is a viscosity subsolution to $f - \lambda A_\dagger f = h$.
\end{proof}

\begin{proof}[Proof of Lemma \ref{lemma: regularization}]
	We only prove the first statement. Let $u$ be a subsolution to $f - \lambda A_\dagger f = h$. Let $(f^\dagger,g^\dagger) \in A_\dagger$. As $u$ is a subsolution, there exists a sequence $(\pi_n)_{n\in\mathbb N}\in E$ such that
	\begin{equation} \label{eqn:viscosity_push_topology_1}
		\lim_{n\to+\infty} u(\pi_n) - f^\dagger(\pi_n) = \sup_{\rho\in E} \left\{u(\rho) - f^\dagger(\rho) \right\}
	\end{equation}
	and
	\begin{equation} \label{eqn:viscosity_push_topology_2}
		\limsup_{n\to+\infty} u(\pi_n) - \lambda g^\dagger(\pi_n) - h(\pi_n) \leq 0.
	\end{equation}
	Let $u^*$ be the weak upper semi-continuous regularization of $u$. Recall that  $u^*$ is the smallest weak upper semi-continuous function $\hat{u}$ such that $\hat{u} \geq u$. We aim to prove the same statements for $u^*$, which establishes the claim.
	
	\smallskip
	
	As $u^* \geq u$, we have
	\begin{equation} \label{eqn:viscosity_subsolution_topology_push_inequality1}
		\sup_{\rho \in E} \left\{u^*(\rho) - f^\dagger(\rho)  \right\} \geq \sup_{\rho\in E} \left\{u(\rho) - f^\dagger(\rho)  \right\}.
	\end{equation}
	As $u^*$ is the weak upper semi-continuous regularization of $u$, and $f^\dagger$ is weakly lower semi-continuous, it follows that $u^* - f^\dagger$ is the weak upper semi-continuous regularization of $u - f^\dagger$. 

 The constant function $\pi\mapsto \sup_{\rho\in E} \left\{u(\rho) - f^\dagger(\rho)  \right\}$ is weakly upper semi-continuous. In addition, it is dominating the function $u-f^\dagger$, it thus follows that for all $\pi \in E$ that

     \begin{equation*}
	 u^*(\pi) - f^\dagger(\pi) \leq \sup_{\rho\in E} \left\{u(\rho) - f^\dagger(\rho)  \right\}.
 \end{equation*}
 
 Therefore we can conclude that
	\begin{equation*}
		\sup_{\rho\in E} \left\{u^*(\rho) - f^\dagger(\rho)  \right\} =\sup_{\rho\in E} \left\{u(\rho) - f^\dagger(\rho)  \right\}.
	\end{equation*}
	
	We derive then that
	\begin{multline*}
		{\liminf}_{n\to+\infty} u^*(\pi_n) - f^\dagger(\pi_n) \geq \lim_{n\to+\infty} u(\pi_n) - f^\dagger(\pi_n) = \sup_{\rho\in E} \left\{u(\rho) - f^\dagger(\rho) \right\} \\
		=\sup_{\rho\in E} \left\{u^*(\rho) - f^\dagger(\rho) \right\} \geq \limsup_{n\to +\infty} u^*(\pi_n) - f^\dagger(\pi_n)
	\end{multline*}
	which implies in particular that
	\begin{equation*}
		\lim_{n\to+\infty} u(\pi_n) -u^*(\pi_n) = 0.
	\end{equation*}
	Applying this to \eqref{eqn:viscosity_push_topology_1} and \eqref{eqn:viscosity_push_topology_2}, we obtain
	\begin{gather*}
		\lim_{n\to+\infty} u^*(\pi_n) - f^\dagger(\pi_n) = \sup_{\rho\in E} \left\{u^*(\rho) - f^\dagger(\rho)  \right\} \\
		{\limsup_{n \rightarrow \infty} } u^*(\pi_n) - \lambda g^\dagger(\pi_n) - h(\pi_n) \leq 0
	\end{gather*}
	which establishes the claim.

\end{proof}

\section{Proof of Theorem \ref{theorem:carry_over_solutions}}\label{proof:carry_over_solutions}

In this section, we carry out the steps presented at the beginning of Section \ref{section:preparations_proof}. 

Before doing so, however, we will split the proof of Theorem \ref{theorem:carry_over_solutions} into two parts, namely one proof for the case that $\kappa \neq 0$, and one proof for $\kappa = 0$. This split is inspired by the consequences of the properties of the gradient flow in the context of the evolutionary variational inequality. Inded it is shown in items \ref{item: EVI Bakry} and \ref{item: EVI distance bound} of Lemma \ref{lem: EVI implies boundedness} that uniform bounds on $\cE(\mu(t))$ and $e^{\hat{\kappa}t}d_\varepsilon(\pi,\mu(t))$, along a gradient flow $\mu(t)$, can be given only if $\kappa \neq 0$.

We will, therefore, start in Section \ref{section:approximation_kappa_neq_0} proving the result in the case that $\kappa \neq 0$, using the strategy that was outlined at the start of Section \ref{section:preparations_proof}.  The case $\kappa = 0$ will be treated in Section \ref{section:approximation_kappa_is_0} by using the fact that if the gradient flow satisfies \eqref{item:ass_EVI} with $\kappa = 0$, then it also satisfies \eqref{item:ass_EVI}  for any $\kappa < 0$ and a final approximation $\kappa \uparrow 0$ can extend the result.

We believe that all the steps of the proof of Theorem \ref{theorem:carry_over_solutions}, as carried out in Section \ref{section:approximation_kappa_neq_0} in the context $\kappa \neq 0$, can be adapted to $\kappa = 0$ using the non uniform bounds of Lemma \ref{lem: EVI implies boundedness} \ref{item: EVI distance bound}, at the cost of greater technical difficulty. However, we think that the non-direct method employed below leads to proofs that are easier than the ones for the direct method.

\subsection{The proof in the case that \texorpdfstring{$\kappa \neq 0$}{kappa not 0}} \label{section:approximation_kappa_neq_0}

\subsubsection{The weak regularization of solutions for $H_\dagger$ and $H_\ddagger$} \label{section:H1_weakregularize}

Our first result is on the weak regularization of sub and supersolutions.

\begin{lemma} \label{lemma:regularization_1}
    {Let $h \in C_b(E)$ and $\lambda > 0$.}

	Let $u$ be a viscosity subsolution to $f - \lambda H_\dagger f = h$. Then the weak upper semi-continuous regularization $u^*$ of $u$ is also a viscosity subsolution to $f - \lambda H_{\dagger} f = h$.
	
	Let $v$ be a viscosity supersolution to $f - \lambda H_{\ddagger} f = h$. Then the weak lower semi-continuous regularization $v_*$ is also a viscosity supersolution to $f - \lambda H_{\ddagger} f = h$ .
\end{lemma}

\begin{proof}
    The result is immediate by Lemma \ref{lemma: regularization} and Assumption \ref{assumption:weak_topology}.
\end{proof}

\subsubsection{From $H_{1,\dagger}$ to $H_{2,\dagger}$} \label{section:H1toH2}

For our first real step relating two sets of Hamiltonians, we consider the corresponding first step in the approximation of the Tataru distance. In particular, our starting point is the log of the Riemann sum approximation of the integral on the right-hand side of \eqref{eqn:smoothed_version_Tataru} for fixed $m$. To define this object in our definition of $H_{2,\dagger}$, we start out by introducing a smooth approximation of the square root function $r \mapsto \sqrt{2r}$ and the corresponding version of the Tataru distance.

\begin{definition}
    	For sufficiently small $\varepsilon > 0$ define 
	\begin{equation} \label{eqn:def_approx_of_squareroot}
	\psi_\varepsilon(r) := \left(\sqrt{2\varepsilon} + \frac{r-\varepsilon}{\sqrt{2\varepsilon}} - \frac{(r-\varepsilon)^2}{2 (2\varepsilon)^{3/2}}\right) \bONE_{\{0 \leq r \leq \varepsilon\}} + \sqrt{2r} \bONE_{\{r \geq \varepsilon\}}.
	\end{equation}
    
	The modified distance $d_{\varepsilon}$ and the modified Tataru distance $d_{T,\varepsilon}$ are defined for all $\pi,\mu\in E$ as
	\begin{align}
		d_{\varepsilon}(\pi,\mu)&:=\psi_\varepsilon\left(\frac{1}{2}d^2(\pi,\mu)\right),\label{eq: d_epsilon}\\
	    h^{\varepsilon}_{\pi,\mu}(t) &: = e^{\hat\kappa t}d_{\varepsilon}(\pi,\mu(t)),
	\label{eqn:definition_h}\\
    d_{T,\varepsilon}(\pi,\mu) &:= \inf_{t\geq 0}\left\{t + h^{\varepsilon}_{\pi}(t) \right\}. \label{eqn:def_smoothed_Tataru_distance}
\end{align}
\end{definition}

{Of fundamental importance in this definition is the fact that $r\mapsto \psi_\varepsilon\left(\frac{1}{2} r^2\right)$ is an approximation of the identity $r \mapsto r$, in such a way that it is smooth in its input $\frac{1}{2}r^2$.}

The study of the properties of the functions $\psi_\varepsilon(r)$, $d_\varepsilon$ and $d_{T,\varepsilon}$ are postponed to Section \ref{appendixTataru_modified} below. Let us underline, however, that  the two key properties are that $r \mapsto \psi_{\varepsilon}\left(\frac{1}{2}r^2\right)$ is strictly increasing and approximates the identity, and that $\psi_\varepsilon$ is twice continuously differentiable.

In addition to the above approximation, we consider, for any integer $m,n\geq 1 $, the approximation $\lambda_{m,n}$ of the exponential distribution of parameter $m$.
	\begin{equation*}
	    \lambda_{m,n}(\dd t) := c_{m,n}  \sum_{i=1}^{n^2} e^{-m \frac{i}{n}} \delta_{\frac{i}{n}}(\dd t).
	\end{equation*}
	The pre-factor $c_{m,n}$ is a normalizing constant, whose explicit form is uninteresting. 

To simplify the notation, in the following definition and in the rest of the paper, we only explicitly write the dependence of the functions in our definitions on the parameters $\varepsilon,m,n$ since they are the ones that vary in the approximation procedures. All the others parameters ($a,b,c,\rho$ and $\mu$) are omitted. For example, in the next definition we write $f^{2,\dagger}_{\varepsilon,m,n}(\cdot)$ and $h_{\pi}^\varepsilon(\cdot)$ instead of $f^{2,\dagger}_{\varepsilon,m,n,\rho,\mu, a, b, c }(\cdot)$ and $h^{\varepsilon}_{\pi,\mu}(\cdot)$ respectively.

\begin{definition}
  Let $\psi_\varepsilon$ be defined as in \eqref{eqn:def_approx_of_squareroot}, $\rho,\mu\in E$ such that $I(\rho) + I(\mu) < \infty$, $a,b,\varepsilon >0$, $c \in \bR$ and  $m,n \geq 1$. 	For all $\pi\in E,t>0$ we define
	\begin{equation} \label{eqn:definition_lambda}
	    \Lambda_{\varepsilon,m,n}(\pi) := \int_0^\infty  e^{-m h^{\varepsilon}_{\pi}(t)}\lambda_{m+1,n}(\dd t).
	  \end{equation}
   with $h^\varepsilon_\pi$ as in \eqref{eqn:definition_h}. Consider the test function $f^{2,\dagger}_{\varepsilon,m,n}$ given by
	\begin{equation}\label{eq:f2dag}
	f^{2,\dagger}_{\varepsilon,m,n}(\pi) = \frac{1}{2} a d^2(\pi,\rho) + b\left( - \frac{1}{m} \log \Lambda_{\varepsilon,m,n}(\pi)\right) +c.
	\end{equation}
	Moreover, we define
	\begin{equation}
	\begin{split}
	 g^{2,\dagger}_{\varepsilon,m,n}(\pi) &:=  a \left[\cE(\rho) - \cE(\pi)\right] - a \frac{\kappa}{2}d^2(\pi,\rho) + \frac{a^2}{2} d^2 (\pi,\rho) + ab d(\pi,\rho)   + \frac{b^2 }{2}\\
	& \qquad + b \Lambda_{\varepsilon,m,n}^{-1}(\pi) \int_0^\infty e^{-m h^{\varepsilon}_{\pi}(t)} \psi_\varepsilon'\left(\frac{1}{2}d^2\left(\pi,\mu\left(t \right)\right)\right) e^{\hat{\kappa}t} \left[\cE\left(\mu\left(t\right)\right) - \cE(\pi)\right] \lambda_{m+1,n}(\dd t)    \\
	& \qquad -  b \frac{\hat{\kappa}}{2} \Lambda_{\varepsilon,m,n}^{-1}(\pi) \int_0^\infty  e^{-m h^{\varepsilon}_{\pi}(t)} \left(\frac{1}{m} \vee h^{\varepsilon}_{\pi}(t) \right)  \lambda_{m+1 ,n}(\dd t). \label{eq:g2dag}
    \end{split}
	\end{equation}
	Then $H_{2,\dagger}$ is defined as the operator given by all pairs
	\begin{equation*}
	    H_{2,\dagger} := \left\{(f^{2,\dagger}_{\varepsilon,m,n},g^{2,\dagger}_{\varepsilon,m,n}) \, \middle| \, \rho,\mu: \, I(\rho) + I(\mu) < \infty, a,b,\varepsilon>0, c \in \bR, m,n \geq 1\right\}.
	\end{equation*}

    For $\gamma,\pi\in E$ such that $ I(\gamma)+ I(\pi) < \infty$, $a,b,\varepsilon >0$, $c \in \bR$ and  $m,n \geq 1$, consider the test function $f^{2,\ddagger}_{\varepsilon,m,n}$ given by 
	\begin{equation*}\label{eq:f2ddag}
	f^{2,\ddagger}_{\varepsilon,m,n}(\mu) = - \frac{1}{2} a d^2(\gamma,\mu) - b\left( - \frac{1}{m} \log \Lambda_{\varepsilon,m,n}(\mu)\right) +c,
	\end{equation*}
 where $\Lambda_{\varepsilon,m,n}(\mu)$ is defined as in \eqref{eqn:definition_lambda}, using $h^{\varepsilon}_{\mu}(t)$,  inverting the role of $\pi$ and $\mu$.
	Moreover, we define 
	\begin{equation*}
	\begin{split}
	 g^{2,\ddagger}_{\varepsilon,m,n}(\mu) &:=  a \left[\cE(\mu) - \cE(\gamma)\right] + a \frac{\kappa}{2}d^2(\gamma,\mu) + \frac{a^2}{2}d^2 (\gamma,\mu) - ab d(\gamma,\mu) {-\frac{1}{2}b^2} \\
	& \qquad - b \Lambda_{\varepsilon,m,n}^{-1}(\mu) \int_0^\infty e^{-m h^{\varepsilon}_{\mu}(t)} \psi_\varepsilon'\left(\frac{1}{2}d^2\left(\mu,\pi\left(t \right)\right)\right) e^{\hat{\kappa}t} \left[\cE\left(\pi\left(t\right)\right) - \cE(\mu)\right] \lambda_{m+1,n}(\dd t)    \\
	& \qquad +  b \frac{\hat{\kappa}}{2} \Lambda_{\varepsilon,m,n}^{-1}(\mu) \int_0^\infty  e^{-m h^{\varepsilon}_{\mu}(t)} \left(\frac{1}{m} \vee h^{\varepsilon}_{\mu}(t) \right)  \lambda_{m+1 ,n}(\dd t). \label{eq:g2ddag}
    \end{split}
	\end{equation*}
	Then $H_{2,\ddagger}$ is defined as the operator given by all pairs
	\begin{equation*}
	    H_{2,\dagger} := \left\{(f^{2,\ddagger}_{\varepsilon,m,n},g^{2,\ddagger}_{\varepsilon,m,n}) \, \middle| \, \gamma,\pi: \,  I(\gamma)+I(\pi) < \infty, a,b,\varepsilon>0, c \in \bR, m,n \geq 1\right\}.
	\end{equation*}

\end{definition}

Note that the terms in $g^{2,\dagger}_{\varepsilon,m,n}$ and $g^{2,\ddagger}_{\varepsilon,m,n}$, are, up to the terms in the final 2 lines equal to those of \eqref{eqn:H_tataru_sub} and \eqref{eqn:H_tataru_super}. Thus, once we have established the next lemma, our main focus will be on showing that we are approximating $b > 0$ times the Tataru function in the right way, and that the time derivative of the gradient flow along this approximation is bounded above by $b$.

\begin{remark}
    This remark is not important on first reading. Lemma \ref{lemma:push_1_2} would hold without the factor $\vee \frac{1}{m}$ in the final line of the definitions of $g^{2,\dagger}_{\varepsilon,m,n}$ and $g^{2,\ddagger}_{\varepsilon,m,n}$. This factor is added for later purposes: the map $s \mapsto \exp(-ms)(m^{-1} \vee s)$ is decreasing, whereas $s \mapsto \exp(-ms)s$ is not. This will aid us in the step from 2 to 3. Note, however, that the relaxation is only a minor one: in the step from 3 to 4, we send $m \rightarrow \infty$, effectively removing this this factor.
    
\end{remark}

\begin{lemma} \label{lemma:push_1_2}
    {Let $h \in C_b(E)$ and $\lambda > 0$.}

	Every viscosity subsolution to $f - \lambda H_{1,\dagger} f = h$ is a viscosity subsolution to $f - \lambda H_{2,\dagger} f = h$.  
	
	Every viscosity supersolution to $f - \lambda H_{1,\ddagger} f = h$ is a viscosity supersolution to $f - \lambda H_{2,\ddagger} f = h$. 
\end{lemma}

\begin{proof}

We only prove the statement for $H_{1,\dagger}$ and $H_{2,\dagger}$ as the analogous statement for supersolutions follows similarly.

     For a fixed choice of admissible parameters $\rho,\mu,\varepsilon,a,b,m,n,c$  let $f^{2,\dagger}_{\varepsilon,m,n}$ be given by \eqref{eq:f2dag}. Note that the term
    \begin{equation*}
        \theta(\pi) :=b\left(- \frac{1}{m} \log  \Lambda_{\varepsilon,m,n}(\pi) \right) +c
    \end{equation*}
    rewrites, using the definition of $\Lambda_{\varepsilon,m,n}$, to
    \begin{equation}\label{eqn:varphi}
     \theta(\pi) := \varphi\left(\frac{1}{2}d^2\left(\pi,\mu\left(\frac{1}{n}\right) \right), \dots, \frac{1}{2}d^2\left(\pi,\mu(n)\right)\right),
     \end{equation}
     where
     \begin{equation*}
         \varphi(r_1,\dots,r_{n^2}) := b\left( - \frac{1}{m} \log \left( c_{m,n}   \sum_{i=1}^{n^2} e^{-m \left(e^{\hat{\kappa}t} \psi_\varepsilon\left(r_i \right) \right) -(m+1) \frac{i}{n}}\right)\right) +c.
     \end{equation*}
    Using that $\psi_\varepsilon$ is smooth with positive derivative, we find that our test function is a smooth function that is applied to a finite number of metric squared type objects and satisfies $\partial_i \phi > 0$. In other words, there exists $g^{1,\dagger}$ such that $(f^{2,\dagger}_{\varepsilon,m,n},g^{1,\dagger}) \in H_{1,\dagger}$. 

	In particular, using \eqref{eq:reg_g_1dag} we obtain that we can choose  $g^{1,\dagger}$ as follows
	\begin{align}
	g^{1,\dagger}(\pi) & = a \left[\cE(\rho) - \cE(\pi)\right] - a \frac{\kappa}{2}d^2(\pi,\rho) + \frac{1}{2} a^2 d^2(\pi,\rho)\label{eqn:worked_out_Hdagger} \\
    & \qquad + b \Lambda_{\varepsilon,m,n}^{-1}(\pi) \int_0^\infty  e^{-m h^\varepsilon_\pi(t)} \psi_\varepsilon'\left(\frac{1}{2}d^2\left(\pi,\mu(t)\right)\right)e^{\hat{\kappa}t} \times \notag \\
	& \hspace{3.5cm} \left[\cE\left(\mu(t)\right) - \cE(\pi) - \frac{\kappa}{2} d^2\left(\pi,\mu(t) \right)\right] \lambda_{m+1,n}(\dd t) \notag \\
	& \qquad + a b d(\pi,\rho) \Lambda_{\varepsilon,m,n}^{-1}(\pi) \int_0^\infty e^{-m h^\varepsilon_\pi(t)}    \psi_\varepsilon'\left(\frac{1}{2}d^2\left(\pi,\mu(t)\right)\right)e^{\hat{\kappa}t} d\left(\pi,\mu(t)\right) \lambda_{m+1,n}(\dd t) \notag\\
	& \qquad + \frac{1}{2}  b^2 \left[\Lambda_{\varepsilon,m,n}^{-1}(\pi) \int_0^\infty e^{-m h^{\varepsilon}_\pi(t)}   \psi_\varepsilon'\left(\frac{1}{2}d^2\left(\pi,\mu(t)\right)\right) e^{\hat{\kappa}t}  d\left(\pi,\mu(t)\right) \lambda_{m+1,n}(\dd t) \right]^2. \notag
    \end{align}
    To obtain the above expression we used the fact that, for $\theta$  as in \eqref{eqn:varphi}  and $i\in\{1, \dots, n^2\}$, we have
	\begin{align*}
	& \partial_i \varphi\left(\frac{1}{2} d^2\left(\pi,\mu\left(\frac 1 n \right)\right),\dots, \frac{1}{2} d^2\left(\pi,\mu\left(\frac {n^2} n \right)\right) \right) \\ 
    & = b \left(-\frac 1 m \right) \Lambda_{\varepsilon,m,n}^{-1}(\pi)  c_{m+1,n}  (-m)e^{\hat{\kappa}\frac i n}e^{-m h^\varepsilon_\pi(\frac i n) - (m+1) \frac{i}{n}} \psi_\varepsilon'\left(\frac{1}{2}d^2\left(\pi,\mu\left(\frac i n\right)\right)\right).
	\end{align*}
	To finish the proof, it suffices to show that $g^{1,\dagger}(\pi) \leq g^{2,\dagger}_{\varepsilon,m,n}(\pi)$ for all $\pi\in E$. For the final two lines of \eqref{eqn:worked_out_Hdagger}, note that by Lemma \ref{lemma:approximate_square_root} \ref{item:lemma_approx_sqrt_bound_on_product} and recalling that $\hat{\kappa}{\leq}0$, we have
	\begin{equation*}
	    d(\pi,\mu(t)) \psi_\varepsilon'\left(\frac{1}{2}d^2\left(\pi,\mu(t)\right)\right) e^{\hat{\kappa}t} \leq 1
	\end{equation*}
    thus reducing the integrals to $\Lambda_{\varepsilon,m,n}(\pi)$, so that 
	\begin{align*}
	g^{1,\dagger}(\pi) & \leq  a \left[\cE(\rho) - \cE(\pi)\right] - a \frac{\kappa}{2}d^2(\pi,\rho) + \frac{1}{2} a^2 d^2(\pi,\rho) + a b d(\pi,\rho)   + \frac{1}{2} b^2 \\
    & \qquad + b \Lambda_{\varepsilon,m,n}^{-1}(\pi) \int_0^\infty  e^{-m e^{\hat{\kappa}t} d_\varepsilon(\pi,\mu(t))} \psi_\varepsilon'\left(\frac{1}{2}d^2\left(\pi,\mu(t)\right)\right)e^{\hat{\kappa}t} \times \notag \\
	& \hspace{3.5cm} \left[\cE\left(\mu(t)\right) - \cE(\pi) - \frac{\kappa}{2} d^2\left(\pi,\mu(t) \right)\right] \lambda_{m+1,n}(\dd t). \notag 
	\end{align*}
	We focus now on the final integral term and perform a final estimate on the part that involves $\frac{\kappa}{2} d^2(\pi,\mu(t))$. First of all, $\hat{\kappa} \leq \kappa$, so that we can replace $\kappa$ in front of $d^2$ by $\hat{\kappa}$ thus fixing the sign of the integral term. Note furthermore that  by Lemma~\ref{lemma:approximate_square_root} (c), we have 
	\begin{equation*}
	    e^{\hat{\kappa}t} d^2(\pi,\mu(t)) \psi_\varepsilon'\left(\frac{1}{2}d^2\left(\pi,\mu(t)\right)\right) \leq e^{\hat{\kappa}t}d(\pi,\mu(t)) \leq \frac{1}{m} \vee \left( e^{\hat{\kappa}t}d_\varepsilon(\pi,\mu(t))\right),
	\end{equation*}
	which yields
	\begin{align*}
	& g^{1,\dagger}(\pi) \\
	& \leq  a \left[\cE(\rho) - \cE(\pi)\right] - a \frac{\kappa}{2}d^2(\pi,\rho) + \frac{1}{2} a^2 d^2(\pi,\rho) + a b d(\pi,\rho)   + \frac{1}{2} b^2 \\
    & \qquad + b \Lambda_{\varepsilon,m,n}^{-1}(\pi) \int_0^\infty e^{-m e^{\hat{\kappa}t} d_\varepsilon(\pi,\mu(t))} \psi_\varepsilon'\left(\frac{1}{2}d^2\left(\pi,\mu(t)\right)\right)e^{\hat{\kappa}t}  \left[\cE\left(\mu(t)\right) - \cE(\pi) \right] \lambda_{m+1,n}(\dd t) \notag \\
	& \qquad - b \frac{\hat{\kappa}}{2}  \Lambda_{\varepsilon,m,n}^{-1}(\pi) \int_0^\infty  e^{-m e^{\hat{\kappa}t} d_\varepsilon(\pi,\mu(t))} \left( \frac{1}{m} \vee e^{\hat{\kappa}t} d_\varepsilon(\pi,\mu(t)) \right) \lambda_{m+1,n}(\dd t).
	\end{align*}
	We conclude that $g^{1,\dagger}(\pi) \leq g^{2,\dagger}_{\varepsilon,m,n}(\pi)$.
\end{proof}

\subsubsection{Approximating the integral: from $H_{2,\dagger}$ to $H_{3,\dagger}$}

In this section, we will make explicit that the Riemann sum featuring in the approximation of the Tataru distance nicely converges to the corresponding integral. Recall that the measure $\lambda_{m,n}$ appearing in the Riemann sum was a discrete approximation of the exponential measure with mean $m^{-1}$. We denote this exponential measure by $\lambda_m$:
	\begin{equation*}
	    \lambda_{m}(\dd t) := \bONE_{\{t \geq 0\}} m e^{-m t} \dd t.
	\end{equation*}

We next give the definitions of $H_{3,\dagger}$ and $H_{3,\ddagger}$. Note that the only change  is the replacement of the Riemann sum by an integral.

\begin{definition} 
For given $\rho,\mu\in E$ such that $I(\rho) + I(\mu) < \infty$, $a,b,\varepsilon >0$, $c \in \bR$ and  $m\geq 1$ we define
	\begin{equation*}
	    \Lambda_{\varepsilon,m}(\pi) := \int_0^\infty  e^{-m h^{\varepsilon}_{\pi}(t)}\lambda_{m+1}(\dd t)
	\end{equation*}
	{ with $h_\pi^\varepsilon$ as in \eqref{eqn:definition_h},}
	\begin{equation}\label{def: f3}
	f^{3,\dagger}_{\varepsilon,m}(\pi) : = \frac{1}{2} a d^2(\pi,\rho) - \frac{b}{m} \log \Lambda_{\varepsilon,m}(\pi) +c
	\end{equation}
	and
	\begin{equation}\label{def: g3}
	\begin{split}
	 g^{3,\dagger}_{\varepsilon,m}(\pi)&:=  a \left[\cE(\rho) - \cE(\pi)\right] - a \frac{\kappa}{2}d^2(\pi,\rho) + \frac{ a^2}{2} d^2(\pi,\rho) + a b d(\pi,\rho)   + \frac{b^2}{2} \\
	&  \qquad + b \Lambda_{\varepsilon,m}^{-1}(\pi) \int_0^\infty e^{-m h^{\varepsilon}_{\pi}(t)} \psi_\varepsilon'\left(\frac{1}{2}d^2\left(\pi,\mu\left(t \right)\right)\right) e^{\hat{\kappa}t} \left[\cE\left(\mu\left(t\right)\right) - \cE(\pi)\right] \lambda_{m+1}(\dd t)    \\
	 &\qquad -  b \frac{\hat{\kappa}}{2} \Lambda_{\varepsilon,m}^{-1}(\pi) \int_0^\infty  e^{-mh^{\varepsilon}_{\pi}(t)} \left( \frac{1}{m} \vee h^{\varepsilon}_{\pi}(t)  \right)   \lambda_{m+1}(\dd t).
	\end{split}
	\end{equation}
    Finally, $H_{3,\dagger}$ is the operator given by all pairs
	\begin{equation*}
	    H_{3,\dagger} := \left\{(f^{3,\dagger}_{\varepsilon,m},g^{3,\dagger}_{\varepsilon,m}) \, \middle| \, \rho,\mu: \, I(\rho) + I(\mu) < \infty, a,b,\varepsilon >0, c \in \bR, m \geq 1 \right\}.
	\end{equation*}
	For given $\gamma,\pi\in E$ such that $ I(\gamma)+ I(\pi) < \infty$, $a,b,\varepsilon >0$, $c \in \bR$ and  $m\geq 1$ consider
	\begin{equation*}
	f^{3,\ddagger}_{\varepsilon,m}(\mu) : = -\frac{a}{2} d^2(\gamma,\mu) + \frac{b}{m} \log \Lambda_{\varepsilon,m}(\mu) +c,
	\end{equation*}
	where $\Lambda_{\varepsilon,m}(\mu) $ is defined as $\Lambda_{\varepsilon,m}(\pi) $ inverting the role of $\mu$ and $\pi$, and 
	\begin{equation*}
	\begin{split}
	 g^{3,\ddagger}_{\varepsilon,m}(\mu)&:=  a \left[\cE(\mu) - \cE(\gamma)\right] + a \frac{\kappa}{2}d^2(\mu,\gamma) + \frac{ a^2}{2} d^2(\mu,\gamma) - a b d(\mu,\gamma)   {- \frac{1}{2}b^2} \\
	&  \qquad - b \Lambda_{\varepsilon,m}^{-1}(\mu) \int_0^\infty e^{-m h^{\varepsilon}_{\mu}(t)} \psi_\varepsilon'\left(\frac{1}{2}d^2\left(\mu,\pi\left(t \right)\right)\right) e^{\hat{\kappa}t} \left[\cE\left(\pi\left(t\right)\right) - \cE(\mu)\right] \lambda_{m+1}(\dd t)    \\
	 & \qquad +  b \frac{\hat{\kappa}}{2} \Lambda_{\varepsilon,m}^{-1}(\mu) \int_0^\infty  e^{-mh^{\varepsilon}_{\mu}(t)} \left( \frac{1}{m} \vee h^{\varepsilon}_{\mu}(t)  \right)   \lambda_{m+1}(\dd t).
	\end{split}
	\end{equation*}
    Finally, $H_{3,\ddagger}$ is the operator given by all pairs
	\begin{equation*}
	    H_{3,\ddagger} := \left\{(f^{3,\ddagger}_{\varepsilon,m},g^{3,\ddagger}_{\varepsilon,m}) \, \middle| \, \gamma,\pi: \,  I(\gamma)+I(\pi) < \infty, a,b,\varepsilon >0, c \in \bR, m \geq 1 \right\}.
	\end{equation*}
\end{definition}

\begin{theorem} \label{thm: 2to3}
Let $\lambda > 0$ and let $h \in C_b(E)$ be continuous for the weak topology.

	Every weakly upper semi-continuous viscosity subsolution to $f - \lambda H_{2,\dagger} f = h$ is also a viscosity subsolution to $f - \lambda H_{3,\dagger} f = h$.
	
	Every weakly lower semi-continuous viscosity supersolution to $f - \lambda H_{2,\ddagger} f = h$ is also a viscosity supersolution to $f - \lambda H_{3,\ddagger} f = h$.
\end{theorem}

\begin{proof}
    We only prove the first claim.
    We will argue on the basis of Proposition \ref{proposition:Hamiltonian_convergence_pseudo_coercive}. Let $(f^{3,\dagger}_{\varepsilon,m},g^{3,\dagger}_{\varepsilon,m}) \in H_{3,\dagger}$. Thus, there are $\rho,\mu$ such that $I(\rho) + I(\mu)  < \infty$, $a,b, \varepsilon > 0$ and $c \in \bR$, $m \in \{1,2,\dots\}$ such that 
	\begin{equation*}
	f^{3,\dagger}_{\varepsilon,m}(\pi) = \frac{1}{2}a d^2(\pi,\rho) - \frac{b}{m} \log \Lambda_{\varepsilon,m}(\pi,\mu) + c.
	\end{equation*}
 	For the same $\rho,\mu,a,b,c,\varepsilon, m$ and for all $n\geq 1$, we choose now $(f^{2,\dagger}_{\varepsilon,m,n},g^{2,\dagger}_{\varepsilon,m,n}) \in H_{2,\dagger}$ as in \eqref{eq:f2dag} and proceed to verify the hypothesis of Proposition \ref{proposition:Hamiltonian_convergence_pseudo_coercive}  with $n \rightarrow \infty$ as a running variable.
 	We do so in three steps; in the first step we verify (a), in the second we verify \ref{item:Ham_conv_weaktop_equi_pointwise} and \ref{item:Ham_conv_weaktop_equi_uniform_on_compacts_f} and we conclude by verifying \ref{item:Ham_conv_weaktop_equi_uniform_on_compacts_g}.
 	\begin{itemize}
 	    \item \underline{Step 1: Verification of (a).} It is easily seen that
 	    \begin{equation}\label{eq:H2_to_H3_1}
 	   \forall \pi,n, \quad f^{2,\dagger}_{\varepsilon,m,n}(\pi) \geq \frac{1}{2}a d^2(\pi,\rho) + c.
 	    \end{equation}
 	     Thus, it remains to verify condition \eqref{eq:weak_compactness_via_levelsets 2}.  
       
 To this aim, it is enough  to check that the contribution of the two integral terms appearing in the definition of $g^{2,\dagger}_{\varepsilon,m,n}$ is bounded above on $B_R(\rho)$ by a function of $\cE(\pi)$ that grows less fast than $a\cE(\pi)$, for any $R>0$. To handle the first term, it suffices to observe that the function $t\mapsto \exp(\hat\kappa t)\psi'_{\varepsilon}(\frac{1}{2}d^2(\pi,\mu(t)))$ is non-negative and upper bounded due to Lemma \ref{lemma:approximate_square_root} and that $t\mapsto \cE(\mu(t))-\cE(\pi)$ is upper bounded on $B_{R}(\rho)$ thanks to \eqref{eq: lower bound on energy} and \eqref{eq_energy_identity}. The second integral term can be bounded in terms of $\sqrt{\cE(\pi)}$ using \eqref{eq: EVI implies boundedness 6} from Lemma \ref{lem: EVI implies boundedness}, and the desired conclusion follows.

     \item \underline{Step 2: Verification of (b) and (c).} By the definition of the test functions, (b) comes down to the convergence of $\Lambda_{\varepsilon, m,n}(\pi)$ towards $\Lambda_{\varepsilon,m}(\pi)$ as $n\rightarrow +\infty$, which is a consequence of the  weak convergence of $\lambda_{m+1,n}$ towards $\lambda_{m+1}$ and the continuity of $t\mapsto \mu(t)$, see Lemma \ref{lem: EVI implies boundedness} \ref{item: EVI stability}. In order to establish item (c), let us note that it is enough to  prove  
     \begin{equation}
     \label{eq: 2to3 2}\limsup_{n\rightarrow+\infty} \Lambda_{\varepsilon,m,n}(\pi_n) \leq \Lambda_{\varepsilon,m}(\pi_{\infty}) 
    \end{equation}
      along any weakly converging subsequence $\pi_n\rightarrow\pi_{\infty}$ in  $K^\rho_{\bar c,\bar d}$for any $\bar c,\bar d\in \R$, see \eqref{def: K^rho}. To do so, consider a subsequence weakly converging  to $\pi_\infty$  and observe that, since $d(\cdot,\cdot)$ is weakly lower semi-continuous, both $d_\varepsilon(\cdot,\cdot)$ and $h^{\varepsilon}_{(\cdot)}(\cdot)$ are weakly lower semi-continuous.  Let $t_\infty\in[0,+\infty)$, then for any $t_n \rightarrow t_{\infty}$, we have 
      	        \begin{equation*}
               \limsup_{n\rightarrow+\infty} \exp(-m h^{\varepsilon}_{\pi_n}(t_n)) \leq \exp(-mh^{\varepsilon}_{\pi_{\infty}}(t_{\infty})).   \end{equation*}
      	         But then, recalling the definition of $\Lambda_{\varepsilon,m,n}$, we observe that they are uniformly upper bounded
                and we can deduce \eqref{eq: 2to3 2} from Lemma \ref{lem:weak conv usc}. 
     	      \item \underline{Step 3: Verification of (d).} For any $\bar c,\bar d\in \R$, consider a weakly converging sequence $\pi_n\rightarrow \pi_{\infty}$ such that $(\pi_n)_{n\geq1}\subseteq K^\rho_{\bar c,\bar d}$   and
     	       \begin{equation}\label{eq: 2to3 11} 
     	       \lim_{n \to +\infty} f^{2,\dagger}_{\varepsilon,m,n}(\pi_n) = f^{3,\dagger}_{\varepsilon,m}(\pi_{\infty}).
     	       \end{equation}
     	       We need to show that 
     	       \begin{equation}\label{eq: 2to3 12}
     	       \limsup_{n \to +\infty} g^{2,\dagger}_{\varepsilon,m,n}(\pi_n) \leq g^{3,\dagger}_{\varepsilon,m}(\pi_{\infty}).
     	       \end{equation}
     	               Using the fact that $d(\cdot,\cdot)$ is lower semi-continuous for the weak topology and \eqref{eq: 2to3 2} we deduce from  \eqref{eq: 2to3 11} that
     	                \begin{equation}\label{eq: 2to3 5}
     	                   \lim_{n \to +\infty} d^2(\pi_{n},\rho)=d^2(\pi_{\infty},\rho),\quad \lim_{n \to +\infty} \Lambda_{\varepsilon,m,n}(\pi_n) = \Lambda_{\varepsilon,m}(\pi_\infty). 
     	                \end{equation}

     	                  Next, we observe that to prove \eqref{eq: 2to3 12} it suffices to show that	                 
     	                \begin{equation}\label{eq: 2to3 3}
     	               \limsup_{n\to+\infty} \int F^{i}_n(t) \lambda_{m+1,n}(\dd t)
	                    \leq \int F^{i}_{\infty}(t) \lambda_{m+1}(\dd t), \quad i=1,2,3, \end{equation}
     	                hold, where for $(t,n)\in \R \times {\mathbb{N}_{\infty}}$,  $i=1,2,3,$ we set $F^i_n:[0,+\infty)\to\R$ as 
     	                \begin{equation}\label{eq: 2to3 14} F^1_n(t):=e^{-m h^{\varepsilon}_{\pi_n}(t)} \psi_\varepsilon'\left(\frac{1}{2}d^2\left(\pi_n,\mu(t)\right)\right) e^{\hat{\kappa}t} \left[\cE\left(\mu(t)\right) - \sup_{k\geq 1} \cE(\pi_k)\right],  \end{equation}
     	                \begin{equation*} 
     	                F^2_n(t):=e^{-m h^{\varepsilon}_{\pi_n}(t)} \psi_\varepsilon'\left(\frac{1}{2}d^2\left(\pi_n,\mu(t)\right)\right) e^{\hat{\kappa}t} \left[ \sup_{k\geq 1} \cE(\pi_k) - \cE(\pi_n)\right],
     	                \end{equation*}
     	                and 
     	                \begin{equation*}
     	                    F^3_n(t):=e^{-m h^{\varepsilon}_{\pi_n}(t)} \left(\frac{1}{m} \vee h^{\varepsilon}_{\pi_n}(t) \right). 
     	                \end{equation*}
     	              We first prove \eqref{eq: 2to3 3} for $i=2,3$, as both $F_n^2 \geq 0$ and $F_n^3 \geq 0$. Indeed, using that the maps $s \mapsto e^{-ms}$, $s \mapsto e^{-ms}(\frac{1}{m} \vee s)$ and $s \mapsto \psi_\varepsilon'(s)$ are decreasing (See Lemma \ref{lemma:approximate_square_root}), we can leverage the lower semi-continuity of $d$ and $\cE$ to obtain that
                     for any $t_\infty\in[0,+\infty)$, and any $t_n \rightarrow t_{\infty}$,  we have 
     	                \begin{equation*}
     	                    \limsup_{n\to+\infty} F^{i}_n(t_n) \leq F^{i}_{\infty}(t_{\infty}), \quad i=2,3.
     	                \end{equation*} 
 	                    Moreover, since $(\pi_n)_{n\geq 1}\subseteq K^\rho_{\bar c,\bar d}$ and   \eqref{eq: lower bound on energy} holds, we deduce that  $\sup_{k\geq1}|\cE(\pi_k)|<+\infty$ and eventually that
 	                    \begin{equation*}
                      \sup_{n,t}F^{i}_n(t)<+\infty, \quad i=2,3. 
                      \end{equation*}
 	                    Thus, since $\lambda_{m+1,n}\longrightarrow \lambda_{m+1}$ weakly, we can apply Lemma \ref{lem:weak conv usc}, which gives the desired conclusion. The case $i=1$ is more delicate. 
 	                    Indeed, note that we cannot invoke Lemma \ref{lem:weak conv usc} here since the function $t\mapsto \cE\left(\mu(t)\right)-\sup_k\cE(\pi_k)$ may take negative values and obliges us to proceed otherwise. The proof of this step is carried out separately in Lemma \ref{lem: 2to3 skorokhod} below.
                            \end{itemize} 
\end{proof}

\begin{lemma}\label{lem: 2to3 skorokhod}
For fixed $a,b,c,\rho,\mu$, let $f^{3,\dagger}_{\varepsilon,m}$ be as in \eqref{def: f3}. Moreover, for any $n\geq 1$, let $f^{2,\dagger}_{\varepsilon,m,n}$ be as in \eqref{eq:f2dag}. For any $\bar c,\bar d\in \R$, consider a weakly converging sequence $\pi_n\rightarrow \pi_{\infty}$ such that $(\pi_n)_{n\geq1}\subseteq K^\rho_{\bar c,\bar d}$  and that 
\eqref{eq: 2to3 11} holds. Then
\begin{equation*}
\limsup_{n\to+\infty} \int F^{1}_n(t) \lambda_{m+1,n}(\dd t)
	                    \leq \int F^{1}_{\infty}(t) \lambda_{m+1}(\dd t), 
                     \end{equation*}
	                    where $F^{1}_n(t)$ has been defined in \eqref{eq: 2to3 14}.
\end{lemma}

\begin{proof}

 	                       Thanks to  Skorokhod's Theorem we can find random variables $(\zeta_n)_{n\in\bar{\mathbb{N}}}$ defined on the same probability space of $\lambda_{m+1,n}$  and such that
 	                      \begin{equation}\label{eq: 2to3 16} 
                        \zeta_{n}\sim \lambda_{m+1,n}\, \forall n\in\mathbb{N},\quad \zeta_{\infty}\sim\lambda_{m+1},\quad \zeta_n\rightarrow \zeta_{\infty} \, a.s. 
                        \end{equation}
                            We then proceed by contradiction and assume that there exist a subsequence $(n_k)_{k\geq 1}$ and $r>0$ such that 
                            \begin{equation}\label{eq: 2to3 6} \forall k\geq 1 \  \mathbb{E}[F^{1}_{n_k}(\zeta_{n_k})] \geq  \mathbb{E}[F^{1}_{\infty}(\zeta_{\infty})]+r.
                            \end{equation}
                            We derive a contradiction in two steps.
                            \begin{itemize}
                                \item\underline{Step 1: Almost sure convergence of $F^1_{n_k}(\zeta_{n_k})$ to $F^{1}_{\infty}(\zeta_{\infty})$ along a subsequence.} Due to the continuity of $t\mapsto \cE(\mu(t))$, see Lemma \ref{lem: EVI implies boundedness}, and of $r\mapsto \psi'_{\varepsilon}(\frac{r^2}{2})$, it is sufficient to show that 
                                \begin{equation}\label{eq: 2to3 7} 
                                \exp(-m h^{\varepsilon}_{m,\pi_{n_k}}(\zeta_{n_k})) \longrightarrow \exp(-m h^{\varepsilon}_{m,\pi_{\infty}}(\zeta_{\infty})) \quad \text{a.s.}  
                                \end{equation}
                                along a subsequence. To do so, it suffices to show that convergence in measure takes place.

We will do so on the basis of Lemma \ref{lemma:a.s.limsup_and_meanconvergene_inplies_prob_liminf} using the random variables $\zeta_{n_k}$, $\zeta_\infty$ and
\begin{equation*}
    f_k := \exp\left(-m h^{\varepsilon}_{m,\pi_{n_k}}(\zeta_{n_k})\right), \qquad f_\infty := \exp\left(-m h^{\varepsilon}_{m,\pi_{\infty}}(\zeta_\infty)\right).
\end{equation*}
We check the conditions of the Lemma. Item \ref{item:a.s.limsup_limsup} follows by \eqref{eq: 2to3 16} using the weak convergence of $\pi_n$ towards $\pi_{\infty}$ and the lower semi-continuity of  $d_{\varepsilon}$. Item \ref{item:a.s.limsup_mean} is implied by the second statement of \eqref{eq: 2to3 5}. Thus Step 1 is concluded by noting that by the Borel-Cantelli's Lemma any sequence converging in probability allows for a subsequence converging almost surely.

                                \item\underline{Step 2: Conclusion.}
                                Consider the  subsequence $(n_{k_l})$ along which we have 
                                \begin{equation}\label{eq: 2to3 19} 
                                F^1_{n_{k_l}}(\zeta_{n_{k_l}}) \longrightarrow F^1_{\infty}(\zeta_{\infty}) \quad \text{a.s.}
                                \end{equation}
                                 given  in  \underline{Step 1}.
                                Since $\cE(\mu(t))\leq \cE(\mu_0)$ and 
                                \begin{equation*}
                                \sup_{t\geq0,l\in \mathbb{N}} \Big|\exp\big(-m h^{\varepsilon}_{m,\pi_{n_{k_l}}}(\zeta_{n_{k_l}})\big)\, \psi_{\varepsilon}'\Big(\frac{1}{2}d^2_{\varepsilon}\big(\pi_{n_{k_l}},\mu(\zeta_{n_{k_l}}\big)\Big) \Big|<+\infty
                                \end{equation*}
                                we find
                                \begin{equation*}
                                \sup_{t\geq0,l \in\mathbb{N}} F^{1}_{n_{k_l}}(\zeta_{n_{k_l}})<+\infty.
                                \end{equation*}
                                But then, using Fatou's Lemma we find
                                \begin{equation*}
                                    \begin{split}
                                        \limsup_{l\to +\infty} \mathbb{E}[F^{1}_{n_{k_l}}(\zeta_{n_{k_l}})] &\leq  \mathbb{E}[\limsup_{l\to + \infty} F^{1}_{n_{k_l}}(\zeta_{n_{k_l}})] \\
                                       & \stackrel{\eqref{eq: 2to3 19}}{=}\mathbb{E}[ F^{1}_{\infty}(\zeta_{\infty})],
                                    \end{split}
                                \end{equation*}
                                which contradicts \eqref{eq: 2to3 6}.
                            \end{itemize}
\end{proof}

	\subsubsection{From $H_{3,\dagger}$ to $H_{4,\dagger}$}
In the definition of $H_{4,\dagger}$, we appeal again to the  approximation of the Tataru distance $d_T(\pi,\mu)$, defined as in \eqref{eqn:def_smoothed_Tataru_distance}. We recall that  

\begin{equation*} 
    d_{T,\varepsilon}(\pi,\mu) = \inf_{t \geq 0} \left\{ t + e^{\hat{\kappa}t} \psi_\varepsilon\left(\frac{1}{2}d^2(\pi,\mu(t))\right) \right\} =\inf_{t \geq 0} \left\{ t + h_{\pi}^\varepsilon(t) \right\} ,
\end{equation*}
where $\psi_\varepsilon$ and $h_\pi^\varepsilon$ have been defined in \eqref{eqn:def_approx_of_squareroot} and \eqref{eqn:definition_h} respectively. When sending $m \rightarrow \infty$ for the functions in the domain of $H_{3,\dagger}$ and $H_{3,\ddagger}$ we obtain by large deviation arguments the following operator.

\begin{definition} \label{definition:H4}
	Fix $\rho,\mu \in E$ such that $I(\rho) + I(\mu)  < \infty$, $a,b,\varepsilon > 0$ and $c \in \R$. Set
	\begin{equation}\label{def: fH4}
	f^{4,\dagger}_{\varepsilon}(\pi) := \frac{1}{2}a d^2(\pi,\rho) + b d_{T,\varepsilon}(\pi,\mu)  + c.
	\end{equation}
	For $\pi\in {E}$, consider the set $\Xi(\pi) \subseteq \bR^+$ given by
    \begin{equation*}
        \Xi(\pi):= \argmin_{t \geq 0} \left\{ t +e^{\hat{\kappa}t} d_{\varepsilon}(\pi,\mu(t)) \right\} = \argmin_{t \geq 0} \left\{t + h_{\pi}^\varepsilon(t)\right\}
    \end{equation*}
        and define	
		\begin{align}\label{def: gH4}
	g^{4,\dagger}_{\varepsilon}(\pi) & := a \left[\cE(\rho) - \cE(\pi)\right] - a\frac{\kappa}{2}d^2(\pi,\rho) + \frac{1}{2} a^2 d^2(\pi,\rho) + a b d(\pi,\rho) + \frac{1}{2} b^2 \\
	& \nonumber \qquad + b \sup_{t \in \Xi(\mu)} \left\{ e^{\hat{\kappa}t} \left[ \cE(\mu(t)) - \cE(\pi)\right]\psi_\varepsilon'\left(\frac{1}{2}d^2(\pi,\mu(t))\right)  - \frac{\hat{\kappa}}{2}e^{\hat{\kappa}t}  d_\varepsilon\left(\pi,\mu(t)\right)  \right\}. 
	\end{align}

     $H_{4,\dagger}$ is the operator given by all pairs: 
	\begin{align}\label{def: H4}
	    H_{4,\dagger} & := \left\{(f^{4,\dagger}_{\varepsilon},g^{4,\dagger}_{\varepsilon}) \, \middle| \, \rho,\mu: \, I(\rho) + I(\mu) < \infty,a,b>0,c\in \bR \right\}.
	\end{align}
	
		Fix $\gamma,\pi\in E$ such that $ I(\gamma)+ I(\pi) < \infty$, $a,b,\varepsilon > 0$ and $c \in \R$. Set
	\begin{equation*}
	f^{4,\ddagger}_{\varepsilon}(\mu) := -\frac{1}{2}a d^2(\gamma,\mu) - b d_{T,\varepsilon}(\mu,\pi)  + c.
	\end{equation*}
	For $\mu\in {E}$, consider the set $\Xi(\mu) \subseteq \bR^+$ given by
    \begin{equation*}
        \Xi(\mu):= \argmin_{t \geq 0} \left\{ t +e^{\hat{\kappa}t} d_{\varepsilon}(\mu,\pi(t)) \right\} = \argmin_{t \geq 0} \left\{t + h_{\mu}^\varepsilon(t)\right\}
    \end{equation*}
        and define	
    \begin{align*}
	g^{4,\ddagger}_{\varepsilon}(\mu) & := a \left[\cE(\mu) - \cE(\gamma)\right] + a\frac{\kappa}{2}d^2(\gamma,\mu) + \frac{1}{2} a^2 d^2(\gamma,\mu) - a b d(\gamma,\mu) { - \frac{1}{2}b^2 } \\
	& \nonumber \qquad - b \sup_{t \in \Xi(\mu)} \left\{ e^{\hat{\kappa}t} \left[ \cE(\pi(t)) - \cE(\mu)\right]\psi_\varepsilon'\left(\frac{1}{2}d^2(\mu,\pi(t))\right)  - \frac{\hat{\kappa}}{2}e^{\hat{\kappa}t}  d_\varepsilon\left(\mu,\pi(t)\right)  \right\}.
	\end{align*}

     $H_{4,\ddagger}$ is the operator given by all pairs: 
	\begin{align*}\label{def: H4d}
	    H_{4,\ddagger} & := \left\{(f^{4,\ddagger}_{\varepsilon},g^{4,\ddagger}_{\varepsilon}) \, \middle| \, \gamma,\pi: \,  I(\gamma) +I(\pi)< \infty,a,b>0,c\in \bR \right\}.
	   \end{align*}
\end{definition}

In the next theorem we prove the main result of this section.
\begin{theorem} \label{thm: 3to4}
{Let $\lambda > 0$ and let $h \in C_b(E)$ be continuous for the weak topology.}

	Every weakly upper semi-continuous viscosity subsolution to $f - \lambda H_{3,\dagger} f = h$ is also a viscosity subsolution to $f - \lambda H_{4,\dagger} f = h$.
	
	Every weakly lower semi-continuous viscosity supersolution to $f - \lambda H_{3,\ddagger} f = h$ is also a viscosity supersolution to $f - \lambda H_{4,\ddagger} f = h$. 
\end{theorem}

\begin{proof}
  As in the proof of Theorem \ref{thm: 2to3}, we only prove the first claim and argue on the basis of Proposition \ref{proposition:Hamiltonian_convergence_pseudo_coercive}. 
  
  Let $(f^{4,\dagger}_{\varepsilon},g^{4,\dagger}_{\varepsilon}) \in H_{4,\dagger}$. Thus, there exist $\rho,\mu,a,b,c,\varepsilon$ fulfilling the requirements in \eqref{def: H4} such that $(f^{4,\dagger}_{\varepsilon},g^{4,\dagger}_{\varepsilon})$ take the form \eqref{def: fH4} and \eqref{def: gH4}. We proceed by considering for the same $\rho,\mu,a,b,c,\varepsilon$ and for all $m\geq 1$, $f^{3,\dagger}_{\varepsilon,m},g^{3,\dagger}_{\varepsilon,m}$ as in \eqref{def: f3}, \eqref{def: g3} and show that the hypothesis of Proposition \ref{proposition:Hamiltonian_convergence_pseudo_coercive} are satisfied, which yields the desired result. We break down the proof in three steps.
  	\begin{itemize}
 	    \item \underline{Step 1: Verification of (a).} This follows as in the proof of Theorem \ref{thm:  2to3}.
 	    \item \underline{Step 2: Verification of (b) and (c).}
 	    To establish (b) we need to show that for all $\pi\in E$,
 	    \begin{equation}\label{eq: 3to4 5} 
 	    \lim_{m\to +\infty}\frac1m \log \Lambda_{\varepsilon,m}(\pi) = d_{T,\varepsilon}(\pi,\mu). 
 	    \end{equation}
 	     To this aim, we observe that $h_{\pi}^\varepsilon(t) = e^{\hat\kappa t}d_{\varepsilon}(\pi,\mu(t))$ is a bounded continuous function and that the sequence of measures $(\lambda_{m+1})_{m\geq 1}$ satisfies the large deviation principle on $\bR_+$ with rate function $\mathcal{I}(t)=t$. Thus the hypotheses of Varadhan's Lemma, see Proposition \ref{proposition:Varadhan}, are satisfied, and \eqref{eq: 3to4 5} holds. 
 	     
 	     To verify (c), for any $\bar c,\bar d\in \R$, consider a weakly converging sequence $\pi_m\rightarrow \pi_{\infty}$ such that $(\pi_m)_{m\geq1}\subseteq K^\rho_{\bar c,\bar d}$. We argue using Proposition \ref{proposition:Varadhan_usc} for the continuous and bounded functions $h_m := h_{\pi_m}^\varepsilon$ and the limiting function $h_\infty := h_{\pi_\infty}^\varepsilon$. Note that hypothesis \eqref{proposition:hypothesis_Varadhan_usc} follows from the weak lower semi-continuity of $d_{\varepsilon}(\cdot,\cdot)$. Since $(\lambda_{m+1})_{m\geq 1}$ satisfies the LDP with rate function $\mathcal{I}(t)=t$, Proposition \ref{proposition:Varadhan_usc} \ref{item:prop:usc varadhan_upper_bound} yields
 	     \begin{equation}\label{eq: 3to4 7} \limsup_{m\to+\infty} \frac1m \log \Lambda_{\varepsilon,m}(\pi_m) \leq- d_{T,\varepsilon}(\pi_{\infty},\mu). 
 	     \end{equation}
 	     Using once again the lower semi-continuity of the distance, we immediately deduce that 
 	     \begin{equation*}
 	         \liminf_{m\to +\infty}f^{3,\dagger}_{\varepsilon,m}(\pi_m) \geq f^{4,\dagger}_{\varepsilon}(\pi_{\infty}).
 	     \end{equation*}
 	     The proof of (c) is now complete.
   \item \underline{Step 3: Verification of (d)}
 	      For any $\bar c,\bar d\in \R$, consider a weakly converging sequence $\pi_m\rightarrow \pi_{\infty}$ such that $(\pi_m)_{m\geq1}\subseteq K^\rho_{\bar c,\bar d}$ such 
 	     \begin{equation}\label{eq: 3to4 6}
 	      \lim_{m\rightarrow+\infty}f^{3,\dagger}_{\varepsilon,m}(\pi_m)=f^{4,\dagger}_\varepsilon(\pi_{\infty}). 
 	     \end{equation}
        We will show that
        \begin{equation}\label{eq: 3to4 new1}
 	      \limsup_{m\rightarrow+\infty}g^{3,\dagger}_{\varepsilon,m}(\pi_m) \leq g^{4,\dagger}_\varepsilon(\pi_{\infty}). 
 	     \end{equation}
        As in Step 3 of the proof of Theorem \ref{thm: 2to3}, we deduce From \eqref{eq: 3to4 6}, \eqref{eq: 3to4 7}, and the lower semi-continuity of $d$, that
 	\begin{equation} \label{eqn:3to4_improved_limits}
 	    \lim_{m\rightarrow+\infty}d(\pi_m,\rho) = d(\pi_{\infty},\rho), \quad \lim_{m\to +\infty} -\frac{1}{m} \log \Lambda_{\varepsilon,m}(\pi_m) = d_{T,\varepsilon}(\pi_{\infty},\mu).
 	\end{equation}
       
 	     Given these assumptions, \eqref{eq: 3to4 new1} follows if we establish
 	       \begin{align}\label{eq: 3to4 2}
 	         \limsup_{m\rightarrow +\infty}\int_0^\infty\left( \psi'_{\varepsilon}\left(\frac{1}{2}d^2(\pi_m,\mu(t))\right)\exp(\hat\kappa t)[\cE(\mu(t))-\cE(\pi_m)]-\frac{\hat\kappa}{2}\Big(\frac1m \vee h_{\pi_m}^\varepsilon(t) \Big)\right)\theta_m(\dd t) \nonumber\\
 	        \leq  \sup_{t\in \Xi(\pi_{\infty})} \left\{ \psi'_{\varepsilon}\left(\frac{1}{2}d^2(\pi_{\infty},\mu(t))\right)\exp(\hat\kappa t)[\cE(\mu(t))-\cE(\pi_\infty)]-\frac{\hat\kappa}{2} h_{\pi_\infty}^\varepsilon(t) \right\}
 	      \end{align}
 	      where the sequence of probability measures $(\theta_m)_{m\geq 1}$ is defined by
 	    \begin{equation}\label{theta}
 	        \theta_m(\dd t) = \Lambda^{-1}_{\varepsilon,m}(\pi_m)\exp(-m h_{\pi_m}^\varepsilon(t)) \lambda_{m+1}(\dd t).
 	    \end{equation}
    Due to the second limit of \eqref{eqn:3to4_improved_limits}, we can apply Proposition \ref{proposition:Varadhan_usc} \ref{item:prop:usc varadhan_limit_points}, to obtain that sequence $(\theta_m)_{m \geq 1}$ is tight and any accumulation point is supported on $\Xi(\pi_{\infty})$. 
    
    \smallskip

    Thus, \eqref{eq: 3to4 2} holds if for any converging subsequence $(\theta_{m_k})_{k \geq 1}$ of $(\theta_m)_{m \geq 1}$ with limit $\theta_\infty$ there is a further subsequence $(\theta_{m_{k_l}})_{l \geq 1}$ such that
        \begin{equation} \label{eqn:Gi_limsup}
        \limsup_{l\rightarrow +\infty} \int_0^\infty G_{m_{k_l}}^i \dd \theta_{m_{k_l}} \leq \int_0^\infty G_{\infty}^i \dd \theta_{\infty}, \qquad i \in \{1,2,3\},
        \end{equation}
where the functions $G_m^i:[0,+\infty)\to\R$, $m \in \{1,2,\dots\} \cup \{\infty\} $, $i \in \{1,2,3\}$ are defined by
    \begin{align*}
        G_m^1(t) & := \psi_\varepsilon'\left(\frac{1}{2}d^2\left(\pi_m,\mu(t)\right)\right) e^{\hat{\kappa}t} \left[\cE\left(\mu(t)\right) - \sup_{n\geq 1} \cE(\pi_n)\right], \\
        G_m^2(t) & := \psi_\varepsilon'\left(\frac{1}{2}d^2\left(\pi_m,\mu(t)\right)\right) e^{\hat{\kappa}t} \left[ \sup_{n\geq 1} \cE(\pi_n) - \cE(\pi_m)\right], \\
        G_m^3(t) & :=  \left( \frac{1}{m} \vee h_{\pi_m}^\varepsilon(t) \right),
    \end{align*}
    and where the functions $G_\infty^i: [0,+\infty)\to\R$ are the corresponding terms obtainable from the second line of \eqref{eq: 3to4 2}.

\smallskip
    
   Following the proof of Theorem \ref{thm: 2to3}, we obtain \eqref{eqn:Gi_limsup} for $i = 2$ as the terms are non-negative and because we can exploit weak lower semi-continuity of $d$ and $\cE$.

    In this context, both the proof for $i = 1$ and for $i=3$ are more delicate. We prove these in Lemma \ref{lem: 3to4 skorokhod} below. 
    \end{itemize}

\end{proof}

\begin{lemma} \label{lem: 3to4 skorokhod} 
For fixed $a,b,c,\rho,\mu$, let $f^{4,\dagger}_{\varepsilon}$ be as in \eqref{def: fH4}. Moreover, for any $m\geq 1$, let $f^{3,\dagger}_{\varepsilon,m}$ be as in \eqref{def: f3}. For any $\bar c,\bar d\in \R$, consider a weakly converging sequence $\pi_m\rightarrow \pi_{\infty}$ such that $(\pi_m)_{m\geq1}\subseteq K^\rho_{\bar c,\bar d}$  and that 
\eqref{eq: 3to4 6} holds. Let $\{\theta_{m
}\}_{m \geq 1}$ be the sequence of probability measures defined as in \eqref{theta} and  $(\theta_{m_k})_{k \geq 1}$  any converging subsequence of  $(\theta_m)_{m \geq 1}$ with limit $\theta_\infty$.
    Then there exists a subsequence $\{\theta_{m_{k_l}}\}_{l \geq 1}$ satisfying
    \begin{equation*} 
        \limsup_{l\rightarrow +\infty} \int_0^\infty G_{m_{k_l}}^i \dd \theta_{m_{k_l}} \leq \int_0^\infty G_{\infty}^i \dd \theta_{\infty}
    \end{equation*}
    for $i \in \{1,3\}$.
\end{lemma}

The proof will be analogous to that of Lemma \ref{lem: 2to3 skorokhod}. We will therefore start with the Skorokhod's representation Theorem, and prove that we can get a.s. convergence also for the terms involving the metric $d(\pi_m,\mu(\cdot))$ and $h_{\pi_m}^\varepsilon(\cdot)$ appearing in $G^1_m$ and $G^3_m$.

\begin{proof}
    Let $\{\theta_{m_k}\}_{k\geq 0}$ be a sequence of measures with limit point $\theta_\infty$ supported on $\Xi(\pi_\infty)$. By the Skorokhod'  Theorem we can find a sequence of random variables $(\xi_{m_k})_{k \geq 1}$ defined on a common probability space satisfying
    \begin{equation}\label{eq: 3to4 4}
 	\forall k\geq 1 \quad \xi_{m_k} \sim \theta_{m_k}, \quad \xi_{\infty} \sim \theta_{\infty},\quad \xi_{m_k}\rightarrow \xi_{\infty} \, \text{a.s.}.
    \end{equation}
       
    Using Proposition \ref{proposition:Varadhan_usc} \ref{item:prop_usc varadhan_convergence_mean_weight} and Lemma \ref{lemma:a.s.limsup_and_meanconvergene_inplies_prob_liminf} for the sequence $f_k = - h_{m_k}(\xi_{m_k})$, we can use Borel-Cantelli's Lemma to extract a further subsequence $\{m_{k_l}\}_{l\geq 1}$ satisfying
    \begin{equation} \label{eqn:3to4_convergence_in_as_h_thm}
         h_{\pi_{m_{k_l}}}^\varepsilon(\xi_{m_{k_l}})  \rightarrow h_{\pi_\infty}^\varepsilon(\xi_\infty)\quad \text{a.e.}.
    \end{equation}

    At this point, we can first establish the result for $i=3$. To do so, we invoke the Lebesgue dominated convergence Theorem. Note that this theorem is applicable due to   \eqref{eqn:3to4_convergence_in_as_h_thm} and the fact that the functions $G_m^3$ are bounded uniformly by \eqref{eq: EVI implies boundedness 7} of Lemma \ref{lem: EVI implies boundedness}. 

    \smallskip

    We next proceed with the proof for $i=1$. Recall that
    \begin{equation*}
        h_{\pi_{m_k}}^\varepsilon(t) = e^{\hat{\kappa} t} \psi_\varepsilon\left( \frac{1}{2} d^2 (\pi_{m_k},\mu(t))\right).
    \end{equation*}
    First of all, the first statement of \eqref{eq: 3to4 4} implies that
    \begin{equation}\label{eqn:3to4_convergence_prob_expon}
        e^{-\hat{\kappa}\xi_{m_{k_l}}} \rightarrow e^{-\hat{\kappa}\xi_{\infty}} \qquad \text{a.s.}
    \end{equation}
    in probability. In combination with \ref{eqn:3to4_convergence_in_as_h_thm} this yields that
    \begin{equation*}
        \psi_\varepsilon\left( \frac{1}{2} d^2 (\pi_{m_{k_l}},\mu(t))\right) \rightarrow \psi_\varepsilon\left( \frac{1}{2} d^2 (\pi_{\infty},\mu(t))\right) \qquad \text{a.e.}.
    \end{equation*}
    As $r \mapsto \psi_\varepsilon\left(\frac{1}{2}r^2\right)$ is strictly increasing, it is invertible. Thus applying the inverse function on the above result, we obtain 
    \begin{equation}  \label{eqn:3to4_convergence_as_distance_thm} 
      d(\pi_{m_{k_l}},\mu(\xi_{m_{k_l}})) \rightarrow d(\pi_\infty,\mu(\xi_\infty))\quad \text{a.e.}, 
    \end{equation}
    The result for $i=1$ thus follows by noting that as $t \mapsto \cE(\mu(t))$ and $\psi_\varepsilon'$ are continuous, the  term with $i=1$ is continuous in $t$ and bounded from above. We can therefore conclude by \eqref{eqn:3to4_convergence_as_distance_thm} and Fatou's Lemma. 

\end{proof}

\subsubsection{From $H_{4,\dagger}$ to $H_{5,\dagger}$}

In this section, we do not carry out an approximation step, but rather focus ourselves on bounding the Hamiltonian using properties of the gradient flow. In particular, in the definition below, note e.g. from \eqref{def: fH4} and \eqref{eq: f5} that 
\begin{equation*}
    f^{5,\dagger}_{\varepsilon} = f^{4,\dagger}_{\varepsilon}, \qquad  f^{5,\ddagger}_{\varepsilon} = f^{4,\ddagger}_{\varepsilon}.
\end{equation*}
Comparing on the other hand \eqref{def: gH4}, \eqref{eq: g5}

 {we see that the only difference between $g^{4,\dagger}_{\varepsilon}$ and $g^{5,\dagger}_{\varepsilon}$ lies in the action of the gradient flow on our approximation of the Tataru distance. Correspondingly, we will prove below in Lemmas \ref{lemma:push_4_5} and \ref{lemma:bound_gradflow_for_4to5} that for any $\pi$, we have
	\begin{equation*}
	      \sup_{t \in \Xi(\mu)} \left\{ e^{\hat{\kappa}t} \left[ \cE(\mu(t)) - \cE(\pi)\right]\psi_\varepsilon'\left(\frac{1}{2}d^2(\pi,\mu(t))\right)  - \frac{\hat{\kappa}}{2}e^{\hat{\kappa}t}  d_\varepsilon\left(\pi,\mu(t)\right)  \right\} \leq 1
	\end{equation*}
 }{reflecting the idea that the Tataru distance is Lipschitz along the gradient flow.}

We proceed with the formal definitions of $H_{5,\dagger}$ and $H_{5,\ddagger}$.

\begin{definition} \label{definition:H5}
	Fix $\rho,\mu\in E$ such that $I(\rho) + I(\mu)  < \infty$, $a,b,\varepsilon > 0$ and $c\in\bR $. Set
	\begin{align}
	f^{5,\dagger}_{\varepsilon}(\pi) & := \frac{1}{2}a d^2(\pi,\rho) + b d_{T,\varepsilon}(\pi,\mu)  + c, \label{eq: f5} \\
g^{5,\dagger}_{\varepsilon}(\pi) & := a\left[\cE(\rho) - \cE(\pi)\right] - a \frac{\kappa}{2}d^2(\pi,\rho) + \frac{1}{2} a^2 d^2(\pi,\rho) \label{eq: g5} \\
	& \qquad + b  + a b d(\pi,\rho) + \frac{1}{2} b^2. \notag
	\end{align}
	$H_{5,\dagger}$ is the operator given by all pairs
	\begin{equation*}
	    H_{5,\dagger} := \left\{(f^{5,\dagger}_{\varepsilon},g^{5,\dagger}_{\varepsilon})\, \middle| \, \rho,\mu: \, I(\rho) + I(\mu) < \infty, a,b,\varepsilon>0, c \in \bR \right\}.
	\end{equation*}
	
		Fix $\gamma,\pi\in E$ such that $I(\gamma) + I(\pi)  < \infty$, $a,b,\varepsilon > 0$ and $c\in\bR $. Set
	\begin{align*}
	f^{5,\ddagger}_{\varepsilon}(\mu) & := - \frac{1}{2}a d^2(\gamma,\mu) - b d_{T,\varepsilon}(\mu,\pi)  + c,  \\
g^{5,\ddagger}_{\varepsilon}(\mu) & := a\left[\cE(\mu) - \cE(\gamma)\right] + a \frac{\kappa}{2}d^2(\gamma,\mu) + \frac{1}{2} a^2 d^2(\gamma,\mu)  \\
	& \qquad - b  - a b d(\gamma,\mu) {- \frac{1}{2}b^2}. 
	\end{align*}
	$H_{5,\ddagger}$ is the operator given by all pairs 
	\begin{equation*}
	    H_{5,\ddagger} := \left\{(f^{5,\ddagger}_{\varepsilon},g^{5,\ddagger}_{\varepsilon})\, \middle| \, \gamma,\pi: \, I(\gamma) + I(\pi) < \infty, a,b,\varepsilon>0, c \in \bR \right\}.
	\end{equation*}
\end{definition}

\begin{lemma} \label{lemma:push_4_5}
{Let $h \in C_b(E)$ and $\lambda >0$.}

	Every viscosity subsolution to $f - \lambda H_{4,\dagger} f = h$ is a viscosity subsolution to $f - \lambda H_{5,\dagger} f = h$.  
	
	Every viscosity supersolution to $f - \lambda H_{4,\ddagger} f = h$ is a viscosity subsolution to $f - \lambda H_{5,\ddagger} f = h$. 
\end{lemma}
The proof of the lemma above follows immediately from  the following lemma.

\begin{lemma} \label{lemma:bound_gradflow_for_4to5}
Fix $a,b>0,c\in \R$  and $\rho,\mu\in E$ such that $I(\rho) + I(\mu)  < \infty$ and let the corresponding test functions $f^{4,\dagger}_{\varepsilon}, f^{5,\dagger}_{\varepsilon}$ be given by \eqref{def: fH4} and \eqref{eq: f5}. (Note that $f^{4,\dagger}_{\varepsilon}=f^{5,\dagger}_{\varepsilon}$.) Then we have that
\[ g^{4,\dagger}_{\varepsilon}(\pi)\leq g^{5,\dagger}_{\varepsilon}(\pi) \quad \forall \pi\in E.\]
\end{lemma}

\begin{proof}
From the definition of $H_{4,\dagger},H_{5,\dagger}$ it is sufficient to show that for all $\pi\in E$ and for all $t\in\Xi(\mu)$
\begin{equation} \label{eqn:proof_4to5_1}
 e^{\hat{\kappa}t}[\cE(\mu(t))-\cE(\pi)] {\psi'_{\varepsilon}\left(\frac{1}{2}d^2(\pi,\mu(t))\right)}-\frac{\hat\kappa}{2}e^{\hat\kappa t}d_{\varepsilon}(\pi,\mu(t)) \leq 1.
\end{equation}

By construction, if $t\in\Xi(\mu)$, then $t$ minimizes $t+e^{\hat\kappa t}d_{\varepsilon}(\pi,\mu_t)$, whence $\frac{d^+}{d t} \Big({ t+ } e^{\hat \kappa t}d_{\varepsilon}(\pi,\mu(t))\Big) \geq 0$. Since $\psi_{\varepsilon}$ is continuously differentiable and increasing we can apply the chain rule for the upper right derivative. After doing so, we apply \eqref{item:ass_EVI} and $\psi_\varepsilon' \geq 0$, see Lemma \ref{lemma:approximate_square_root}, to obtain
\begin{equation}\label{eqn:proof_4to5_2}
\begin{aligned}
    0 & \leq 1 + \hat{\kappa}e^{\hat{\kappa}t}d_{\varepsilon}(\pi,\mu(t)) + e^{\hat{\kappa}t}\psi'_{\varepsilon}\left(\frac{1}{2}d^2(\pi,\mu(t))\right) \left(\frac{d^+}{d t} \frac12 d^2(\pi,\mu(t)) \right) \\
    & \leq 1 + \hat{\kappa}e^{\hat{\kappa}t}d_{\varepsilon}(\pi,\mu(t)) \\
    & \hspace{3cm} + e^{\hat{\kappa}t}\psi'_{\varepsilon}\left(\frac{1}{2}d^2(\pi,\mu(t))\right)\left(\left[\cE(\pi) - \cE(\mu(t)) \right] - \frac{\kappa}{2} d^2(\pi,\mu(t)) \right).
\end{aligned}
\end{equation}
We work on the final term on the right-hand side:
\begin{align*}
    - \frac{\kappa}{2}e^{\hat{\kappa}t}\psi'_{\varepsilon}\left(\frac{1}{2}d^2(\pi,\mu(t))\right) d^2(\pi,\mu(t)) & \leq - \frac{\hat{\kappa}}{2}e^{\hat{\kappa}t}\psi'_{\varepsilon}\left(\frac{1}{2}d^2(\pi,\mu(t))\right) d^2(\pi,\mu(t)) \\
    & \leq - \frac{\hat{\kappa}}{2}e^{\hat{\kappa}t} d(\pi,\mu(t)) \\
    & \leq - \frac{\hat{\kappa}}{2}e^{\hat{\kappa}t} d_\varepsilon(\pi,\mu(t)).
\end{align*}
In line one we used that $\psi_\varepsilon' \geq 0$ and $\kappa \geq \hat{\kappa}$, in line two we used Lemma \ref{lemma:approximate_square_root} \ref{item:lemma_approx_sqrt_bound_on_product} and in line three we used $d \leq d_\varepsilon$. Using this result in \eqref{eqn:proof_4to5_2}, we obtain
\begin{equation*}
    0 \leq 1 + \frac{\hat{\kappa}}{2}e^{\hat{\kappa}t}d_{\varepsilon}(\pi,\mu(t)) + e^{\hat{\kappa}t}\psi'_{\varepsilon}\left(\frac{1}{2}d^2(\pi,\mu(t))\right)\left[\cE(\pi) - \cE(\mu(t)) \right]
\end{equation*}
which is equivalent to \eqref{eqn:proof_4to5_1}.
\end{proof}

\subsubsection{From $H_{5,\dagger}$ to $H_{6,\dagger}$}

In this small section, we send $\varepsilon \downarrow 0$ in the approximation of $d_{T,\varepsilon}$ to $d_T$. Note that the only difference in definitions lies in the removal of the $\varepsilon$ in the test function, and that the bounds on the action of the Hamiltonian on the test function is unchanged.

\begin{definition} \label{definition:H6}
	Fix $\rho,\mu\in E$ such that $I(\rho) + I(\mu)  < \infty$, $a,b > 0$ and $c\in\bR $. Set
	\begin{align*}
	f^{6,\dagger}(\pi) & := \frac{1}{2}a d^2(\pi,\rho) + b d_{T}(\pi,\mu)  +c, \\
	g^{6,\dagger}(\pi) & := a\left[\cE(\rho) - \cE(\pi)\right] - a \frac{\kappa}{2}d^2(\pi,\rho) + \frac{1}{2} a^2 d^2(\pi,\rho) \\
	& \qquad + b  + a b d(\pi,\rho) + \frac{1}{2} b^2.
	\end{align*}
	 $H_{6,\dagger}$ is the operator given by all pairs
	\begin{align*}
	    H_{6,\dagger} & := \left\{(f^{6,\dagger},g^{6,\dagger}) \, \middle| \, \rho,\mu: \, I(\rho) + I(\mu) < \infty, a,b>0,c \in \bR \right\}. \\
	\end{align*}
	Fix $\gamma,\pi\in E$ such that $I(\gamma) + I(\pi)  < \infty$, $a,b>0$ and $c\in\bR $. Set 
	\begin{align*}
	f^{6,\ddagger}(\mu) & := -\frac{1}{2}a d^2(\gamma,\mu) -b d_{T}(\mu,\pi)  +c, \\
	g^{6,\ddagger}(\mu) & := a\left[\cE(\mu) - \cE(\gamma)\right] + a \frac{\kappa}{2}d^2(\gamma,\mu) + \frac{1}{2} a^2 d^2(\gamma,\mu) \\
	& \qquad - b  - a b d(\gamma,\mu) { - \frac{1}{2} b^2}.
	\end{align*}
	 $H_{6,\ddagger}$ is the operator given by all pairs
	\begin{align*}
	    H_{6,\ddagger} & := \left\{(f^{6,\ddagger},g^{6,\ddagger}) \, \middle| \, \gamma,\pi: \, I(\gamma) + I(\pi) < \infty, a,b>0,c \in \bR \right\}. \\
	\end{align*}
\end{definition}

\begin{lemma} \label{lemma:push_5_6}
Let $h \in C_b(E)$ and $\lambda >0$.

	Every viscosity subsolution to $f - \lambda H_{5,\dagger} f = h$ is a viscosity subsolution to $f - \lambda H_{6,\dagger} f = h$.
	
	Every viscosity supersolution to $f - \lambda H_{5,\ddagger} f = h$ is a viscosity subsolution to $f - \lambda H_{6,\ddagger} f = h$.
\end{lemma}

\begin{proof}
 Fix $\rho,\mu\in E$ such that $I(\rho) + I(\mu) < \infty$, $a,b > 0$, $c\in\bR $ and let $(f^{6,\dagger},g^{6,\dagger})$ be the corresponding pair in $H_{6,\dagger}$. For any $\varepsilon>0$, consider and $(f^{5,\dagger}_\varepsilon,g^{5,\dagger}_\varepsilon) \in H_{5,\dagger}$ as in \eqref{eq: f5},\eqref{eq: g5}, for the same $a,b,c$. By construction we have that $g^{5,\dagger}_{\varepsilon}=g^{6,\dagger}$ for all $\varepsilon>0$. 
Moreover, Lemma \ref{lemma:uniform_convergence_approxTataru_to_Tataru} \ref{item:lemma_approx_sqrt_uniform_Tataru}, makes sure that
\begin{equation*}
    \lim_{\varepsilon\rightarrow0} \sup_{\pi\in E}|f^{5,\dagger}_{\varepsilon}(\pi)-f^{6,\dagger}(\pi)|=0.
\end{equation*}
We have thus verified the hypothesis of Lemma \ref{lemma:viscosity_push}, whose application gives the conclusion.
\end{proof}

\subsubsection{From $H_{6,\dagger}$ to $\widetilde{H}_{\dagger}$}

In this final approximation step, we remove the restriction on the domain of $H_{6,\dagger}$ and $H_{6,\ddagger}$ that impose to the configurations to which we compare the distance must come from the domain of the Fisher energy $I$.

Indeed, comparing the definitions of  $H_{6,\dagger}$ and $H_{6,\ddagger}$ with those of $\widetilde{H}_\dagger$ and $\widetilde{H}_\ddagger$, we see that the only change lies in the fact that the domain of $I$ is replaced by that of $\cE$.

\begin{lemma} \label{lemma:push_6_7}
{Let $h \in C_b(E)$ and $\lambda >0$.}

	Every viscosity subsolution to $f - \lambda H_{6,\dagger} f = h$ is also a viscosity subsolution to $f - \lambda \widetilde{H}_\dagger f = h$.
	
	Every viscosity supersolution to $f - \lambda H_{6,\ddagger} f = h$ is also a viscosity supersolution to $f - \lambda \widetilde{H}_\ddagger f = h$.
\end{lemma}

\begin{proof}
    We establish the result only for subsolutions. Fix $h \in C_b(E)$ and $\lambda > 0$ and let $u^*$ be a viscosity subsolution to $f - \lambda H_{6,\dagger} f = h$. We prove this also holds for the equation in terms of  $\widetilde{H}_\dagger$.

    We argue on the basis of Lemma \ref{lemma:viscosity_push}. Thus let $(f_0,g_0) \in \widetilde{H}_\dagger$ be given by
    \begin{align*}
         f_0(\pi) & := \frac{1}{2}a d^2(\pi,\rho) + b d_T(\pi,\mu) + c \\
        g_0(\pi) & := a\left[ \cE(\rho) - \cE(\pi)\right] - a { \frac{\kappa}{2} } d^2(\pi,\rho) + b  + \frac{1}{2} a^2 d^2(\pi,\rho) + ab d(\pi,\rho) + \frac{1}{2} b^2.
    \end{align*}
    for $a,b > 0$, $c \in \bR$, $\mu, \rho \in E$ with $\cE(\rho) < \infty$. Consider the gradient flow $t \mapsto \rho(t)$ started from $\rho(0) = \rho$ and $t \mapsto \mu(t)$ started from $\mu(0) = \mu$. Set
    \begin{align*}
         f_t(\pi) & := \frac{1}{2}a d^2(\pi,\rho(t)) + b d_T(\pi,\mu(t)) + c \\
        g_t(\pi) & := a\left[ \cE(\rho(t)) - \cE(\pi)\right] - a {\frac{\kappa}{2} } d^2(\pi,\rho(t)) + b  \\
        & \qquad + \frac{1}{2} a^2 d^2(\pi,\rho(t)) + ab d(\pi,\rho(t)) + \frac{1}{2} b^2.
    \end{align*}
    By Lemma \ref{lem: EVI implies boundedness} \ref{item:Irightcontinuous} $I(\rho(t)) + I(\mu(t)) < \infty$ so that $(f_t,g_t) \in H_{6,\dagger}$. We then verify that assumptions of Lemma \ref{lemma:viscosity_push} hold for any sequence of times $t \downarrow 0$.

    First of all, by Lemma \ref{lem: EVI implies boundedness} \ref{item: EVI stability} $\rho(t) \rightarrow \rho$, $\mu(t) \rightarrow \mu$, so that uniform convergence of $f_t \wedge c$ to $f_0 \wedge c$ for any $c$ follows by the Lipschitzianity of $d$ and $d_T$ obtained in Lemma \ref{lemma:estimates_Tataru}.

    Uniform convergence $g_t \rightarrow g_0$ on appropriate sets follows similarly using that Lemma \ref{lem: EVI implies boundedness} \ref{item: EVI energy identity} yields $\cE(\rho(t)) \rightarrow \cE(\rho)$.
       
    Thus, the conclusion follows by an application of Lemma \ref{lemma:viscosity_push}.
    
\end{proof}

\subsection{Proof in the case that \texorpdfstring{$\kappa = 0$}{kappa = 0}.} \label{section:approximation_kappa_is_0}

In the previous subsection, we have established Theorem \ref{theorem:carry_over_solutions} in the context where $\kappa \neq 0$. The proof in the case $\kappa= 0$ needs changes due to the non-uniform estimates on the behaviour of the metric along  the gradient flow as established in  \cite{MuSa20}, restated in convenient form for our exposition in Lemma \ref{lem: EVI implies boundedness} \ref{item: EVI distance bound}. Instead of using the more elaborate control of the metric along the gradient flow throughout our proofs, we instead opt in this section to argue on the basis of an additional approximation $\kappa \uparrow 0$.

\begin{proof}[Proof of Theorem \ref{theorem:carry_over_solutions} in the setting that $\kappa = 0$]
    We argue for subsolutions only, as the supersolution case can be taken care of similarly. Let $h \in C_b(E)$ and $\lambda >0$.

    Denote by $\widetilde{H}_{\dagger,\kappa}$ the operator introduced in Definition \ref{definition:HdaggerHddagger}, where we now explicitly write the dependence on $\kappa$. Note that the dependence of a pair $(f,g) \in \widetilde{H}_{\dagger,\kappa}$ on $\kappa$ is present in $f$ via the Tataru distance as well as in $g$ via the gradient flow estimates.
    
    Denote by $\overline{H}_\dagger := \bigcup_{\kappa < 0} \widetilde{H}_{\dagger,\kappa}$. Let $u$ be a viscosity subsolution to $f - \lambda H_\dagger f = h$ and let $u^*$ be its weak upper semi-continuous regularization. 
    
    By assumption \eqref{item:ass_EVI} holds for $\kappa = 0$. Consequently, \eqref{item:ass_EVI} holds for any $\kappa < 0$. Thus, for any $\kappa < 0$, Theorem \ref{theorem:carry_over_solutions} implies that $u^*$ is a viscosity subsolution for $f - \lambda \widetilde{H}_{\dagger,\kappa} f = h$. Combining these results, we  obtain that $u^*$ is also a subsolution for $f - \lambda \overline{H}_\dagger f = h$.

    To conclude we proceed in two steps. 
    \begin{itemize}
    \item \textit{Step 1} We show $u^*$ is also a subsolution for $f - \lambda H_{6,\dagger} f = h$, where $H_{6,\dagger}$ is the operator from Definition \ref{definition:H6} for $\kappa = 0$.
    \item \textit{Step 2} We apply Lemma \ref{lemma:push_6_7} in the context $\kappa = 0$ to obtain the final result.
    \end{itemize}

    It thus suffices to carry out Step 1. We will argue on the basis of Proposition \ref{proposition:Hamiltonian_convergence_pseudo_coercive} and Lemma \ref{lemma:convergence_Tataru_in_kappa}. Indeed, it is the condition on the finiteness of $I$ in Lemma \ref{lemma:convergence_Tataru_in_kappa} that forces us to split our proof in two steps and work with $H_{6,\dagger}$ and perform again Lemma \ref{lemma:push_6_7} instead of working with $\widetilde{H}_\dagger$ directly.

    Let $(f_0,g_0) \in H_{6,\dagger}$ as in Definition \ref{definition:H6} for $\kappa = 0$, where $\rho,\mu$  such that $I(\rho) + I(\mu) < \infty$, $a,b > 0$ and $c \in \bR$. Now let $(f_\kappa,g_\kappa) \in \widetilde{H}_{\dagger,\kappa} \subseteq \overline{H}_\dagger$ be defined for the various $\kappa < 0$ as
    \begin{align*}
    f_\kappa(\pi) & := \frac{1}{2}a d^2(\pi,\rho) + b d_{T}(\pi,\mu)  +c, \\
	g_\kappa(\pi) & := a\left[\cE(\rho) - \cE(\pi)\right] - a \frac{\kappa}{2}d^2(\pi,\rho) + \frac{1}{2} a^2 d^2(\pi,\rho) \\
	& \qquad + b  + a b d(\pi,\rho) + \frac{1}{2} b^2.
    \end{align*}

    We now apply Proposition \ref{proposition:Hamiltonian_convergence_pseudo_coercive} for $\kappa \uparrow 0$. Assumption \ref{item:Ham_conv_weaktop_compactness} is immediate by the definition of our test functions. Assumptions \ref{item:Ham_conv_weaktop_equi_pointwise} and \ref{item:Ham_conv_weaktop_equi_uniform_on_compacts_f} follow from Lemma \ref{lemma:convergence_Tataru_in_kappa} \ref{item:convergence_Tataru_in_kappa_uniform} and \ref{item:convergence_Tataru_in_kappa_lsc} respectively.  Assumption \ref{item:Ham_conv_weaktop_equi_uniform_on_compacts_g} now follows immediately.

    Thus step 1 follows by application of Proposition \ref{proposition:Hamiltonian_convergence_pseudo_coercive} completing the proof of the lemma.

\end{proof}

\appendix

\section{Consequences of EVI and properties of the Tataru distances} \label{EVI-Tataru}

The formulation of gradient flows in terms of an evolutional variational inquality has far-reaching consequences. The following sections {include some key results from \cite{MuSa20} and \cite{CoKrTo21} as well as some new minor lemmas and} are included for completeness and readability.

\subsection{Consequences of EVI}  \label{section:consequences_EVI}

In this section we deduce from EVI various estimates on the behavior of $d$ and $\cE$ along the gradient flow. The first part of the following result is a copy of that in \cite{CoKrTo21}, but contains three new statements. All statements can be obtained from those in \cite{MuSa20}. 

\begin{lemma}\label{lem: EVI implies boundedness}
Let Assumption \ref{assumption:distance_and_energy} and \ref{assumption:gradientflow}  hold (in particular EVI inequality \eqref{item:ass_EVI}). 

For $\mu\in E$, let $(\mu(t))_{t\geq 0 }$ be the corresponding gradient flow starting at $\mu.$

Then the following holds:
\begin{enumerate}[(a)]
 \item\label{EVI: quadratic lower bound} There exists $\sigma\in E$ and constants $c_1,c_2\geq0$ such that
             \begin{equation}\label{eq: lower bound on energy} \quad \cE(\pi) \geq -c_1d^2(\pi,\sigma)-c_2\quad\forall\pi\in E.
             \end{equation} 
\item\label{item: EVI energy identity} For any $t>0$ we have
		\begin{equation}\label{eq_energy_identity}
		\cE(\mu(t)) - \cE(\mu) = -\int_0^t I(\mu(s)) \dd s.
		\end{equation}
		\item \label{item:domainIdense} The domain $\cD(I)$ is dense in $\cD(\cE)$ and dense in $E$. In particular, the domain $\cD(\cE)$ of $\cE$ is dense in $E$.
		\item \label{item:Irightcontinuous}
		For any $t>0$, we have $I(\mu(t)) < \infty$. The map $t \mapsto I(\mu(t))$ is right-continuous at any $t_0 \geq 0$ such that $I(\mu(t_0)) < \infty$.
        \item\label{item: EVI contraction}		Let $\nu\in E$  and let $(\nu(t))_{t\geq 0}$ be the corresponding gradient flow starting at $\nu$. Then we have 
	\begin{equation}\label{lemma:distance_contracting_under_gradient_flow}
	    d(\mu(t),\nu(t)) \leq e^{-{\kappa} t }d(\mu,\nu)\quad \forall t\in[0,+\infty).
		\end{equation}
		In particular, for a given $\mu\in E$, there is at most one solution of \eqref{item:ass_EVI} such that $\mu(t)\rightarrow \mu$ as $t\rightarrow0$.
		   \item\label{item: EVI stability} If $(\mu_n)_{n\in\mathbb{N}}\in E$  and $(t_n)_{n\in\mathbb{N}}\in [0,+\infty)$ are such that $\mu_n\rightarrow \mu$ and $t_{n}\rightarrow t$, then $\mu_n(t_n)\rightarrow \mu(t)$.
\item\label{item: EVI Bakry} Let $\mu \in E $ be such that $I(\mu)<+\infty$. Then for all $t>0$ we have
            \begin{equation}\label{eq: EVI implies boundedness 1}
                I(\mu(t)) \leq I(\mu)\exp(-2\kappa t)
            \end{equation}
            and
            \begin{equation}\label{eq: EVI implies boundedness 2}
            \begin{aligned}
                 \mathcal{E}(\mu(t)) & \geq \cE(\mu) + \frac{\exp(-2\kappa t)-1}{2\kappa}I(\mu), && \text{if } \kappa \neq 0, \\
                \mathcal{E}(\mu(t)) & \geq \cE(\mu) -t I(\mu), && \text{if } \kappa = 0, 
            \end{aligned}
            \end{equation}
            In particular, if $\kappa \neq 0$, the function $t\mapsto \exp(2\hat\kappa t)\mathcal{E}(\mu(t))$ is bounded from below.
\item\label{item: EVI distance bound} 
If  $I(\mu)<+\infty$, then there are constants $c_{1,\mu} > 0$, $c_{2,\mu} \in \bR$ such that for any $t>0$
\begin{equation}\label{eq: EVI implies boundedness 6}
\begin{aligned}
    & \sup_{0 \leq \varepsilon \leq 1} \exp(\hat\kappa t) d_{\varepsilon}(\pi,\mu(t)) \leq  d(\pi,\mu) +  c_{1,\mu} \sqrt{|\cE(\pi)|} + c_{2,\mu} && \text{if } \kappa \neq 0, \\
    & \sup_{0 \leq \varepsilon \leq 1} \frac{1}{t} d_{\varepsilon}(\pi,\mu(t)) \leq d(\pi,\mu) +  c_{1,\mu} \sqrt{|\cE(\pi)|} + c_{2,\mu} && \text{if } \kappa = 0,
\end{aligned}
\end{equation}
where, $d_\varepsilon$ is defined as in \eqref{eq: d_epsilon} and we write $d_0$ for the metric $d$. As a consequence, we have for any $\rho \in E$ and $c,d \in \bR$
\begin{equation}\label{eq: EVI implies boundedness 7}
\begin{aligned}
    & \sup_{0 \leq \varepsilon \leq 1} \sup_{\pi\in K_{c,d}^\rho,t\geq 0}\exp(\hat\kappa t)d_{\varepsilon}(\pi,\mu(t))<+\infty && \text{if } \kappa \neq 0, \\
    & \sup_{0 \leq \varepsilon \leq 1} \sup_{\pi\in K_{c,d}^\rho,t\geq 0} \frac{1}{t} d_{\varepsilon}(\pi,\mu(t))<+\infty && \text{if } \kappa = 0,
\end{aligned}
\end{equation}
where $K_{c,d}^\rho$ was defined in Assumption \ref{assumption:weak_topology}.
\end{enumerate}
\end{lemma}

\begin{remark}

Note that if \eqref{eq: lower bound on energy} holds for some $\sigma\in E$, then it also holds for any $\rho\in E$ with possibly larger constants.
\end{remark}

\begin{proof}
The proofs of \ref{EVI: quadratic lower bound} to \ref{item: EVI contraction} can be found in \cite{MuSa20} and can also be found in Section 4 of \cite{CoKrTo21}. We proceed with the final three items. First of all, item \ref{item: EVI stability} is \cite[Thm 3.5, Eq 3.20]{MuSa20}. 

\smallskip

We proceed with \ref{item: EVI Bakry}. Equation \eqref{eq: EVI implies boundedness 1} is \cite[Thm 3.5, Eq 3.12]{MuSa20} applied for $I = |\partial \cE|^2$. Moreover, we observe that \eqref{eq: EVI implies boundedness 2} is easily obtained applying \eqref{eq: EVI implies boundedness 1} in \ref{item: EVI energy identity}. 

\smallskip

The only statement left to prove is \ref{item: EVI  distance bound}. We only consider the case $\varepsilon = 0$, as the general case then follows by Lemma \ref{lemma:approximate_square_root} \ref{item:lemma_approx_sqrt_convergence}.

\smallskip

We first work out the case $\kappa\neq 0$, postponing the analysis of the case $\kappa=0$ to the end of the proof. We begin by observing that we can rewrite \eqref{item:ass_EVI} in the form (see also \cite[Eq. 3.9]{MuSa20})
\begin{equation*}
\frac{\dd^{+}}{\dd t} \Big(\frac{1}{2}\exp(\kappa t)d^2(\pi,\mu(t)) \Big) \leq \exp(\kappa t) \left[\mathcal{E}(\pi)-\mathcal{E}(\mu(t)) \right].   
\end{equation*}
If we now plug in the estimate \eqref{eq: EVI implies boundedness 2} we find
\begin{equation} \label{eqn:consequences_EVI_upperbound_forTataru}
    \frac{\dd^{+}}{\dd t} \Big(\frac{1}{2}\exp(\kappa t)d^2(\pi,\mu(t))\Big) \leq  \exp(\kappa t)[\mathcal{E}(\pi)-\mathcal{E}(\mu) ]   +\frac{I(\mu)}{2\kappa} (\exp(\kappa t)-\exp(-\kappa t )).
\end{equation} 
Note that as a consequence of the local Lipschitz property for $t \mapsto \mu(t)$ \cite[Thm 3.5, Eq 3.11]{MuSa20} the function 
\begin{equation*}
    t \mapsto \frac{1}{2}\exp(\kappa t)d^2(\pi,\mu(t))
\end{equation*}
is locally absolutely continuous. Thus, integrating both sides of \eqref{eqn:consequences_EVI_upperbound_forTataru} yields
\begin{equation}\label{eq: EVI implies boundedness 4}
    \begin{split}
\frac{1}{2}\exp(\kappa t)d^2(\pi,\mu(t)) \leq & \frac{1}{2}d^2(\pi,\mu)+ \frac{\exp(\kappa t)-1}{\kappa} [\cE(\pi)-\cE(\mu)]\\
&+\frac{I(\mu)}{2{\kappa}^2}[\exp(\kappa t)+\exp(-\kappa t)-2]. 
\end{split}
\end{equation}
We argue for $\kappa < 0$ first. Multiplication of \eqref{eq: EVI implies boundedness 4} by $\frac{1}{2}e^{\kappa t}$ yields
\begin{equation*}\label{eq: EVI implies boundedness 5}
\begin{split}
\exp(2\kappa t)d^2(\pi,\mu(t)) \leq & \exp(\kappa t) d^2(\pi,\mu)+ \frac{\exp(2\kappa t)-\exp(\kappa t)}{\kappa} [\cE(\pi)-\cE(\mu)]\\
& + \frac{I(\mu)}{\kappa^2}[\exp(2\kappa t)+1-\exp(\kappa t)],
\end{split}
\end{equation*}
after which taking a square root and using the upper bound $\sqrt{\sum x_i} \leq \sum \sqrt{|x_i|}$ leads to
\begin{equation*}
\begin{split}
\exp(\kappa t)d(\pi,\mu(t)) \leq & \exp(\frac{1}{2}\kappa t) d(\pi,\mu) +  \sqrt{\frac{\exp(2\kappa t)-\exp(\kappa t)}{\kappa}} \left(\sqrt{|\cE(\pi)|} + \sqrt{|0 \wedge \cE(\mu)|} \right) \\
& + \frac{\sqrt{1+ \exp(2\kappa t)-\exp(\kappa t)}}{\kappa} |\partial \cE|(\mu).
\end{split}
\end{equation*}
which gives \eqref{eq: EVI implies boundedness 6}.

The case $\kappa > 0$ follows similarly from \eqref{eq: EVI implies boundedness 4} after multiplication by $e^{-\kappa t}$. A similar procedure for $\kappa = 0$ yields
\begin{equation*}
    \frac{1}{2}d^2(\pi,\mu(t)) \leq \frac{1}{2}d^2(\pi,\mu) + t\left[\cE(\pi) - \cE(\mu)\right] + \frac{1}{2}t^2 I(\mu)
\end{equation*}
which gives \eqref{eq: EVI implies boundedness 6} for $\varepsilon = 0$ after multiplication by $t^{-2}$.

\end{proof}

\subsection{Properties of the Tataru distance}\label{appendixTataru}

In this Section, we focus on the properties of the Tataru distance of Definition \ref{definition:Tataru}. The first two properties are well known for the distance, and for our $\kappa$ dependent context stated as Lemma 4.3 and Lemma 4.4 in \cite{CoKrTo21}. We add a third property: namely that the Tataru distance, as a function of $\kappa$ is continuous.

\begin{lemma} \label{lemma:estimates_Tataru}
Let Assumptions \ref{assumption:distance_and_energy} and \ref{assumption:gradientflow} be satisfied.	We have for all $\mu,\hat{\mu}, \nu,\hat{\nu}\in E$ and $r > 0$ that
	\begin{enumerate}
		\item [(a)]
		\begin{equation*}
		d_T(\mu,\nu) - d_T(\hat{\mu},\hat{\nu}) \leq d(\mu,\hat{\mu}) + d(\nu,\hat{\nu})
		\end{equation*}
		\item [(b)]
		\begin{equation*}
		\frac{d_T(\nu(r),\hat{\nu}) - d_T(\nu,\hat{\nu})}{r} \leq 1.
		\end{equation*}
	\end{enumerate}
\end{lemma}

\begin{lemma} \label{lemma:triangle_inequality_Tataru}
	For $\rho, \mu,\nu\in E$, we have
	\begin{equation*}
	d_T(\rho,\nu) \leq d_T(\rho,\mu) + d_T(\mu,\nu).
	\end{equation*}
\end{lemma}

For the next lemma, we denote the dependence of the Tataru distance on $\kappa$ by $d_{T,\kappa}$.

\begin{lemma}\label{lemma:convergence_Tataru_in_kappa}

Let $\mu$ be such that $I(\mu) < \infty$ and $\rho \in E$ such that $\cE(\rho) < \infty$. Fix $\kappa_0 \in \bR$.

\begin{enumerate}[(a)]
    \item \label{item:convergence_Tataru_in_kappa_uniform} For any $c,d \in \bR$ we have
\begin{equation*}
    \lim_{\kappa \uparrow \kappa_0} \sup_{\pi \in K_{c,d}^\rho} \left|d_{T,\kappa}(\pi,\mu) - d_{T,\kappa_0}(\pi,\mu) \right| = 0.
\end{equation*}
\item \label{item:convergence_Tataru_in_kappa_lsc} Let $\pi$ be such that $\cE(\pi) < \infty$ and consider for $\kappa \leq \kappa_0$ a sequence $\pi_\kappa$ that converges weakly to $\pi$ as $\kappa \uparrow \kappa_0$. We then have that
\begin{equation*}
    \liminf_{\kappa \uparrow \kappa_0} d_{T,\kappa}(\pi_\kappa,\mu) \geq d_{T,\kappa_0}(\pi,\mu).
\end{equation*}
\end{enumerate}
\end{lemma}

\begin{remark}
Note that if $\kappa_1 \leq \kappa_2$, then $d_{T,\kappa_1} \leq d_{T,\kappa_2}$.
\end{remark}

\begin{proof}

We start with the proof of \ref{item:convergence_Tataru_in_kappa_uniform}.

Let $\kappa_1 < \kappa_2 \in \bR$, $\pi \in K_{c,d}^\rho$ and $\mu$ such that $I(\mu) < \infty$. We now estimate $|d_{T,\kappa_1}(\pi,\mu) - d_{T,\kappa_2}(\pi,\mu)| = d_{T,\kappa_2}(\pi,\mu) - d_{T,\kappa_1}(\pi,\mu)$. Let $t^*$ be an optimal time in the definition for $d_{T,\kappa_1}(\pi,\mu)$. Then we have that $t^* \leq d(\pi,\mu)$ and
\begin{align*}
    \left|d_{T,\kappa_1}(\pi,\mu) - d_{T,\kappa_2}(\pi,\mu)\right| & = d_{T,\kappa_2}(\pi,\mu) - d_{T,\kappa_1}(\pi,\mu) \\
    & \leq \left(e^{\hat{\kappa}_2 t^*} - e^{\hat{\kappa}_1 t^*}\right) d(\pi,\mu(t^*)) \\
    & \leq t^* \left(1 - e^{(\hat{\kappa}_1 - \hat{\kappa}_2) t^*}\right) \frac{1}{t^*} e^{\hat{\kappa}_2 t^*}d(\pi,\mu(t^*)) \\ 
    & \leq \left(1 - e^{(\hat{\kappa}_1 - \hat{\kappa}_2) t^*}\right) d(\pi,\mu)  \frac{1}{t^*} e^{\hat{\kappa}_2 t^*}d(\pi,\mu(t^*)) 
\end{align*}
The result thus follows by Lemma \ref{lem: EVI implies boundedness} \ref{item: EVI distance bound}.

For the proof of \ref{item:convergence_Tataru_in_kappa_lsc}, let $\pi$ be such that $\cE(\pi) < \infty$ and consider for $\kappa \leq \kappa_0$ a sequence $\pi_\kappa \rightarrow \pi$ weakly. We then have by Lemma \ref{lemma:estimates_Tataru} that
\begin{align*}
    d_{T,\kappa_0}(\pi,\mu) - d_{T,\kappa}(\pi_\kappa,\mu) & \leq \left[d_{T,\kappa_0}(\pi,\mu) - d_{T,\kappa}(\pi,\mu)\right] + \left[d_{T,\kappa}(\pi,\mu) - d_{T,\kappa}(\pi_\kappa,\mu)\right] \\
    & \leq \left[d_{T,\kappa_0}(\pi,\mu) - d_{T,\kappa}(\pi,\mu)\right] + d(\pi,\pi_\kappa).
\end{align*}
Rearranging and taking $\liminf_{\kappa \uparrow 0}$ we find
\begin{equation*}
    \liminf_{\kappa \uparrow \kappa_0} d_{T,\kappa}(\pi_\kappa,\mu) \geq d_{T,\kappa_0}(\pi,\mu) + \liminf_{\kappa \uparrow \kappa_0} 
 \left\{ \left[d_{T,\kappa_0}(\pi,\mu) - d_{T,\kappa}(\pi,\mu)\right] + d(\pi,\pi_\kappa) \right\}.
\end{equation*}
As $d$ is weakly lower semi-continuous by Assumption \ref{assumption:weak_topology}, it follows by the first part of the lemma that
\begin{equation*}
    \liminf_{\kappa \uparrow \kappa_0} d_{T,\kappa}(\pi_\kappa,\mu) \geq d_{T,\kappa_0}(\pi,\mu)
\end{equation*}
establishing the second claim.
\end{proof}

\subsection{Properties of the modified Tataru distances} \label{appendixTataru_modified}

In this section we first  examine the properties of the approximation of the square root. Secondly, we study the stability of the Tataru distance.

\smallskip

The following lemma is a slightly modified version of Lemma A.10 of \cite{Fe06}, which we therefore state without proof. 

\begin{lemma} \label{lemma:approximate_square_root}
	Let $\psi_{\varepsilon}$ be as in \eqref{eqn:def_approx_of_squareroot}.
	Then:
	\begin{enumerate}[(a)]
		\item \label{item:lemma_approx_sqrt_derivatives} $\psi_\varepsilon',\psi_\varepsilon'' \in C_b( \bR^+)$, $\psi_\varepsilon$ is positive and strictly increasing, $\psi_\varepsilon'$ is positive and strictly decreasing and $\sup_r r |\psi_\varepsilon''(r)| < \infty$.
		\item \label{item:lemma_approx_sqrt_convergence} $\sqrt{2r} \leq \psi_\varepsilon(r) \leq \sqrt{2\varepsilon} \vee \sqrt{2r}$, and $\lim_{\varepsilon \downarrow 0} \sup_{r \geq 0} \left|\psi_\varepsilon(r) - \sqrt{2r} \right| = 0$.
		\item \label{item:lemma_approx_sqrt_bound_on_product} $0 \leq r \psi_\varepsilon'\left(\frac{1}{2}r^2\right) \leq 1$.
	\end{enumerate}
\end{lemma}

\begin{lemma} \label{lemma:uniform_convergence_approxTataru_to_Tataru}
Let Assumptions \ref{assumption:distance_and_energy} and \ref{assumption:gradientflow} be satisfied. Let $d_{\varepsilon}$  be the modified distance defined in \eqref{eq: d_epsilon},  $d_{T,\varepsilon}$ the modified Tataru distance defined in \eqref{eqn:def_smoothed_Tataru_distance} and $d_T$ the standard Tataru distance.
\begin{enumerate}[(a)]
    \item \label{item:lemma_approx_sqrt_uniform_Tataru}  We have
    \begin{equation*}
        \lim_{\varepsilon \downarrow 0} \sup_{\pi,\rho} \left|d_{T,\varepsilon}(\pi,\rho) - d_T(\pi,\rho)\right| = 0
    \end{equation*}
    \item \label{item:lemma_approx_sqrt_lsc_Tataru} Let $\pi_m\rightarrow \pi_{\infty}$ weakly. If $\mu$ is such that $\sup_m d_\varepsilon (\pi_m,\mu)<+\infty$  and $I(\mu)<+\infty$, then we have
		\begin{equation}\label{eq: approx tataru lsc}
		    \liminf_m d_{T,\varepsilon}(\pi_m,\mu) \geq d_{T,\varepsilon}(\pi_{\infty},\mu)
		\end{equation}
		\end{enumerate}
\end{lemma}

\begin{proof}
    We start with the proof of \ref{item:lemma_approx_sqrt_uniform_Tataru}. Let $\psi_{\varepsilon}$ be as in \eqref{eqn:def_approx_of_squareroot}. By Lemma \ref{lemma:approximate_square_root}, using that $\hat{\kappa}\leq 0$, we have
    \begin{equation} \label{eqn:bound_on_Tataru_input}
    \begin{aligned}
        & \sup_{\pi,\rho} \sup_{t \geq 0} \left|t + e^{\hat{\kappa}t}\psi_{\varepsilon}\left(\frac{1}{2} d^2(\pi,\rho(t))\right) - (t + e^{\hat{\kappa}t}d(\pi,\rho(t))  \right| \\
        & \quad \leq \sup_{\pi,\rho} \sup_{t \geq 0} \left| \psi_{\varepsilon}\left(\frac{1}{2} d^2(\pi,\rho(t))\right) - d(\pi,\rho(t))  \right| \\
        & \leq c_\varepsilon,
        \end{aligned}
    \end{equation}
    where $c_\varepsilon \geq 0$ is a constant depending on $\varepsilon$ such that $\lim_{\varepsilon \rightarrow 0} c_\varepsilon = 0$. 
    
    \smallskip
    
    Let $t^*$ be an  optimal time for $d_T(\pi,\rho)$. We then have
    \begin{equation*}
        d_T(\pi,\rho) - d_{T,\varepsilon}(\pi,\rho) \leq e^{\hat{\kappa}t^*} \left(d(\pi,\rho(t^*)) - \psi_{\varepsilon}\left(\frac{1}{2}d^2(\pi,\rho(t^*))\right)\right)  \leq c_\varepsilon
    \end{equation*}
    For other inequality, let $t_\varepsilon^*$ be an optimal time for $d_{T,\varepsilon}(\pi,\rho)$. Then:
    \begin{equation*}
        d_{T,\varepsilon}(\pi,\rho) - d_{T}(\pi,\rho) \leq e^{\hat{\kappa}t_\varepsilon^*} \left(\psi_{\varepsilon}\left(\frac{1}{2}d^2(\pi,\rho(t^*_\varepsilon))\right) - d(\pi,\rho(t^*_\varepsilon))\right) \leq c_\varepsilon.
    \end{equation*}
    Both these inequalities and the fact that $c_\varepsilon \rightarrow 0$ establish the claim.

    \smallskip
    
     We proceed with the proof of \ref{item:lemma_approx_sqrt_lsc_Tataru}.
    
    Assume by contradiction that there exists a subsequence $(\pi_{k_m})_{m\geq 1}$ and $\delta >0$ such that  
    \begin{equation}\label{eq: contradiction}
    \forall m\geq 1, \quad d_{T,\varepsilon}(\pi_{k_m},\mu)<d_{T,\varepsilon}(\pi_{\infty},\mu)-\delta.    
    \end{equation} 
     Upon relabeling the subsequence we can assume w.l.o.g. that $k_m=m$. 
     
     Let $t_m\in \argmin \{t+ e^{\hat\kappa t}d_{\varepsilon}(\pi_m,\mu(t))\}$. Then, since $t_m\leq d_{\varepsilon}(\pi_m,\mu)$ and the latter is bounded by assumption, we have that $t_{l_m}\rightarrow t_{\infty}$ along a subsequence. As before we can assume w.l.o.g. that $l_m=m$ for all $m$. We have 
    \begin{align*}
    d_{T,\varepsilon}(\pi_\infty,\mu) & \leq t_{\infty}+e^{\hat\kappa t_\infty}d_{\varepsilon}(\pi_{\infty},\mu(t_{\infty})) \leq \liminf_m t_{m}+e^{\hat\kappa t_m}d_\varepsilon(\pi_{m},\mu(t_{m})) \\
    & = \liminf_m d_{T,\varepsilon}(\pi_m,\mu),\end{align*}
    where we used lower semicontinuity of $d_{\varepsilon}$ to obtain the second inequality. We have thus obtained a contradiction to \eqref{eq: contradiction}.
    
\end{proof}

	\section{Large Deviations and weak convergence} \label{Appendix:Varadhan_tilt}

To facilitate the application of Proposition \ref{proposition:Hamiltonian_convergence_pseudo_coercive} in the proofs of Theorem \ref{thm: 2to3} and \ref{thm: 3to4}, we list some key properties of weak convergence and large deviations in the context where we work with lower semi-continuous functions that are bounded from below.

We start out with two results on weak convergence of measures.

\begin{lemma}\label{lem:weak conv usc}
    Let $(\lambda_m)_{m\geq 1}$ be a sequence of probability measures on a Polish space $X$ converging weakly to $\lambda_{\infty}$ a probability measures on $X$. Moreover let $(f_m)_{m\geq 1}$ be a sequence of uniformly {upper }bounded measurable functions with the following property:  there exists a measurable function $f_{\infty}$ such that for any $t_\infty \in X$ and any sequence  $(t_m)_{m\geq 1}$ in $X$ converging to $t_{\infty}$, we have 
    \begin{equation*}
     \limsup_{m\rightarrow +\infty} f_{m}(t_m) \leq f_{\infty}(t_{\infty}).
    \end{equation*}
    Then 
    \begin{equation*}
        \limsup_{m\rightarrow+\infty} \int f_m \dd \lambda_m \leq \int f_{\infty} \dd \lambda_{\infty}.
    \end{equation*}
\end{lemma}

\begin{lemma} \label{lemma:a.s.limsup_and_meanconvergene_inplies_prob_liminf}
    Let $(\Omega,\cF,\PR)$ be a probability space with random variables $f_m$ and $f_\infty$ taking values in $[0,1]$ such that 
    \begin{enumerate}
        \item \label{item:a.s.limsup_limsup} $\limsup_{k\to +\infty}  f_k \leq f_\infty$ almost surely,
        \item \label{item:a.s.limsup_mean} $\lim_{k\to +\infty}  \bE[f_k] = \bE[f_\infty]$.
    \end{enumerate}
    Then $f_k$ converges to $f_\infty$ in probability.
\end{lemma}

\begin{proof}
    By Fatou's Lemma for the $\limsup$, we find using assumption \ref{item:a.s.limsup_limsup}that
    \begin{equation*}
        \limsup_{k\to +\infty}  \bE\left[(f_k - f_\infty) \bONE_{\{f_k \geq f_\infty\}}\right]  \leq 0.
    \end{equation*}
    Using the positivity of the integrand, we strengthen the statement to
    \begin{equation*}
        \lim_{k\to +\infty}  \bE\left[(f_k - f_\infty) \bONE_{\{f_k \geq f_\infty\}}\right]  = 0.
    \end{equation*}
    Using \ref{item:a.s.limsup_mean}, it follows that
    \begin{equation*}
        0 = \lim_{k\to +\infty} \bE\left[f_\infty - f_k\right] = \lim_{k\to +\infty}  \bE\left[(f_\infty - f_k) \bONE_{\{f_\infty > f_m\}}\right].
    \end{equation*}
    Thus, the claim follows by an application of Markov's inequality.
\end{proof}

We proceed with two results in the context of large deviations. For these results only the large deviation upper bound is needed. 

\begin{definition}
Let $(\lambda_m)_{m \geq 1}$ be a sequence of probability measures on a Polish space $X$. We say that $(\lambda_m)_{m \geq 1}$ satisfies the large deviation upper bound at speed $m$ with good rate function $\mathcal I : X \to\bR$ if the sets $\{t \, | \,\mathcal I(t) \leq a\}$ are compact for all $a$ for any closed set $A \subseteq X$ we have
        \begin{equation*} 
            \lim_{m \rightarrow \infty} \frac{1}{m} \log \int_A  \lambda_m(\dd t) = - \inf_{t\in A} \left\{\mathcal I(t) \right\}.
        \end{equation*}
\end{definition}

The first of our two results is the upper bound side of Varadhan's lemma.

\begin{proposition} \label{proposition:Varadhan}
 Let $(\lambda_m)_{m \geq 1}$ be a sequence of probability measures on a Polish space $X$ satisfying a large deviation principle at speed $m$ with good rate function $\mathcal I : X \to \bR$. Let $h: X \to \R$ be a continuous, bounded from below, function. Then, we have
        \begin{equation*}
            \lim_{m \rightarrow \infty} \frac{1}{m} \log \int e^{-m h(t)} \lambda_m(\dd t) = - \inf_{t\in X} \left\{\mathcal I(t)+ h(t) \right\}.
        \end{equation*}
\end{proposition}

The following result is a strengthening of the above statement which follows from an immediate adaptation of the typical proof of Varadhan's lemma. see e.g. Lemma 3.8 of \cite{RASe15}.

\begin{proposition}\label{proposition:Varadhan_usc}
    $(\lambda_m)_{m \geq 1}$ be a sequence of probability measures on a Polish space $X$ satisfying a large deviation upper bound at speed $m$ with good rate function $\cI : X \to \bR$. 
    Let $h_m$, $h_\infty : X \to \R$ be continuous functions satisfying 
    \begin{itemize}
        \item $\sup_{t \in X,m\geq1} h_m(t) > - \infty$, $\sup_{t\in X} h_\infty(t) > - \infty$.
        \item For any  $t_\infty \in X$ and any sequence  $(t_m)_{m\geq 1}$ in $X$ converging to $t_{\infty}$,  we have
        \begin{equation}\label{proposition:hypothesis_Varadhan_usc}
            \liminf_{m\to+\infty} h_m(t_m) \geq h_\infty(t_\infty).
        \end{equation}
    \end{itemize}
    Then
    \begin{enumerate}[(a)]
        \item \label{item:prop:usc varadhan_upper_bound} We have
    \begin{equation*}
        \limsup_{m\to+\infty} \frac{1}{m} \log \int e^{-m h_m(t)} \lambda_m(\dd t) \leq -\inf_{t\in X} \left\{\mathcal I(t)+ h_\infty(t) \right\}. 
    \end{equation*}
    \item \label{item:prop:usc varadhan_limit_points} Suppose that
    \begin{equation*}
        \lim_{m\to+\infty} \frac{1}{m} \log \int e^{-m h_m(t)} \lambda_m(\dd t) = -\inf_{t\in X} \left\{\mathcal I(t)+ h_\infty(t) \right\}. 
    \end{equation*}
    Then sequence of probability measures $(\nu_m)_{m\geq 1}$ defined by 
    \begin{equation*}
        \nu_m(\dd t) =  \frac{1}{\Lambda_m} e^{-m h_m(t)} \lambda_m(\dd t), \quad \Lambda_m = \int e^{-m h_m(t)} \lambda_m(\dd t)
    \end{equation*}
    is tight and any accumulation point is supported on $\argmin
\{\mathcal I+h_{\infty}\}$.
    \item \label{item:prop_usc varadhan_convergence_mean_weight} 
    Let the assumption in \ref{item:prop:usc varadhan_limit_points}  be satisfied and let  $\{\nu_{m_k}\}_{k \geq 1}$ be a subsequence of the sequence $\{\nu_m\}_{m \geq 1}$ converging to the limit $\nu_\infty$, then we have
    \begin{equation*}
        \lim_{k \rightarrow +\infty} \int h_{m_k} \dd \nu_{m_k} = \int h_\infty \dd \nu_\infty.
    \end{equation*}
    \end{enumerate}
    \end{proposition}

\begin{proof}
We start with the proof of \ref{item:prop:usc varadhan_upper_bound}. Let $\bN_\infty := \bN \cup\{\infty\}$ that is equipped with the regular topology on $\bN$ but with $\infty$ as the limit point of any unbounded sequence.

We embed our problem into the topological space $X \times \bN_\infty$. First of all, note that the function
\begin{equation*}
    h(x,m) := h_m(x), \qquad x \in X, m \in \bN_\infty
\end{equation*}
is bounded and lower semi-continuous on $X \times \bN_\infty$.

By assumption, the measures $\lambda_m$ satisfy a large deviation upper bound on $X$. Secondly, the measures 
\begin{equation*}
    \delta_{m}(\dd k)
\end{equation*}
satisfy a large deviation principle on $\bN_\infty$ with good rate function 
\begin{equation*}
    J(m) = \begin{cases}
    0 & \text{if } m = \infty, \\
    \infty & \text{if } m \neq \infty.
    \end{cases}
\end{equation*}
It follows that the measures $\lambda_m \times \delta_m$ satisfy a large deviation upper bound on $X \times \bN_\infty$ with good rate function 
\begin{equation*}
    \hat{\mathcal I}(t,m) = \begin{cases}
    \mathcal I(t) & \text{if } m = \infty, \\
    \infty & \text{if } m \neq \infty.
    \end{cases}
\end{equation*}
By the part of Varadhan's lemma that involves upper semi-continuous functions that are bounded above (Lemma 3.8 of \cite{RASe15}) it follows (arguing for $-h$) that
\begin{align*}
    & \limsup_{m\to +\infty} \frac{1}{m} \log \int e^{-m h_m(t)} \lambda_m(\dd t) \\
    & \qquad = \limsup_{m\to +\infty}  \frac{1}{m} \log \int e^{-m h(t,k)} \lambda_m(\dd t) \delta_m(\dd k) \\
    & \qquad \leq \sup_{t,m} \left\{ -h(t,m) - \hat{\mathcal { I}}(t,m) \right\} \\
    & \qquad = - \inf_t \left\{ h_\infty(t) + \mathcal I(t) \right\}
\end{align*}
establishing the claim.

\smallskip

For the proof of \ref{item:prop:usc varadhan_limit_points}, note that the a combination of the upper bound and the existence of the limit yield the large deviation upper bound for the sequence of measures $\nu_m$ with rate function
\begin{equation*}
    J(t) = \mathcal I(t) + h_\infty(t) - \inf \{\mathcal I + h_\infty \}.
\end{equation*}
Note that $J$ has compact sublevel sets due to the fact that $\mathcal I$ is good and $h_\infty$ bounded from below and continuous. It follows by e.g. Exercise 4.1.10 (c) in \cite{DZ98} that the sequence of measures $\nu_m$ is tight. Finally, any limit point of this sequence must be supported on the minimizers of the rate function $J$, a variant of this result was proven in \cite[Lemma C.1]{HeKrKu19}.

\smallskip

We proceed with the proof of \ref{item:prop_usc varadhan_convergence_mean_weight}. In this proof, we use the relative entropy functional $H : \cP(X) \times \cP(X) \rightarrow [0,\infty]$ defined by
\begin{equation*}
    H(\alpha \, | \, \beta) := \begin{cases}
        \int \frac{\dd \alpha}{\dd \beta} \log \frac{\dd \alpha}{\dd \beta} \dd \beta.
    \end{cases}
\end{equation*}

For any probability measure $\mu \in \cP(X)$ and measurable function $h$ {such that $e^{-h}$ is integrable} set
\begin{equation*}
    \nu := \frac{e^{-h}}{\Lambda_h} \mu,
\end{equation*}
with $\Lambda_h$ the appropriate normalization constant. A straightforward computation yields
\begin{equation} \label{eqn:relative_entropy_specific}
    H(\pi \, | \, \nu) = H(\pi \, | \, \mu) + \int h \dd \pi + \log \Lambda_h.
\end{equation} 
Choosing in our context $\mu = \lambda_m$, $h = m h_m$ and $\pi = \nu_m$, {we work with sufficiently integrable functions, and \eqref{eqn:relative_entropy_specific} reads}
\begin{equation*}
    0 =  H(\nu_m \, | \, \lambda_m) +  m \int h_m \dd \nu_m - \log \Lambda_m.
\end{equation*}
Rearranging and dividing by $m$ yields
\begin{equation*}
  - \frac{1}{m} \log \Lambda_m  = \frac{1}{m} H(\nu_m \, | \, \lambda_m) +  \int h_m \, \dd \nu_m.
\end{equation*}
Using the assumption in \ref{item:prop:usc varadhan_limit_points}, we we can extract a converging subsequence $\{\nu_{m_k}\}_{k \geq 1}$ with limit $\nu_\infty$ that has support on $\argmin \{\cI + h_\infty\}$, we find that
\begin{equation} \label{eqn:Varadhan_weight_convergence1}
    \begin{aligned} 
    \inf \{\cI + h_\infty\} & = \lim_{k \rightarrow +\infty} \frac{1}{m_k} \log \Lambda_{m_k} \\
    & = \liminf_{k \rightarrow +\infty}\frac{1}{{m_k}} H(\nu_{m_k} \, | \, \lambda_{m_k}) +  \int h_{m_k} \, \dd \nu_{m_k} \\
    & \geq \liminf_{k \rightarrow +\infty}\frac{1}{{m_k}} H(\nu_{m_k} \, | \, \lambda_{m_k}) +  \liminf_{k \rightarrow \infty} \int h_{m_k} \, \dd \nu_{m_k}.
\end{aligned}
\end{equation}
Using the lower semicontinuity of \eqref{proposition:hypothesis_Varadhan_usc} in combination with Lemma \ref{lem:weak conv usc} for the first statement and Theorem 3.5 (P1) to (H2) of \cite{Ma18} for the second, we obtain
\begin{equation}\label{eqn:Varadhan_weight_convergence2}
    \begin{aligned}
    & \liminf_{k \rightarrow+ \infty} \int h_{m_k} \, \dd \nu_{m_k} \geq \int h_\infty \, \dd \nu_\infty, \\
    & \liminf_{k \rightarrow +\infty} \frac{1}{{m_k}} H(\nu_{m_k} \, | \, \lambda_{m_k}) \geq \int \cI \, \dd \nu_\infty.
\end{aligned}
\end{equation}
Applying these two statements in \eqref{eqn:Varadhan_weight_convergence1}, we obtain
\begin{equation*}
    \inf \{\cI + h_\infty \}\geq \liminf_{k \rightarrow +\infty}\frac{1}{{m_k}} H(\nu_{m_k} \, | \, \lambda_{m_k}) +  \liminf_{k \rightarrow +\infty} \int h_{m_k} \, \dd \nu_{m_k} \geq \int \cI \, \dd \nu_\infty + \int h_\infty \, \dd \nu_\infty.
\end{equation*}
Using now that $\nu_\infty$ is supported on $\argmin \{\cI + h_\infty\}$, we find that all the inequalities in the above equation are equalities and that
\begin{equation*}
    \lim_{k \rightarrow +\infty}\frac{1}{{m_k}} H(\nu_{m_k} \, | \, \lambda_{m_k}) +  \int h_{m_k} \, \dd \nu_{m_k} = \int \cI + h_\infty \, \dd \nu_\infty.
\end{equation*}
In combination with \eqref{eqn:Varadhan_weight_convergence2}, this implies that both separate $\liminf_k$ statements must be limits:
\begin{equation*}
    \begin{aligned}
    & \lim_{k \rightarrow +\infty} \frac{1}{{m_k}} H(\nu_{m_k} \, | \, \lambda_{m_k}) \geq \int \cI \, \dd \nu_\infty, \\
    & \lim_{k \rightarrow +\infty} \int h_{m_k} \, \dd \nu_{m_k} \geq \int h_\infty \, \dd \nu_\infty,
\end{aligned}
\end{equation*}
establishing the claim.

\end{proof}

\section{Bounded smooth cylindrical test functions}\label{appendix:boundedcylinders}

In Section \ref{section:intro_technical_smoothcylinders} of our introduction we started out our computations with a slightly different set. In this appendix, we connect these test functions with our main result. Recall the definition of $\cT$ of \eqref{eqn:defT}. In our definition below, we consider bounded elements of $\cT$:

\begin{equation*}
        \cT_b  := \left\{\varphi \in \cT \, \middle| \, \varphi \text{ is bounded} \right\}. 
\end{equation*}

Consider the following set of Hamiltonians acting on bounded cylinders.

\begin{definition}[Smooth Hamiltonians] \label{definition:cylindrical_Hamiltonian}

\begin{enumerate}[(1)]
	\item For $\varphi \in \cT_b$  and $\bm\rho=(\rho,\rho_1,\ldots,\rho_{k}) \in {\cD(I) }$  we consider the functions
	\begin{subequations}
	    \begin{align}
	        f^{\dagger}(\pi) & :=\varphi\left( \frac{1}{2}d^2(\pi,\bm\rho)\right),  \label{eq:reg_f_dag}\\
	       g^{\dagger}(\pi) & := \sum_{i=0}^k \partial_i \varphi\left( \frac{1}{2}d^2(\pi,\bm\rho)\right) \left[\cE(\rho_i) - \cE(\pi) - \frac{\kappa}{2} d^2(\pi,\rho_i) \right] \label{eq:reg_g_dag}\\
            & \nonumber \quad + \frac{1}{2}\Big( \sum_{i=0}^k \partial_i \varphi\left(\frac{1}{2}d^2(\pi,\bm\rho)\right) d(\pi,\rho_i) \Big)^2 .
        \end{align}
	\end{subequations}
	and define $H_{0,\dagger}$  by
	\begin{equation*}
	    H_{0,\dagger} := \left\{ (f^{\dagger},g^{\dagger}) \, \middle| \,  \varphi \in \cT_{b}, \bm\rho \in {\cD(I) } \right\}.
	\end{equation*}
    \item For $\varphi \in \cT_b$ and $\bm\gamma=(\gamma,\gamma_1\ldots,\gamma_{k}) \in \cD(I)$ we consider
    \begin{subequations}
         \begin{align}
	        f^{\ddagger}(\mu) & :=-\varphi\left( \frac{1}{2}d^2(\bm\gamma, \mu)\right), \label{eq:reg_f_ddag} \\
	        g^{\ddagger}(\mu) & :=  \sum_{i=0}^k \partial_i \varphi\left(\frac{1}{2}d^2(\bm\gamma,\mu)\right)\left[\cE(\mu) -\cE(\gamma_i) + \frac{\kappa}{2}d^2(\gamma_i,\mu) \right] \label{eq:reg_g_ddag}\\
    	    &\nonumber \qquad + \frac{1}{2} \sum_{i=0}^k \partial_i \varphi\left(\frac{1}{2}d^2(\bm\gamma,\mu)\right)^2 d^2(\gamma_i,\mu) \\
    	    \nonumber & \qquad - \frac{1}{2} \sum_{i \neq j}^k \partial_i \varphi\left(\frac{1}{2}d^2(\bm\gamma,\mu)\right)\partial_j \varphi\left(\frac{1}{2}d^2(\bm\gamma,\mu)\right) d(\gamma_i,\mu)d(\gamma_j,\mu)
	    \end{align}
    \end{subequations}
	and define  $H_{0,\ddagger}$ by
	\begin{equation*}
	    H_{0,\ddagger} := \left\{ (f^{\ddagger},g^{\ddagger}) \, \middle| \,  \varphi \in \cT_{b}, \bm\gamma \in {\cD(I)} \right\}.
	\end{equation*}
\end{enumerate}

\end{definition}

Then we have the following result relating $(H_{0,\dagger},H_{0,\ddagger})$ to our main results.

\begin{lemma} \label{lemma:push_0_1}
    Let $h \in C_b(E)$ and $\lambda > 0$.

	Every viscosity subsolution to $f - \lambda H_{0,\dagger} f = h$ is also a viscosity subsolution to $f - \lambda H_{\dagger} f = h$.
	
	Every viscosity supersolution to $f - \lambda H_{0,\ddagger} f = h$ is also a viscosity supersolution to $f - \lambda H_{\ddagger} f = h$ .
\end{lemma}

\begin{proof}
    We only prove the first statement as the proof of the second statement is completely analogous once we perform the elementary lower bound \eqref{eqn:elementary_lowerbound} on the quadratic terms for $i \neq 0$. We argue on the basis of Lemma \ref{lemma:viscosity_push}.
    
    Let $\varphi_0 \in \cT$, $a > 0$, $\rho \in \cD(I)$ and $\bm\rho=(\mu_1,\ldots,\mu_{k})$ such that $\bm\rho\in {\cD(I) }$ and set
    \begin{equation*}
        f^{1,\dagger}(\pi) := \frac{a}{2}d^2(\pi,\rho) + \varphi_0\left( \frac{1}{2}d^2(\pi,\bm\mu)\right)
    \end{equation*}
    and set $g^{1,\dagger}$ as the corresponding function such that $(f^{1,\dagger},g^{1,\dagger}) \in H_{1,\dagger}$. We next construct approximating pairs in $H_{0,\dagger}$. Let $\iota_n$ be a a smooth increasing function such that $\iota_n(r) \leq r$ and
    \begin{equation*}
        \iota_n(r) = \begin{cases}
        r & \text{if } r \leq n, \\
        n+1 & \text{if } r \geq n+2.
        \end{cases}
    \end{equation*}
    Set $\varphi_n(r_0,r_1,\dots,r_k) := \iota_n(r_0 + \varphi_0(r_1,\dots,r_k))$ and set 
    \begin{equation*}
        f^{\dagger}_n(\pi) =\varphi_n\left(\frac{1}{2} d^2(\pi,\rho), \frac{1}{2}d^2(\pi,\bm\mu) \right).
    \end{equation*}
    As $\varphi_n$ is bounded, we have that $f^\dagger_n$ is of the form as in \eqref{eq:reg_f_dag}. Let corresponding $g^\dagger_n$ be the corresponding action as in \eqref{eq:reg_g_dag}.

    We next check the conditions of Lemma \ref{lemma:viscosity_push} (a). Observing that for any $c\in \bR$ and $n \geq c$ we have $f^{\dagger}_n \wedge c = f^{1,\dagger} \wedge c$, ensuring that the first condition of item (a) is satisfied. For the second condition, note that since $f^{1,\dagger} \geq f^{\dagger}_n$, we have $f^{1,\dagger}\vee f^{\dagger}_n= f^{1,\dagger}$ , hence we need to check that for any $d\in\R$
    \begin{equation*}
        \limsup_{n \rightarrow \infty} \sup_{\pi :  f^{1,\dagger}(\pi) \leq c} g^{\dagger}_n(\pi) \vee d - g^{1,\dagger}(\pi)\vee d \leq 0.
    \end{equation*}
    To do this, we observe that for $n \geq c$, the gradients of $(r_0,r_1,\dots,r_k) \mapsto \varphi_n(r_0,r_1,\dots,r_k)$ and  $(r_0,r_1,\dots,r_k) \mapsto r_0 + \varphi_0(r_1,\dots,r_k)$ coincide, at least, on the set $\{\bm r \in \mathbb{R}^{k+1}: r_0 + \varphi_0(r_1,\dots,r_k) \leq c \}$. In particular this yields the inclusion
    \begin{equation*}
    \forall n\geq c+1, \quad   \{ \pi:f^{1,\dagger}(\pi)\leq c\} \subseteq \{ \pi: g^{\dagger}_n(\pi)=g^{1,\dagger}(\pi) \}
   \end{equation*}
    from which it easily follows that the second condition of item (a) is verified. We can then apply Lemma \ref{lemma:viscosity_push}, to obtain the conclusion.
\end{proof}

\printbibliography

\end{document}